\documentclass[11pt,oneside,article]{memoir}

\settrims{0pt}{0pt} 
\settypeblocksize{*}{35.5pc}{*} 
\setlrmargins{*}{*}{1} 
\setulmarginsandblock{1.1in}{1.4in}{*} 
\setheadfoot{\onelineskip}{2\onelineskip} 
\setheaderspaces{*}{1.5\onelineskip}{*} 
\checkandfixthelayout

\usepackage[centertags,sumlimits,intlimits,namelimits,reqno]{amsmath}
\usepackage{latexsym}
\usepackage{amsfonts,amsthm,amssymb}
\usepackage{mathtools}
\usepackage[cal=euler,scr=rsfso]{mathalfa}
\usepackage[only,boxslash,fatsemi]{stmaryrd}
\usepackage{newpxtext}
\usepackage{exscale}
\usepackage{enumitem}
\usepackage{breakcites}
\usepackage[colorlinks,linkcolor=darkblue,citecolor=darkblue,urlcolor=darkblue,breaklinks=true]{hyperref}
\usepackage{tikz}
\usepackage[capitalize]{cleveref}
\usepackage[backend=biber, maxbibnames = 10, style = alphabetic]{biblatex}
\DeclareMathAlphabet{\mathpzc}{OT1}{pzc}{m}{it}
\usepackage[color=white]{todonotes}

\usepackage[all,cmtip,2cell]{xy}

\DeclareFontFamily{U}{mathx}{\hyphenchar\font45}
\DeclareFontShape{U}{mathx}{m}{n}{
      <5> <6> <7> <8> <9> <10>
      <10.95> <12> <14.4> <17.28> <20.74> <24.88>
      mathx10
      }{}
\DeclareSymbolFont{mathx}{U}{mathx}{m}{n}
\DeclareFontSubstitution{U}{mathx}{m}{n}
\DeclareMathAccent{\widecheck}{0}{mathx}{"71}

\usetikzlibrary{backgrounds,circuits,circuits.ee.IEC,shapes,fit,matrix,automata,decorations.markings}

  \pgfdeclarelayer{edgelayer}
  \pgfdeclarelayer{nodelayer}
  \pgfdeclarelayer{background}
  \pgfsetlayers{background,edgelayer,nodelayer,main}

  \tikzstyle{none}=[inner sep=0pt]

 \tikzset{
     oriented WD/.style={
        every to/.style={out=0,in=180,draw},
        label/.style={
           font=\everymath\expandafter{\the\everymath\scriptstyle},
           inner sep=0pt,
           node distance=2pt and -2pt},
        semithick,
        node distance=1 and 1,
        decoration={markings, mark=at position \stringdecpos with \stringdec},
        ar/.style={postaction={decorate}},
        execute at begin picture={\tikzset{
           x=\bbx, y=\bby,
           every fit/.style={inner xsep=\bbx, inner ysep=\bby}}}
        },
     string decoration/.store in=\stringdec,
     string decoration={\arrow{stealth};},
     string decoration pos/.store in=\stringdecpos,
     string decoration pos=.7,
     bbx/.store in=\bbx,
     bbx = 1.5cm,
     bby/.store in=\bby,
     bby = 1.5ex,
     bb port sep/.store in=\bbportsep,
     bb port sep=1.5,
     bb port length/.store in=\bbportlen,
     bb port length=4pt,
     bb penetrate/.store in=\bbpenetrate,
     bb penetrate=0,
     bb min width/.store in=\bbminwidth,
     bb min width=1cm,
     bb rounded corners/.store in=\bbcorners,
     bb rounded corners=2pt,
     bb small/.style={bb port sep=1, bb port length=2.5pt, bbx=.4cm, bb min width=.4cm, 
bby=.7ex},
		 bb medium/.style={bb port sep=1, bb port length=2.5pt, bbx=.4cm, bb min width=.4cm, 
bby=.9ex},
     bb/.code 2 args={
        \pgfmathsetlengthmacro{\bbheight}{\bbportsep * (max(#1,#2)+1) * \bby}
        \pgfkeysalso{draw,minimum height=\bbheight,minimum width=\bbminwidth,outer 
sep=0pt,
           rounded corners=\bbcorners,thick,
           prefix after command={\pgfextra{\let\fixname\tikzlastnode}},
           append after command={\pgfextra{\draw
              \ifnum #1=0{} \else foreach \i in {1,...,#1} {
                 ($(\fixname.north west)!{\i/(#1+1)}!(\fixname.south west)$) +(-
\bbportlen,0) 
  coordinate (\fixname_in\i) -- +(\bbpenetrate,0) coordinate (\fixname_in\i')}\fi 
              \ifnum #2=0{} \else foreach \i in {1,...,#2} {
                 ($(\fixname.north east)!{\i/(#2+1)}!(\fixname.south east)$) +(-
\bbpenetrate,0) 
  coordinate (\fixname_out\i') -- +(\bbportlen,0) coordinate (\fixname_out\i)}\fi;
           }}}
     },
     bb name/.style={append after command={\pgfextra{\node[anchor=north] at 
(\fixname.north) {#1};}}}
  }

  \tikzset{
  	unoriented WD/.style={
  		every to/.style={draw},
  		shorten <=-\penetration, shorten >=-\penetration,
  		label distance=-2pt,
  		thick,
  		node distance=\spacing,
  		execute at begin picture={\tikzset{
  			x=\spacing, y=\spacing}}
  		},
  	pack size/.store in=\psize,
  	pack size = 8pt,
	penetration/.store in=\penetration,
	penetration = 2pt,
  	spacing/.store in=\spacing,
  	spacing = 8pt,
  	link size/.store in=\lsize,
  	link size = 2pt,
  	pack color/.store in=\pcolor,
  	pack color = blue,
  	pack inside color/.store in=\picolor,
  	pack inside color=blue!20,
  	pack outside color/.store in=\pocolor,
  	pack outside color=blue!50!black,
  	surround sep/.store in=\ssep,
  	surround sep=8pt,
  	link/.style={
  		circle, 
  		draw=black, 
  		fill=black,
  		inner sep=0pt, 
  		minimum size=\lsize
  	},
  	pack/.style={
  		circle, 
  		draw = \pocolor, 
  		fill = \picolor,
  		inner sep = .25*\psize,
  		minimum size = \psize
  	},
  	outer pack/.style={
  		ellipse, 
  		draw,
  		inner sep=\ssep,
  		color=\pocolor,
  	},
  	intermediate pack/.style={
  		ellipse,
  		dashed, 
  		draw,
  		inner sep=\ssep,
  		color=\pocolor,
  	},
  }
  
\tikzset{
	spider diagram/.style={
		every to/.style={out=0, in=180, draw, thick},
		thick
	},
	dot size/.store in=\dotsize,
	dot size = 5pt,
	dot fill/.store in=\dotfill,
	dot fill = black,
	leg length/.store in=\leglen,
	leg length = 15pt,
	baby/.style={dot size = 2pt, leg length = 6pt},
	young/.style={dot size = 3pt, leg length = 10pt},
	special spider/.code n args={4}{
		\pgfkeysalso{circle, draw, thick, inner sep=0, fill=\dotfill, minimum width=\dotsize,
  		prefix after command={\pgfextra{\let\fixname\tikzlastnode}},
  		append after command={\pgfextra{
  			\ifnum #1=0{} \else {\foreach \i in {1,...,#1} {
					\tikzmath{\anglei={-90*(#1+1-2*\i)/#1};}
  				\draw [thick]
						(\fixname) .. controls 
						($(\fixname.center)-(\anglei:#3/3)$) and ($(\fixname.center)-(\anglei:#3*2/3)$) .. 
						({$(\fixname.center)-(\anglei:#3*2/3)$}-|{$(\fixname.center)-(#3,0)$}) coordinate (\fixname_in\i);
  			}}\fi
  			\ifnum #2=0{} \else {\foreach \i in {1,...,#2} {
					\tikzmath{\anglei={90*(#2+1-2*\i)/#2};}
  				\draw [thick]
						(\fixname.center) .. controls 
						($(\fixname.center)+(\anglei:#4/3)$) and ($(\fixname.center)+(\anglei:#4*2/3)$) .. 
						({$(\fixname.center)+(\anglei:#4*2/3)$}-|{$(\fixname.center)+(#4,0)$}) coordinate (\fixname_out\i);
  			}}\fi
  		}}
		}
	},
	spider/.code 2 args={
		\pgfkeysalso{special spider={#1}{#2}{\leglen}{\leglen}}
	}
}

\tikzset{
	function/.style={->,>= stealth, shorten <=2pt, shorten >=2pt}
}

\tikzstyle{circ}=[circle,fill=black,draw,inner sep=3pt]
\tikzstyle{circ2}=[circle,fill=black,draw,inner sep=1pt]
\tikzstyle{dot}=[circle,fill=black,draw,inner sep=1.5pt]
\tikzstyle{dotbig}=[circle,fill=black,draw,inner sep=2.2pt]
\tikzstyle{sdot}=[circle,fill=white,draw,inner sep=1pt]
\tikzstyle{amp}=[rounded rectangle, rounded rectangle west arc=0pt, draw, inner
sep = 2pt,minimum width=4ex,fill=white]

\tikzstyle{sq}=[rectangle,draw,inner sep=1pt,minimum width=15pt,minimum height=15pt]
  
\tikzstyle{mapsto}=[color=gray,very thick,shorten>=10pt,shorten<=5pt,->,>=stealth]

\tikzstyle{connection}=[circuit symbol open,
    circuit symbol size=width .5 height .5,
    shape=circle ee,
  inner sep=0.25\pgflinewidth]
  \tikzset{->-/.style={decoration={
    markings,
    mark=at position .54 with {\arrow{latex}}},postaction={decorate}}}

\newcommand{\mult}[1]
{
\begin{aligned}
    \resizebox{#1}{!}{
	\begin{tikzpicture}[spider diagram]
		\node[spider={2}{1}] {};
	\end{tikzpicture}
      }
\end{aligned}
}

\newcommand{\unit}[1]
{
  \begin{aligned}
    \resizebox{#1}{!}{
	\begin{tikzpicture}[spider diagram]
		\node[spider={0}{1}] (a) {};
		\node[below =.2 of a] {};
		\node[above =.5 of a] {};
	\end{tikzpicture}
      }
  \end{aligned}
}

\newcommand{\comult}[1]
{
\begin{aligned}
    \resizebox{#1}{!}{
	\begin{tikzpicture}[spider diagram]
		\node[spider={1}{2}] {};
	\end{tikzpicture}
      }
\end{aligned}
}

\newcommand{\counit}[1]
{
  \begin{aligned}
    \resizebox{#1}{!}{
	\begin{tikzpicture}[spider diagram]
		\node[spider={1}{0}] (a) {};
		\node[below =.2 of a] {};
		\node[above =.5 of a] {};
	\end{tikzpicture}
    }
  \end{aligned}
}

\newcommand{\swap}[1]
{
  \begin{aligned}
    \resizebox{#1}{!}{
\begin{tikzpicture}
	\begin{pgfonlayer}{nodelayer}
		\node [style=none] (2) at (-0.5, -0.5) {};
		\node [style=none] (3) at (-2, 0.5) {};
		\node [style=none] (4) at (-0.5, 0.5) {};
		\node [style=none] (5) at (-2, -0.5) {};
	\end{pgfonlayer}
	\begin{pgfonlayer}{edgelayer}
		\draw[line width=2pt] [in=180, out=0, looseness=1.00] (3.center) to (2.center);
		\draw[line width=2pt] [in=0, out=180, looseness=1.00] (4.center) to (5.center);
	\end{pgfonlayer}
\end{tikzpicture}
    }
  \end{aligned}
}

  \usetikzlibrary{ 
  	cd,
  	math,
  	decorations.markings,
		decorations.pathreplacing,
  	positioning,
  	arrows.meta,
  	shapes,
		shadows,
		shadings,
  	calc,
  	fit,
  	quotes,
  	intersections,
    circuits,
    circuits.ee.IEC
  }

  \setlist{nosep}

\newtheorem*{theorem*}{Theorem}
\newtheorem{definition}{Definition}[chapter]
\newtheorem{proposition}[definition]{Proposition}   
\newtheorem{theorem}[definition]{Theorem}
\newtheorem{lemma}[definition]{Lemma}   
\newtheorem{corollary}[definition]{Corollary}   

\theoremstyle{remark}
\newtheorem{example}[definition]{Example}
\newtheorem*{example*}{Example}
\newtheorem{remark}[definition]{Remark}

\crefalias{chapter}{section}

  \author{Brendan Fong and David I.\ Spivak}

  \title{Hypergraph Categories}
  \date{\vspace{-.3in}}
  \definecolor{darkblue}{rgb}{0,0,0.7} 
  \newcommand{\define}[1]{\emph{\upshape #1}}
  
  \newcommand{\ot}{\otimes}

  \newcommand{\rr}{{\mathbb{R}}}
\newcommand{\hyp}{\Cat{Hyp}}
\newcommand{\hcca}[1][-]{A_{#1}}
\newcommand{\cahc}[1][-]{\cat{H}_{#1}}
\newcommand{\alg}{\text{-}\Cat{Alg}}
\newcommand{\calg}{\Cat{Cospan}\alg}

\DeclareMathOperator{\id}{id}

\DeclareMathOperator{\ob}{Ob}

\DeclarePairedDelimiter{\copair}{[}{]}
\DeclarePairedDelimiter{\corners}{\ulcorner}{\urcorner}

\newcommand{\const}[1]{\mathtt{#1}}
\newcommand{\cat}[1]{\mathcal{#1}}
\newcommand{\Cat}[1]{\mathbf{#1}}
\newcommand{\CCat}[1]{{\mathpzc{#1}}}
\newcommand{\fun}[1]{\mathsf{#1}}
\newcommand{\Fun}[1]{\mathsf{#1}}

\newcommand{\finset}{\Cat{FinSet}}
\newcommand{\smset}{\Cat{Set}}
\newcommand{\smcat}{\Cat{Cat}}
\newcommand{\hhyp}{\CCat{Hyp}}
\newcommand{\of}{\const{OF}}
\newcommand{\io}{\const{io}}
\newcommand{\ff}{\const{ff}}
\newcommand{\gens}{\Fun{Gens}}
\newcommand{\Str}{\Fun{Str}}
\newcommand{\str}{\const{str}}

\newcommand{\List}{\Fun{List}}
\newcommand{\sing}{\fun{sing}}
\newcommand{\flatten}{\fun{flat}}
\newcommand{\gathr}[1]{\,\widehat{#1}\,}
\newcommand{\parse}[1]{\widecheck{#1}}

\newcommand{\frob}{\Fun{Frob}}
\newcommand{\prt}{\Fun{Part}}
\newcommand{\concat}{\oplus}
\newcommand{\emptylist}{\varnothing}

\newcommand{\To}[1]{\xrightarrow{#1}}
\newcommand{\Too}[1]{\To{\;\;#1\;\;}}
\newcommand{\from}{\leftarrow}

\newcommand{\surj}{\twoheadrightarrow}
\newcommand{\inj}{\rightarrowtail}
\newcommand{\ul}[1]{{#1}}
\renewcommand{\ss}{\subseteq}
\newcommand{\op}{^\mathrm{op}}

\newcommand{\comp}{\const{comp}}
\newcommand{\cupp}{\const{cup}}
\newcommand{\capp}{\const{cap}}

\newcommand{\cospan}[1][\Lambda]{\Cat{Cospan}_{#1}}
\newcommand{\linrel}{\Cat{LinRel}}
\newcommand{\rel}{\Cat{Rel}}

\newcommand{\cp}{\mathbin{\fatsemi}}

\newcommand{\nn}{\mathbb{N}}

\newcommand{\lax}{\Cat{Lax}}

\newcommand{\commentout}[1]{}

\begin{document}   

\maketitle
\begin{abstract}
Hypergraph categories have been rediscovered at least five times, under various
names, including well-supported compact closed categories, dgs-monoidal
categories, and dungeon categories. Perhaps the reason they keep being
reinvented is two-fold: there are many applications---including to automata,
databases, circuits, linear relations, graph rewriting, and belief
propagation---and yet the standard definition is so involved and ornate as to be
difficult to find in the literature. Indeed, a hypergraph category is, roughly
speaking, a ``symmetric monoidal category in which each object is equipped with
the structure of a special commutative Frobenius monoid, satisfying certain
coherence conditions''.

Fortunately, this description can be simplified a great deal: a hypergraph
category is simply a ``cospan-algebra,'' roughly a lax monoidal functor from cospans to sets. The goal of this paper is to remove the
scare-quotes and make the previous statement precise. We prove two main
theorems. First is a coherence theorem for hypergraph categories, which says
that every hypergraph category is equivalent to an \emph{objectwise-free}
hypergraph category. Second, we prove that the category of objectwise-free
hypergraph categories is equivalent to the category of cospan-algebras. 
\end{abstract}
  
\chapter{Introduction}

Suppose you wish to specify the following picture:
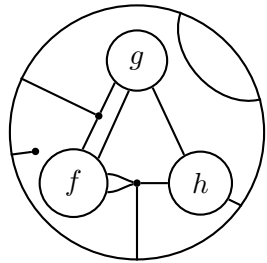
\begin{equation}\label{eqn.main_pic}
  \begin{aligned}
\begin{tikzpicture}[penetration=0, unoriented WD,  pack size=20pt, pack inside
color=white, pack outside color=black, link size=2pt, scale=2]
	\node[pack] at (-1.5,-1) (f) {$f$};
	\node[pack] at (0,1.9) (g) {$g$};
	\node[pack] at (1.5,-1) (h) {$h$};
	\node[outer pack, inner sep=34pt] at (0,.2) (out) {};
	\node[link] at ($(f)!.5!(h)$) (link1) {};
	\node[link] at (-2.4,-.25) (link2) {};
	\node[link] at ($(f.75)!.5!(g.-135)$) (link3) {};
   \begin{scope}
	\draw (out.270) -- (link1);
	\draw (out.190) -- (link2);
	\draw (out.155) -- (link3);
	\draw (out.-35) -- (h.-30);
	\draw (out.15) to[out=-165,in=-110] (out.70);
	\draw (f.15) to[out=0,in=165] (link1);
	\draw (f.-15) to[out=0,in=-165] (link1);
	\draw (h.180) -- (link1);
	\draw (g.-60) -- (h.120);
	\draw (f.45) -- (g.-105);
	\draw (f.75) -- (link3);
	\draw (g.-135) -- (link3);
   \end{scope}
\end{tikzpicture}
\end{aligned}
\end{equation}
This picture might represent, for example, an electrical circuit, a tensor network, or a
pattern of shared variables between logical formulas.

One way to specify the picture in \cref{eqn.main_pic}---given the symbols $f$, $g$, $h$ and their arities---is to define
primitives that represent merging, initializing, splitting, and terminating
wires. The picture can then be constructed piece by piece,
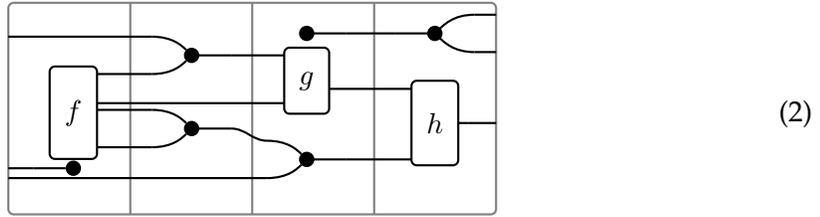
\begin{equation}
\label{eq.keyhyp1}
\begin{aligned}
\begin{tikzpicture}[penetration=0, oriented WD, spider diagram, bb min
width=.4cm, bby=.22cm, inner xsep=.2cm, bbx=.5cm]
	\coordinate (id1) {};
	\node[bb={0}{0}, minimum height=35pt, below=2 of id1] (f) {$f$};
	\node[spider={1}{0}, below=.1 of f] (eta1) {};
	\node[spider={2}{1}, below right=1 and 3 of id1] (mu21) {};
	\node[spider={2}{1}, below=3.5 of mu21] (mu22) {};
	\node[spider={0}{1}, right=6 of id1] (eps3) {};
	\node[bb={0}{0}, minimum height=25pt, below=.4 of eps3] (g) {$g$};
	\node[spider={2}{1}, below=2.3 of g] (mu3) {};
	\node[spider={1}{2}, right=3 of eps3] (del4) {};
	\node[bb={0}{1}, minimum height=32pt, below=2.4 of del4] (h) {$h$};
	\node at ($(f)-(.3,0)$) (left) {};
	\node at ($(h)+(.2,0)$) (right){};
	\node at ($(f)+(0,5)$) (top){};
	\node at ($(f)-(0,4.5)$) (bot){};
	\node[bb={0}{0}, color=gray, fit=(top) (left) (bot) (right)] (outer) {};
	\foreach \i in {1,2,3}{
		\node at ($(outer.west)!.25*\i!(outer.east)$) (col\i) {};
		\draw[thick, gray] (col\i|-outer.south) -- (col\i|-outer.north);}
	\draw (outer.west|-eta1_in1) -- (eta1_in1);
	\draw (outer.west|-mu21_in1) -- (mu21_in1);
	\draw (f.east|-mu21_in2) -- (mu21_in2);
	\draw (mu21_out1) -- (mu21_out1-|g.west);
	\draw (f.east|-g.225) -- (g.225);
	\draw (g.-20) -- (g.-20-|h.west);
	\draw (f.east|-mu22_in1) -- (mu22_in1);
	\draw (f.east|-mu22_in2) -- (mu22_in2);
	\draw (mu22_out1) to[out=0,in=180] (mu3_in1);
	\draw (outer.west|-mu3_in2) -- (mu3_in2);
	\draw (mu3_out1) -- (mu3_out1-|h.west);
	\draw (eps3_out1) -- (del4_in1);
	\draw (del4_out1) -- (del4_out1-|outer.east);
	\draw (del4_out2) -- (del4_out2-|outer.east);
	\draw (h_out1) -- (h_out1-|outer.east);
\end{tikzpicture}
\end{aligned}
\end{equation}
and described as a text-string, as follows:
\[
 (1 \otimes f \otimes \counit{1em} \otimes 1) \cp (\mult{1.5em} \otimes 1 \otimes \mult{1.5em} \otimes 1)\cp
 (\unit{1em} \otimes g \otimes \mult{1.5em}) \cp (\comult{1.5em} \otimes h).
\]
This system of notation provides a way of describing patterns of interconnection
between $f$, $g$ and $h$. The subtlety---a nontrivial one---lies in
understanding when the morphisms defined by two different constructions should be
considered the same. Indeed, we might equally we have chosen to represent the
above picture as 
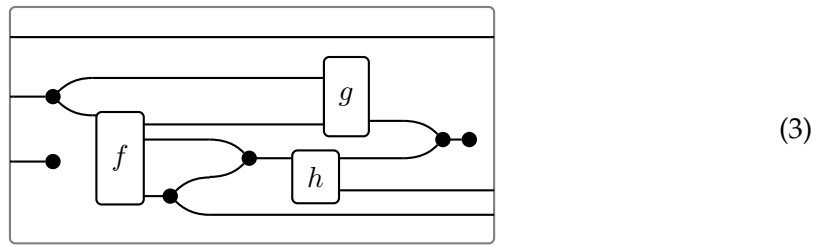
\begin{equation}
\label{eq.keyhyp2}
\begin{aligned}
\begin{tikzpicture}[penetration=0, oriented WD, spider diagram, bb min
width=.4cm, bby=.22cm, inner xsep=.2cm, bbx=.5cm]
	\node[spider={1}{2}] (del1) {};
	\node[spider={1}{0}, below=3 of del1] (eta1) {};
	\node[bb={0}{0}, minimum height=35pt, below right=-.2 and .1 of del1_out2] (f) {$f$};
	\node[special spider={1}{2}{.2em}{\leglen}, right=.5 of f.-58] (del2) {};
	\node[special spider={2}{1}{\leglen}{.2em}, right=2.6 of f] (mu1) {};
	\node[bb={0}{0}, minimum height=20pt, below right=-.8 and 1 of mu1] (h) {$h$};
	\node[bb={0}{0}, minimum height=30pt, right=7 of del1] (g) {$g$};
	\node[special spider={2}{1}{\leglen}{0}, right=6 of mu1_in1] (mu2) {}; 
	\node[special spider={1}{0}{0}{0}, right=.5 of mu2_out1] (eta2) {};
	\node at ($(f)-(1.5,0)$) (left) {};
	\node at ($(g)+(2.5,0)$) (right){};
	\node at ($(f)+(0,7.5)$) (top){};
	\node at ($(f)-(0,3.5)$) (bot){};
	\node[bb={0}{0}, color=gray, fit=(top) (left) (bot) (right)] (outer) {};
	\draw (outer.160) -- (outer.20);
	\draw (outer.west|-del1_in1) -- (del1_in1);
	\draw (del1_out1) -- (del1_out1-|g.west);
	\draw (del1_out2) -- (del1_out2-|f.west);
	\draw (f.55) -- (f.55-|g.west);
	\draw (outer.west|-eta1_in1) -- (eta1_in1);
	\draw (f.east|-mu1_in1) -- (mu1_in1);
	\draw (f.east|-del2_in1) -- (del2_in1);
	\draw (del2_out1) -- (mu1_in2);
	\draw (del2_out2) -- (del2_out2-|outer.east);
	\draw (mu1_out1) -- (mu1_out1-|h.west);
	\draw (g.east|-mu2_in1) -- (mu2_in1);
	\draw (h.east|-mu2_in2) -- (mu2_in2);
	\draw (mu2_out1) -- (eta2);
	\draw (h.-30) -- (h.-30-|outer.east);
\end{tikzpicture}
\end{aligned}
\end{equation}
Similarly, but more simply, the diagrams
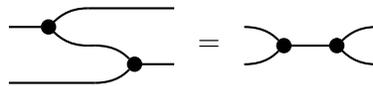
\[
\begin{tikzpicture}[spider diagram]
	\node[spider={1}{2}] (a) {};
	\node[spider={2}{1}, below right=\leglen*2/3 and 1 of a] (b) {};
	\coordinate (c1) at (a_out1-|b_out1);
	\coordinate (c2) at (a_in1|-b_in2);
	\draw (a_out2) -- (b_in1);
	\draw (a_out1) -- (c1);
	\draw (c2) -- (b_in2);
	\node[special spider={2}{1}{\leglen}{0}, right=2.5 of a_out2] (aa) {};
	\node[special spider={1}{2}{0}{\leglen}, right=.5 of aa] (bb) {};
	\draw (aa_out1) -- (bb_in1);
	\coordinate (h1) at ($(b_out1)!.5!(aa_in1)$);
	\node at (h1|-aa) {$=$};
\end{tikzpicture}
\]
both represent the same pattern of interconnection.

Nonetheless, despite this subtlety, the method we have just outlined has been
independently rediscovered, with various names and motivations, many times over
the past few decades. The first were Carboni and Walters, who called this
structure a well-support compact closed category, and used it to study
categories of relations, as well as labelled transition systems and automata
\cite{Carboni:1991a}.  Bruni and Gadducci called it a dgs-monoidal category when studying Petri nets \cite{gadducci1997inductive}.  Morton called it a 
dungeon category when studying belief propagation \cite{MortonBelief2014}.
Finally, converging on the name \emph{hypergraph category}, Kissinger used it to study quantum systems and graph rewriting, and Fong,
Baez, and Pollard used it to study electric circuits and chemical reaction
networks
\cite{kissinger2015finite,fong2015decorated,baez2015compositional,baez2016compositional}. 

Simultaneously, Spivak defined essentially the same structure in his work on
databases \cite{Spivak:2013b}, but preferred a more uniform, combinatorial
approach. Instead of thinking about how to generate such pictures piece by piece,
Spivak focussed on writing down the connection patterns. For example, the
picture in \cref{eqn.main_pic} can be described as follows. First, we define three sets,
corresponding to the ports of the all three inner boxes ($A$, the white circles), the
intermediate nodes ($N$, the black circles), and the ports of the outer box
($B$, the gray circles).
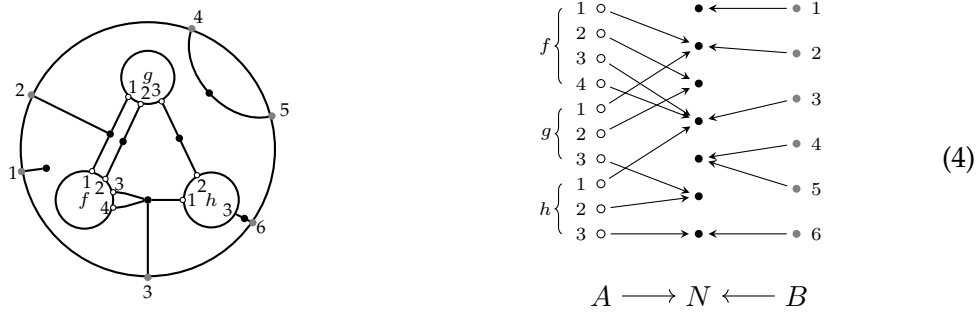
\begin{equation} \label{eqn.cospan_pic}
\begin{aligned}
\begin{tikzpicture}[penetration=0, unoriented WD,  pack size=20pt, pack inside
color=white, pack outside color=black, link size=2pt, font=\tiny, scale=2]
	\node[pack] at (-1.5,-1) (f) {$f$};
	\node[pack] at (0,1.9) (g) {$g$};
	\node[pack] at (1.5,-1) (h) {$h$};
	\node[outer pack, inner sep=34pt] at (0,.2) (out) {};
	\node[link] at ($(f)!.5!(h)$) (link1) {};
	\node[link] at (-2.4,-.25) (link2) {};
	\node[link] at ($(f.75)!.5!(g.-135)$) (link3) {};
   \begin{scope}[label distance=-6pt]
	\draw (out.270) node[link, gray]{} 
		to node[pos=-.2]{3} 
		(link1);
	\draw (out.190) node[link, gray]{} 
		to node[pos=-.4]{1} 
		(link2);
	\draw (out.155) node[link, gray]{} 
		to node[pos=-.15]{2} 
		(link3);
	\draw (out.-35) node[link, gray]{}
		to node[pos=.5,link]{} node[pos=-.5]{6} node[pos=1.5]{3} 
		(h.-30);
	\draw (out.15) node[link, gray]{}
		to[out=-165,in=-110] node[pos=.5,link]{} node[pos=-.1]{5}
		node[pos=1.1]{4}
		(out.70) node[link, gray]{};
	\draw (f.15) node[link,fill=white,thin,label=45:3]{} 
		to[out=0,in=165]
		(link1);
	\draw (f.-15) node[link,fill=white,thin]{}
		to[out=0,in=-165] node[pos=-.2]{4} 
		(link1);
	\draw (h.180) node[link,fill=white,thin]{}
		to node[pos=-.3]{1}  
		(link1);
	\draw (g.-60) node[link,fill=white,thin]{}
		to node[pos=.5,link]{} node[pos=-.15]{3} node[pos=1.15]{2} 
		(h.120) node[link,fill=white,thin]{};
	\draw (f.45) node[link,fill=white,thin,label=-140:2]{}
		to node[pos=.5,link]{} node[pos=1.15]{2} 
		(g.-105) node[link,fill=white,thin]{};
	\draw (f.75) node[link,fill=white,thin]{}
		to node[pos=-.3]{1} 
		(link3);
	\draw (g.-135) node[link,fill=white,thin]{}
		to node[pos=-.3]{1} 
		(link3);
   \end{scope}
\end{tikzpicture}
\end{aligned}
\hspace{3cm}
\begin{aligned}
\begin{tikzpicture}[xscale=1.3]
	\foreach \i in {0,...,9} {
		\node[circle, minimum size=3pt, inner sep=0, draw] at
		(0,-\i/3+1) (L\i) {};
	}
	\foreach \i in {0,...,6} {
		\node[circle, minimum size=3pt, inner sep=0, fill, black] at (1,-\i/2+1) (M\i) {};
	}
	\foreach \i in {0,...,5} {
		\node[circle, minimum size=3pt, inner sep=0, fill, gray] at
		(2,-3*\i/5+1) (N\i) {};
	}
	\node[below=.5 of L9] (a) {$A$};
	\node[below=.5 of M6] (n) {$N$};
	\node[below=.5 of N5] (b) {$B$};
	\draw[->] (a) to (n);
	\draw[->] (b) to (n);
\begin{scope}[font=\tiny, decoration=brace]
	\node[left=0pt of L0] {$1$};
	\node[left=0pt of L1] {$2$};
	\node[left=0pt of L2] {$3$};
	\node[left=0pt of L3] {$4$};
	\node[left=0pt of L4] {$1$};
	\node[left=0pt of L5] {$2$};
	\node[left=0pt of L6] {$3$};
	\node[left=0pt of L7] {$1$};
	\node[left=0pt of L8] {$2$};
	\node[left=0pt of L9] {$3$};
	\draw[decorate] ($(L3)+(-.4,0)$) to node[left] {$f$} ($(L0)+(-.4,0)$);
	\draw[decorate] ($(L6)+(-.4,0)$) to node[left] {$g$} ($(L4)+(-.4,0)$);
	\draw[decorate] ($(L9)+(-.4,0)$) to node[left] {$h$} ($(L7)+(-.4,0)$);
	\node[right=0pt of N0] {$1$};
	\node[right=0pt of N1] {$2$};
	\node[right=0pt of N2] {$3$};
	\node[right=0pt of N3] {$4$};
	\node[right=0pt of N4] {$5$};
	\node[right=0pt of N5] {$6$};
\end{scope}
\begin{scope}[shorten <= 2pt, shorten >= 2pt, ->,>=stealth]
	\draw (L0) to (M1);
	\draw (L1) to (M2);
	\draw (L2) to (M3);
	\draw (L3) to (M3);
	\draw (L4) to (M1);
	\draw (L5) to (M2);
	\draw (L6) to (M5);
	\draw (L7) to (M3);
	\draw (L8) to (M5);
	\draw (L9) to (M6);
	\draw (N0) to (M0);
	\draw (N1) to (M1);
	\draw (N2) to (M3);
	\draw (N3) to (M4);
	\draw (N4) to (M4);
	\draw (N5) to (M6);
\end{scope}
\end{tikzpicture}
\end{aligned}
\end{equation}
The picture is then described by a pair of functions $A \to N$, $B \to N$, that
say how the wires on the boxes connect to the intermediate nodes. Writing
this pair as $A \to N \from B$, we call this a cospan. 

It was already noticed by Carboni and Walters that these two approaches should
be similiar, and aspects of this cospan idea have appeared in almost all the references
above. In this paper we pin down the exact relationship. To do so
requires a thorough investigation of hypergraph categories and their functors,
including discussion of self-dual compact closed structure, free hypergraph
categories, a factorization system on hypergraph functors, and a coherence
theorem for hypergraph categories. Let us be a bit more precise.

\subsection{Composition, wiring diagrams, and cospans}

What is most relevant about the above diagrams is that they can be
\emph{composed}: new diagrams can be built from old. Let's explore how
composition works for both hypergraph categories and cospan-algebras.

We refer to the primitives that represent wires merging, initializing, splitting, and
terminating as \emph{Frobenius generators}, and when their composites obey laws
reflecting the above intuition about interconnection, we call the resulting
structure a \emph{special commutative Frobenius monoid}. A hypergraph category is a
symmetric monoidal category in which every object is equipped with the structure of a special
commutative Frobenius monoid in a way compatible with the monoidal product.

The monoidal structure gives notions of composition that come from
concatenation: we may build new diagrams by placing them end to end---the
categorical composition---or side by side---the monoidal product. The Frobenius
generators, as special morphisms, take care of the network structure.

There is another perspective, however: that of substitution.
Below is a pictorial representation of the sort of composition that makes sense in categories, monoidal categories, traced monoidal categories, and hypergraph categories.
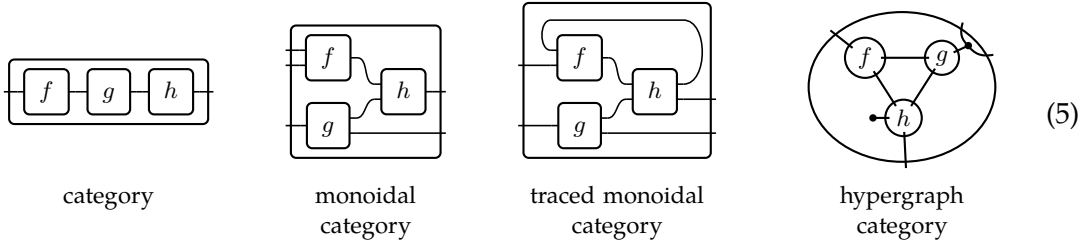
\begin{equation}\label{eqn.various_shapes}
\begin{aligned}
\begin{tikzpicture}
\begin{scope}[font=\footnotesize]
  \begin{scope}[oriented WD, bb port sep=1, bb port length=2.5pt, bb min width=.4cm, bby=.2cm, inner xsep=.2cm, x=.5cm, y=.3cm, text height=1.5ex, text depth=.5ex]
  	\node[bb={1}{1}] (Catf) {$f$};
  	\node[bb={1}{1}, right=.5 of Catf] (Catg) {$g$};
  	\node[bb={1}{1}, right=.5 of Catg] (Cath) {$h$};
  	\node[bb={0}{0}, fit=(Catf) (Cath)] (Cat) {};
  	\node[coordinate] at (Cat.west|-Catf_in1) (Cat_in1) {};
  	\node[coordinate] at (Cat.east|-Catg_out1) (Cat_out1) {};
  	\draw[shorten <=-2pt] (Cat_in1) -- (Catf_in1);
  	\draw (Catf_out1) -- (Catg_in1);
  	\draw (Catg_out1) -- (Cath_in1);
  	\draw[shorten >=-2pt] (Cath_out1) -- (Cat_out1);
  	\node[bb={2}{1}, above right=-.5 and 3 of Cath] (Monf) {$f$};
  	\node[bb={1}{2}, below=1 of Monf] (Mong) {$g$};
		\node[bb={2}{1}] at ($(Monf)!.5!(Mong)+(2,0)$) (Monh) {$h$}; 
  	\node[bb={0}{0}, fit=(Monf) (Mong) (Monh)] (Mon) {};
  	\node[coordinate] at (Mon.west|-Monf_in1) (Mon_in1) {};
  	\node[coordinate] at (Mon.west|-Monf_in2) (Mon_in2) {};
  	\node[coordinate] at (Mon.west|-Mong_in1) (Mon_in3) {};
  	\node[coordinate] at (Mon.east|-Monh_out1) (Mon_out1) {};
  	\node[coordinate] at (Mon.east|-Mong_out2) (Mon_out2) {};
  	\draw[shorten <=-2pt] (Mon_in1) -- (Monf_in1);
  	\draw[shorten <=-2pt] (Mon_in2) -- (Monf_in2);
  	\draw[shorten <=-2pt] (Mon_in3) -- (Mong_in1);
  	\draw[shorten >=-2pt] (Monh_out1) -- (Mon_out1);
  	\draw[shorten >=-2pt] (Mong_out2) -- (Mon_out2);
		\draw (Monf_out1) to (Monh_in1);
		\draw (Mong_out1) to (Monh_in2);
  	\node[bb={2}{1}, right= 5.5 of Monf] (Trf) {$f$};
  	\node[bb={1}{2}, below=1 of Trf] (Trg) {$g$};
		\node[bb={2}{2}] at ($(Trf)!.5!(Trg)+(2,0)$) (Trh) {$h$}; 
  	\node[bb={0}{0}, fit={($(Trf.north west)+(-.5,1)$) (Trg) ($(Trh.north east)+(.5,0)$)}] (Tr) {};
  	\node[coordinate] at (Tr.west|-Trf_in2) (Tr_in1) {};
  	\node[coordinate] at (Tr.west|-Trg_in1) (Tr_in2) {};
  	\node[coordinate] at (Tr.east|-Trh_out2) (Tr_out1) {};
  	\node[coordinate] at (Tr.east|-Trg_out2) (Tr_out2) {};
  	\draw[shorten <=-2pt] (Tr_in1) -- (Trf_in2);
  	\draw[shorten <=-2pt] (Tr_in2) -- (Trg_in1);
  	\draw[shorten >=-2pt] (Trh_out2) -- (Tr_out1);
  	\draw[shorten >=-2pt] (Trg_out2) -- (Tr_out2);
		\draw (Trf_out1) to (Trh_in1);
		\draw (Trg_out1) to (Trh_in2);
  	\draw let \p1=(Trh.east), \p2=(Trf.north west), \n1=\bbportlen, \n2=\bby in
  		(Trh_out1) to[in=0] (\x1+\n1,\y2+\n2) -- (\x2-\n1,\y2+\n2) to[out=180] (Trf_in1);
  \end{scope}
  \begin{scope}[unoriented WD, penetration=0, pack outside color=black, pack inside color=white, link size=2pt]
  	\node[pack, above right=-.5 and 3.3 of Trf] (Hypf) {$f$};
  	\node[pack, right=.5 of Hypf] (Hypg) {$g$};
		\node[pack] at ($(Hypf)!.5!(Hypg)+(0,-.8)$) (Hyph) {$h$};
  	\node[outer pack, inner xsep=3pt, inner ysep=1pt, fit=(Hypf) (Hypg) (Hyph)] (Hyp) {};
  	\node[link] at ($(Hypg.30)!.5!(Hyp.30)$) (link1) {};
  	\node[link,  left=.1 of Hyph.west] (link2) {};
  	\draw (Hypg) -- (link1);
  	\draw[shorten >= -2pt] (link1) to[bend right] (Hyp.20);
  	\draw[shorten >= -2pt] (link1) to[bend left] (Hyp.45);
		\draw (link2) -- (Hyph);
  	\draw[shorten >= -2pt] (Hypf) -- (Hyp);
  	\draw[shorten >= -2pt] (Hyph) -- (Hyp);
		\draw (Hypf) -- (Hyph);
		\draw (Hypf) -- (Hypg);
		\draw (Hypg) -- (Hyph);
  \end{scope}
  \begin{scope}[text height=1.5ex, text depth=.5ex, align=center]
  	\node[below=.65 of Cat.south] (Cat name) {category};
  	\node[text width=1.5cm] at (Cat name-|Mon) {monoidal category};
  	\node[text width=2.5cm] at (Cat name-|Tr) {traced monoidal\\category};
  	\node[text width=2cm] at (Cat name-|Hyp) {hypergraph category};
	\end{scope}
\end{scope}
\end{tikzpicture}
\end{aligned}
\end{equation}
The above pictures are known as wiring diagrams. Here we think of the outer box
as of the same nature as the inner boxes, which allows \emph{substitution} of
one wiring diagram into another.
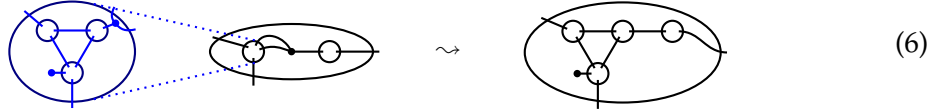
\begin{equation}\label{eqn.subst}
\begin{tikzpicture}[unoriented WD, penetration=0, pack inside color=white, link size=2pt, font=\footnotesize, spacing=20pt, baseline=(leadsto)]
\begin{scope}[blue]
	\node[pack] (Hf) {};
	\node[pack, right=.5 of Hf] (Hg) {};
	\node[pack] at ($(Hf)!.5!(Hg)+(0,-.8)$) (Hh) {};
	\node[outer pack, inner xsep=3pt, inner ysep=1pt, fit=(Hf) (Hg) (Hh)] (H) {};
	\node[link, color=blue] at ($(Hg.30)!.5!(H.30)$) (link1) {};
	\node[link, color=blue,  left=.1 of Hh.west] (link2) {};
	\draw (Hg) -- (link1);
	\draw[shorten >= -2pt] (link1) to[bend right] (H.20);
	\draw[shorten >= -2pt] (link1) to[bend left] (H.45);
	\draw (link2) -- (Hh);
	\draw[shorten >= -2pt] (Hf) -- (H);
	\draw[shorten >= -2pt] (Hh) -- (H);
	\draw (Hf) -- (Hh);
	\draw (Hf) -- (Hg);
	\draw (Hg) -- (Hh);
\end{scope}
\begin{scope}[pack outside color=black]
	\node [pack, right=2 of H.east] (Gx) {};
	\node [pack, right=1 of Gx] (Gy) {};
	\node [outer pack, inner xsep=3pt, inner ysep=3pt, fit=(Gx) (Gy)] (G) {};
	\node [link] at ($(Gx)!.5!(Gy)$) (link3) {};
	\draw (Gx.70) to[out=70, in=110] (link3);
	\draw (Gx.20) to[out=20, in=160] (link3);
	\draw (Gx.270) -- ($(G.south-|Gx.270)$);
	\draw (Gx.150) -- (G.170);
	\draw (link3) -- (Gy.west);
	\draw (Gy.east) -- (G.east);
	\draw[blue, dotted] (H.60) -- (Gx.north);
	\draw[blue, dotted] (H.-60) -- (Gx.south);
	\node[pack, right=4 of G.north east] (If) {};
	\node[pack, right=.5 of If] (Ig) {};
	\node[pack] at ($(If)!.5!(Ig)+(0,-.8)$) (Ih) {};
	\node[link,  left=.1 of Ih.west] (link4) {};
	\node [pack, right=.5 of Ig] (Iy) {};
	\node [outer pack, inner xsep=3pt, inner ysep=1pt, fit=(If) (Ig) (Ih) (Iy)] (I) {};
	\draw[shorten >= -2pt] (If) -- (I);
	\draw (If) -- (Ih);
	\draw (Ih) -- (link4);
	\draw (Ih.270) -- ($(I.south-|Ih.270)$);
	\draw (If) -- (Ig);
	\draw (Ig) -- (Ih);
	\draw (Ig) -- (Iy);
	\draw (Iy.0) to[out=0,in=180] (I.0);
	\node at ($(G.east)!.5!(I.west)$) (leadsto) {$\leadsto$};
\end{scope}
\end{tikzpicture}
\end{equation}
More formally, boxes, wiring diagrams, and substitution can be represented as objects, morphisms, and composition in an operad. The rules of this substitution---e.g.\ whether or not the ``special rule''
$
\begin{tikzpicture}[x=.2cm, baseline=(bot)]
	\node[link] (l1) {};
	\node[link, right=1 of l1] (l2) {};
	\draw (l1) -- +(-1,0);
	\draw (l2) -- +(1,0) node[coordinate] (r) {};
	\draw (l1) to[bend right] node[below=-3pt] (bot) {} (l2);
	\draw (l1) to[bend left] (l2);
	\node [coordinate, right=4 of r] (p1) {};
	\node [coordinate, right=1 of p1] (p2) {};
	\draw (p1) -- (p2);
	\node[font=\footnotesize] at ($(r)!.5!(p1)$) {$=$};
\end{tikzpicture}
$
holds, a question one might ask themselves if checking the details of \cref{eqn.subst}---are controlled by this operad. The above operadic viewpoint on wiring diagrams was put forth by Spivak and collaborators \cite{Spivak:2013b,Rupel.Spivak:2013a,Vagner.Spivak.Lerman:2015a}. In particular it was shown in \cite{Spivak.Schultz.Rupel:2016a} that the operad governing traced monoidal categories is $\Cat{Cob}$, the operad of oriented 1-dimensional cobordisms.

In this paper we prove a similar result: the operad governing hypergraph categories is $\cospan[]$. Informally, what this means is that there is a one-to-one correspondence between the wiring diagrams that can be interpreted in a hypergraph category $\cat{H}$---or more precisely, equivalence classes thereof---and cospans labeled by the objects of $\cat{H}$. 

Strictly speaking, every morphism in a category---including in a hypergraph category---has a domain and codomain, and thus should be represented as a two-sided figure, say a box with left and right sides, just like in the first three cases of \cref{eqn.various_shapes}. However, ``morally speaking'' (in the sense of \cite{cheng2004mathematics}), a morphism $f\in\cat{H}$ in a hypergraph category is indexed not by a pair of objects $x_1,x_2\in\ob\cat{H}$, serving as the domain and codomain of $f$, but instead by a finite set $\{x_1,\ldots,x_n\}\ss\ob\cat{H}$ of objects, which one can visualize as an ``o-mane'' for $f$, e.g.:
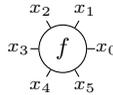
\[
\begin{tikzpicture}
	\node[circle, draw, inner sep=3pt][font=\footnotesize] (circ) {$f$};
	\foreach \i [evaluate=\i as \ii using {int(\i/60)}] in {0, 60, ..., 300} {
		\draw (circ.\i) -- +(\i:3pt) node[pos=2.5][font=\tiny] {$x_\ii$};
	}
\end{tikzpicture}
\]
The reason not to distinguish between inputs and outputs is that the structures
and axioms of hypergraph categories allow us to ``bend arrows'' arbitrarily, as
we see in the difference between \cref{eq.keyhyp1} and \cref{eq.keyhyp2}.
The axioms of hypergraph categories ensure that these two diagrams denote the same composite morphism: directionality is irrelevant.

Thus we think of the cospan representation as an \emph{unbiased} viewpoint on hypergraph categories. As an analogy, consider the case of ordinary monoids. A monoid is usually presented as a set $M$ together with a binary operation $*\colon M\times M\to M$ and a constant, or 0-ary operation, $e\in M$, satisfying three equations. Once this structure is in place, one can uniquely define an $n$-ary operation, for any other $n$, by iterating the 2-ary operation. 

An unbiased viewpoint on monoids is one in which all the $n$-ary operations are put on equal footing, rather than having 0- and 2-ary morphisms be special. One such approach is to say that a monoid is an algebra on the $\List$ monad: it is a set $X$ equipped with a function $h\colon\List(X)\to X$ satisfying the usual monad-algebra equations. The 0-ary and 2-ary case are embedded in this structure as $h$ applied to lists of length $0$ and $2$, respectively. Another unbiased approach is to use operads, which gives a very simple description: a monoid is an algebra on the terminal operad.

We similarly use cospan-algebras in this article to provide an unbiased
viewpoint on hypergraph categories. However, doing so has a cost: while
hypergraph categories and the functors between them are roughly cospan-algebras,
the corresponding statement does not hold when one considers 2-categorical
aspects. In other words, the natural transformations between hypergraph functors
are not visible in the cospan formulation. Indeed, one can consider the category
of cospan-algebras as a decategorification of the 2-category $\hhyp$ of
hypergraph categories.

\subsection{Statement of main theorems}\label{sec.state_main_theorems}

Our first theorem is a strictification theorem. If $\cat{H}$ is a strict
hypergraph category, an \emph{objectwise-free} ($\of$) structure on $\cat{H}$ is
a set $\Lambda$ and a monoid isomorphism $i\colon\List(\Lambda)\cong\ob\cat{H}$;
in this case we say that $\cat{H}$ is $\of$ or \emph{objectwise-free} on
$\Lambda$. Let $\hhyp_\of$ denote the 2-category that has $\of$-hypergraph
categories as objects and for which 1- and 2-morphisms are those between
underlying hypergraph categories. In other words, we have a full and faithful
functor $U\colon\hhyp_\of\to\hhyp$. The strictification theorem says that every
hypergraph category is equivalent to an $\of$-hypergraph category.

\begin{theorem}\label[theorem]{thm.2_equiv}
The 2-functor $U\colon\hhyp_\of\to\hhyp$ is a 2-equivalence.
\end{theorem}

Our main theorem says that the category of cospan-algebras is a decategorification of $\hhyp_\of$; in particular that it is isomorphic to the underlying 1-category $\hyp_\of$ of $\of$-hypergraph categories and all hypergraph functors between them. Before we can state this theorem, we need to say exactly what we mean by the category of cospan-algebras.

Let $\Lambda$ be a set; we think of this as a set of wire labels. By a
\emph{$\Lambda$-labeled finite set}, we mean a nonnegative integer $m\in\nn$ and a
function $x\colon\ul{m}\to\Lambda$, where $\ul{m}\coloneqq\{1,\ldots,m\}$; in
other words, just a list of elements in $\Lambda$. Let $\cospan$ denote
the category whose objects are $\Lambda$-labeled finite sets $(m,x)$ and whose
morphisms $f\colon (m,x)\to(n,y)$ are labeled cospans, i.e.\ isomorphism classes of commutative diagrams%
\footnote{Two labeled cospans $(f_1,p,z,f_2)$ and $(f_1',p',z',f_2')$ as in \cref{eqn.def_cospan} are considered equivalent if there is a bijection $i\colon p\cong p'$ with $f_1'=f_1\cp i$, $f_2'=f_2\cp i$, and $z=i\cp z'$.
\label[footnote]{fn.equiv_cospans}}
\begin{equation}\label{eqn.def_cospan}
\begin{tikzcd}[row sep=15pt]
	\ul{m}\ar[dr, bend right, "x"']\ar[r, "f_1"]&\ul{p}\ar[d, "z"]&\ul{n}\ar[dl, bend left, "y"]\ar[l, "f_2"']\\
	&\Lambda
\end{tikzcd}
\end{equation}
Let $\hyp_{\of(\Lambda)}\ss\hyp$ denote the subcategory of identity-on-objects functors between hypergraph categories that are objectwise-free on $\Lambda$. We will prove that this category is isomorphic to that of lax monoidal functors $a\colon\cospan\to\smset$:
\begin{equation}\label{eqn.iso_one_hyp}
	\hyp_{\of(\Lambda)}\cong\Cat{Lax}(\cospan,\smset).
\end{equation}

Maps in $\hyp_\of$ are just hypergraph functors $F\colon\cat{H}\to\cat{H}'$ between hypergraph categories that happen to be $\of$. In particular they need not send generators to generators; instead they send each generator in $\cat{H}$ to an arbitrary object, which is identified with a \emph{list of generators} in $\cat{H}'$. Let $\smset_\List$ denote the Kleisli category of the list-monad, i.e.\ the category whose objects are sets, e.g.\ $\Lambda$, and for which a morphism from $\Lambda$ to $\Lambda'$ is a function $\Lambda\to\List(\Lambda')$. We will explain that the on-objects part of $F$ induces a functor $\cospan[\ob(F)]\colon\cospan[\Lambda]\to\cospan[\Lambda']$, and that the on-morphisms part of $F$ induces a monoidal natural transformation $\alpha$
\[
\begin{tikzcd}[row sep=0pt]
	\cospan[\Lambda]\ar[dd, "{\cospan[\ob(F)]}"']\ar[rd, bend left=10pt, "a"]\\
	\ar[r, phantom, "\Downarrow\scriptstyle\alpha"]&\smset\\
	\cospan[\Lambda']\ar[ru, bend right=10pt, "a'"']
\end{tikzcd}
\]
Moreover, we will show that every morphism of cospan-algebras arises in this way.

Using the Grothendieck construction, we can package the above isomorphisms
\eqref{eqn.iso_one_hyp} into a single one; this is our second main theorem.

\begin{theorem}\label[theorem]{thm.iso_hypof_cospan_alg}
There is an isomorphism of 1-categories,
\[
  \hyp_\of\quad\cong\quad\int^{\Lambda\in\smset_\List}\Cat{Lax}(\cospan,\smset).
\]
\end{theorem}
\smallskip

\subsection{Plan of paper}
This paper has three remaining sections. In \cref{sec.review} we formally
introduce the key concepts, cospan-algebras
(\textsection\ref{sec.cospan_and_algs}), Frobenius monoids
(\textsection\ref{sec.frob_monoids}), and hypergraph categories
(\textsection\ref{sec.hypergraph_cats}), giving a few basic examples, and
pointing out some basic facts. In particular, in \cref{sec.critiques} we note
that hypergraph categories do not obey the principle of equivalence, which we
argue motivates the cospan-algebra perspective, and in \cref{sec.word_on_operads} we remark on the interaction between the operadic and the monoidal
categorical perspectives.

\cref{sec.properties} develops the theory of hypergraph categories.  We discuss
four key properties. In \cref{sec.sdccc}, we see that hypergraph categories have a natural
self-dual compact closed structure, so morphisms may be described by their
so-called names. In \cref{sec.free_cats}, we show that $\cospan$ is the free
hypergraph category over $\Lambda$. In \cref{sec.io_ff}, we see that there is an
(identity-on-objects, fully faithful) factorization of any hypergraph functor,
and this implies that the category of hypergraph categories is fibred over
$\smset_\List$. In \cref{sec.strictification}, we prove a coherence theorem,
showing that the 2-category of hypergraph categories is 2-equivalent to the
2-category of those that are objectwise-free.

The final section, \cref{sec.hyp_cospan_alg}, is devoted to proving that the
1-categories of cospan-algebras and objectwise-free hypergraph categories are
equivalent. We do this by first showing how cospan-algebras may be constructed
from hypergraph categories (\textsection\ref{ssec.hcca}), then how hypergraph
categories may be constructed from cospan-algebras
(\textsection\ref{ssec.cahc}), and finally that these two constructions define
an equivalence of categories (\textsection\ref{ssec.equivalences}).

\subsection{Acknowledgments}
The authors thank John Baez, Steve Lack, Marcelo Fiore, Samson Abramsky, and Maru Sarazola for
useful conversations. This work was supported by USA AFOSR grants
FA9550-14-1-0031 and FA9550-17-1-0058.


\subsection{Notation and terminology}
\begin{itemize}
\item We generally denote composition in 1-categories using diagrammatic order,
writing $f\cp g$ rather than $g\circ f$.

\item If $\cat{C}$ is a category, we may write $c\in\cat{C}$ to mean $c$ is an object in $\cat{C}$.

\item 
Unless otherwise indicated, we use $\otimes$ to denote the monoidal product in a
monoidal category, and $I$ to denote the monoidal unit.

\item 
By $\smset$ we always mean the symmetric monoidal category of (small) sets and
functions, with the Cartesian product monoidal structure $(\{*\},\times)$.

\item 
By $\finset$ we always mean the (strict skeleton of the) symmetric monoidal
category of finite sets and functions, with the coproduct monoidal structure
$(\varnothing,+)$.

\item 
Following John Baez and his students \cite{moeller2018noncommutative}, we refer to the coherence maps for lax monoidal functors as \emph{laxators}; this is in keeping with widely-used terms like unitor, associator, etc.

\item 
For any natural number $m\in\nn$, we abuse notation to also let
$\ul{m}\coloneqq\{1,\ldots,m\}$; in particular $\ul{0}=\varnothing$. If the
coproduct of $X$ and $Y$ exists, we denote its coproduct by $X+Y$; given
morphisms $f\colon X\to Z$ and $g\colon Y\to Z$, we denote the universal
morphism from the coproduct by $\copair{f,g}\colon X+Y\to Z$.

\item 
If $\Lambda$ is a set, we denote by $\List(\Lambda)$ the set of pairs $(m,x)$,
where $m\in\nn$ and $x\colon\ul{m}\to\Lambda$; we may also denote $(m,x)$ by
$[x_1,\ldots,x_m]$. We may denote the list simply as $x$, in which case it will
be convenient to denote the indexing set $\ul{m}$ as $|x|$. It is well-known
that $\List$ is a functor $\List\colon\smset\to\smset$, and extends to a monad
when equipped with the singleton list transformation
$\sing\colon\id_\smset\to\List$ and the flatten transformation
$\flatten\colon\List\cp\List\to\List$. We denote the concatenation of lists $x$
and $y$ by $x\concat y$, and we often denote the empty list by $\emptylist$.

\item 
Given a functor $F\colon \cat{C} \to \smcat\op$, we write $\int^\cat{C} F\to\cat{C}$ for the
\define{Grothendieck construction} on $F$; this is the category over $\cat{C}$ that has objects given by pairs
$(X,a)$, where $X \in \cat{C}$ and $a \in FX$, and morphisms $(X,a) \to (Y,b)$ given by pairs $(f,g)$, where $f\colon X
\to Y$ is a morphism in $\cat{C}$ and $g\colon a\to F(f)(b)$ is a morphism in $FX$.
\end{itemize}

\chapter{Basic definitions: cospan-algebras and hypergraph categories}\label{sec.review}
In this section we review the definitions of the basic concepts we will use:
cospan-algebras (\textsection\ref{sec.cospan_and_algs}), Frobenius structures
(\textsection\ref{sec.frob_monoids}), and hypergraph categories
(\textsection\ref{sec.hypergraph_cats}). We then discuss some perhaps
undesirable ways in which hypergraph categories do not behave well with respect
to equivalence of categories, hence motivating the cospan-algebra viewpoint
(\textsection\ref{sec.critiques}), and also briefly touch on the (disappearing)
role of operads in this paper (\textsection\ref{sec.word_on_operads}).

\section{Cospans and cospan-algebras}\label{sec.cospan_and_algs}

The main character in our story is $\cospan$, where $\Lambda$ is an arbitrary set. 
We already defined its objects, $\Lambda$-labelled finite sets $(m,x)$, and its morphisms,
which we call \emph{labeled cospans}, in \cref{eqn.def_cospan}. The composition
formula is given by pushout; see \cite{fong2018seven}. The monoidal unit is
denoted $\emptylist$ and defined to be $(0,!)$, where
$!\colon\emptyset\to\Lambda$ is the unique function. The monoidal product is
denoted $\concat$ and defined to be
$(m_1,x_1)\concat(m_2,x_2)\coloneqq(m_1+m_2,\copair{x_1,x_2})$. Note that
$\cospan$ is the category of cospans in the category $\finset_\Lambda$, in
which objects are $\Lambda$-labelled finite sets $(m,x)$, and morphisms are simply
\emph{functions} between them that respect the typing $x$.

When $\Lambda$ is a one-element set, we can suppress it from the notation and
simply write $\cospan[]$. It is the usual category whose objects are finite sets
and whose morphisms are isomorphism classes of cospans, as discussed in \cref{fn.equiv_cospans}.

Note that flattening lists is a coproduct operation. Given
$x\colon\ul{m}\to\List(\Lambda)$, we have $m$-many indexing sets $|x_1|, \ldots,
|x_m|$ and maps $|x_i|\To{x_i}\Lambda$. The flattened list $\flatten(x)$ is indexed by the coproduct of the indexing sets $|x_i|$, and its content is given by the universal map:
\[
  [x_1,\ldots,x_m]\colon\sum_{i\in\ul{m}}|x_i|\to\Lambda.
\]

Given a function $f\colon\Lambda\to\List(\Lambda')$, we define a functor
$\cospan[f]\colon\cospan\to\cospan[\Lambda']$ as follows. For an object
$(m,x)\in\ob(\cospan)$, we obtain a function
$\ul{m}\To{x}\Lambda\To{f}\List(\Lambda')$, and hence
$\flatten(x\cp f)\in\List(\Lambda')=\ob(\cospan[\Lambda'])$ by applying the monad multiplication. Since $\List$ is functorial, composing a $\Lambda$-labeled cospan with $f$, as shown on the left, induces a $\Lambda'$-labeled cospan as shown on the right, by flattening:
\begin{equation}\label{eqn.cospan_on_kleisli_morphisms}
\begin{aligned}
\begin{tikzcd}
	{|x|}\ar[dr, bend right, "x"']\ar[r]&{|z|}\ar[d, "z"]&{|y|}\ar[dl, bend left, "y"]\ar[l]\\
	&\Lambda 
\end{tikzcd}
\end{aligned}
\hspace{.1in}
\xmapsto{\flatten(-\cp f)}
\hspace{.1in}
\begin{aligned}
\begin{tikzcd}
	\sum_{\ul{m}}|x\cp f|\ar[dr, bend right, "{\flatten(x\cp f)}"']\ar[r]&\sum_{\ul{p}}|z\cp f|\ar[d, "{\flatten(x\cp f)}"]&\sum_{\ul{n}}|y\cp f|\ar[dl, bend left, "{\flatten(y\cp f)}"]\ar[l]\\
	&\Lambda'
\end{tikzcd}
\end{aligned}
\end{equation}

\begin{proposition}\label[proposition]{prop.cospan_setlist_cat}
The above defines a functor $\cospan[-]\colon\smset_{\List}\to\smcat$.
\end{proposition}
\begin{proof}
We gave the data for the functor on objects $\Lambda\in\smset_\List$, namely
$\Lambda\mapsto\cospan[\Lambda]$, and on morphisms $\Lambda\to\List(\Lambda')$
in \cref{eqn.cospan_on_kleisli_morphisms}. To check that $\cospan[f]$ is a
functor, first note that it sends identity morphisms in $\cospan$ to those in
$\cospan[\Lambda']$. Then, observe that showing it preserves composition reduces
to checking that, for any $a,b,c\in\nn$ and pushout diagram as to the left below
\[
\begin{tikzcd}
	\ul{a}\ar[r]\ar[d]&\ul{c}\ar[d]\ar[ddr, bend left=20pt, "y"]\\
	\ul{b}\ar[r]\ar[drr, bend right=20pt, "x"']&\ul{b}\sqcup_{\ul{a}}\ul{c}\ar[ul, phantom, very near start, "\ulcorner"]\ar[dr, "z"]\\[-15pt]
	&&[-10pt]\List(\Lambda')
\end{tikzcd}
\hspace{1in}
\begin{tikzcd}
	\displaystyle\sum_{i\in\ul{a}}|w(i)|\ar[r]\ar[d]&
	\displaystyle\sum_{k\in\ul{c}}|y(k)|\ar[d]\\
	\displaystyle\sum_{j\in\ul{b}}|x(j)|\ar[r]&
	\displaystyle\sum_{l\in\ul{b}\sqcup_{\ul{a}}\ul{c}}|z(l)|
\end{tikzcd}
\]
where $w\colon\ul{a}\to\List(\Lambda')$ is the composite map, the diagram to the
right is also a pushout. This is an easy calculation. 

It is also straightforward to observe that $\cospan[-]$ is itself functorial:
$\cospan[\id_\Lambda] = \id_{\cospan}$, and if $f,g$ are composable morphisms in
$\smset_\List$, then $\cospan[f]\cp\cospan[g] =\cospan[f\cp g]$.
\end{proof}

\begin{definition}\label{def.cospan_alg}
A \define{cospan-algebra} consists of a set $\Lambda$, called the \emph{label set}, and a lax symmetric monoidal functor
\[
a \colon (\cospan,\concat) \longrightarrow (\smset,\times).
\]

Let $(\Lambda, a)$ and $(\Lambda',a')$ be cospan-algebras. A \emph{morphism} between them consists of a function $f\colon\Lambda\to\List(\Lambda')$ and a monoidal natural transformation $\alpha$ as shown here:
\[
\begin{tikzcd}[row sep=0pt]
	\cospan[\Lambda]\ar[dd, "{\cospan[f]}"']\ar[rd, bend left=10pt, "a"]\\
	\ar[r, phantom, "\Downarrow\scriptstyle\alpha"]&\smset\\
	\cospan[\Lambda']\ar[ru, bend right=10pt, "a'"']
\end{tikzcd}
\]
We write $\calg$ for the category of cospan-algebras and cospan-algebra
morphisms.
\end{definition}

\begin{example}\label{ex.part}
For a set $\Lambda$, we can define the cospan-algebra
$\prt_\Lambda\colon\cospan\to\smset$ represented by the monoidal unit $0$ of
$\cospan$. Explicitly, it is defined on $x\colon m\to\Lambda$ by
\[\prt_\Lambda(x)\coloneqq\cospan(0,x) = 
  \left\{ (p,f) \,\middle|
\begin{tikzcd}[row sep=15pt]
	\ul{p}\ar[d, "z"]&\ul{m}\ar[dl, bend left, "x"]\ar[l, "f"']\\
	\Lambda
\end{tikzcd}
\right\}
\]
the set of ways to partition $m$ into some number of (possibly empty) parts,
respecting $x\colon m \to \Lambda$. This is clearly lax monoidal: given functions $m\to m'$ and $n\to n'$, one obtains a function $m+n\to m'+n'$.
\end{example}

We shall see in \cref{sec.hyp_cospan_alg} that any hypergraph category gives rise to a cospan-algebra.

The following observation is immediate from \cref{def.cospan_alg}.
\begin{proposition}\label[proposition]{prop.calg_fibration}
We have an isomorphism of categories 
\[
\calg \cong \int^{\Lambda \in \smset_\List} \lax(\cospan,\smset).
\]
\end{proposition}

\section{Special commutative Frobenius monoids}\label{sec.frob_monoids}

In a hypergraph category, every object is equipped with the structure of a special commutative Frobenius monoid, which we call a Frobenius structure.
In this section we recall the definition and give important examples. 

We will represent morphisms in monoidal categories using the string diagrams 
introduced by Joyal and Street \cite{Joyal.Street:1993a}. We draw
$\swap{1em}\colon X \otimes Y\to Y \otimes X$ for the braiding in a symmetric
monoidal category. Diagrams are to be
read left to right; we shall suppress the labels, since we deal with a unique
generating object and a unique generator of each type.  

\begin{definition}\label[definition]{def.scfm}
  A \define{special commutative Frobenius monoid} $(X,\mu,\eta,\delta,\epsilon)$
  in a symmetric monoidal category $(\mathcal C, \otimes, I)$ is an object $X$ of
  $\mathcal C$ together with maps 
\[
\begin{tikzpicture}[spider diagram, dot fill=black]
	\node[spider={2}{1}, label={[below=.5]:{$\mu\colon X\otimes X\to X$}}] (a) {};
	\node[spider={0}{1}, label={[below=.5]:{$\eta\colon I\to X$}}, right=4 of a] (b) {};
	\node[spider={1}{2}, label={[below=.5]:{$\delta\colon X\to X\otimes X$}}, right=4 of b] (c) {};
	\node[spider={1}{0}, label={[below=.5]:{$\epsilon\colon X\to I$}}, right=4 of c] (d) {};
\end{tikzpicture}
\]  
obeying the commutative monoid axioms
\[
  \begin{tikzpicture}[spider diagram,scale=.8]
	\node[spider={2}{1}] (a) {};
	\node[special spider={2}{1}{\leglen}{0}, left=.1 of a_in1] (b) {};
	\draw (a_in1) to (b.east);
	\draw (a_in2) to (b_in1|-a_in2);
	\node[spider={2}{1}, right=3.5 of b] (aa) {};
	\node[special spider={2}{1}{\leglen}{0}, left=.1 of aa_in2] (bb) {};
	\draw (aa_in2) to (bb.east);
	\draw (aa_in1) to (bb_in1|-aa_in1);
	\node[label={[below=1]:{(associativity)}}] at ($(b.west)!.5!(aa.east)$) {$=$};
	\node[spider={2}{1}, right=2.2 of aa_out1] (a) {};
	\node[special spider={0}{1}{\leglen}{0}, left=.1 of a_in1] (b) {};
	\draw (a_in1) to (b.east);
	\draw (a_in2) to ($(a_in2)-(\leglen,0)$);
	\coordinate[right=1 of a_out1] (aa) {};
	\coordinate (bb) at ($(aa)+(3*\leglen,0)$) {};
	\draw (aa) to (bb);
	\node[label={[below=1]:{(unitality)}}] at ($(a_out1)!.5!(aa)$) {$=$};
	\node[spider={2}{1}, right=2.2 of bb] (a) {};
	\coordinate (b1) at ($(a_in1)-(.2,0)$);
	\coordinate (b2) at ($(a_in2)-(.2,0)$);
	\draw (a_in1) to (b1);
	\draw (a_in2) to (b2);
	\node[spider={2}{1}, right=2 of a_out1] (aa) {};
	\coordinate (bb1) at ($(aa_in1)-(.6,0)$);
	\coordinate (bb2) at ($(aa_in2)-(.6,0)$);
	\coordinate (bb) at ($(bb1)!.5!(bb2)$);
	\draw (bb1) to (aa_in2);
	\draw (bb2) to (aa_in1);
	\node[label={[below=1]:{(commutativity)}}] at ($(a_out1)!.5!(bb)$) {$=$};
\end{tikzpicture}
\]
the cocommutative comonoid axioms
\[
\begin{tikzpicture}[spider diagram,scale=.8]
	\node[spider={1}{2}] (a) {};
	\node[special spider={1}{2}{0}{\leglen}, right=.1 of a_out1] (b) {};
	\draw (a_out1) to (b_in1);
	\draw (a_out2) to (b_out1|-a_out2);
	\node[spider={1}{2}, right=2 of b] (aa) {};
	\node[special spider={1}{2}{0}{\leglen}, right=.1 of aa_out2] (bb) {};
	\draw (aa_out2) to (bb_in1);
	\draw (aa_out1) to (bb_out1|-aa_out1);
	\node[label={[below=1]:{(coassociativity)}}] at ($(b.west)!.5!(aa.east)$) {$=$};
	\node[spider={1}{2}, right=3 of aa] (a) {};
	\node[special spider={1}{0}{0}{\leglen}, right=.1 of a_out1] (b) {};
	\draw (a_out1) to (b_in1);
	\draw (a_out2) to ($(a_out2)+(\leglen,0)$);
	\coordinate[right=2 of a] (aa) {};
	\coordinate (bb) at ($(aa)+(3*\leglen,0)$) {};
	\draw (aa) to (bb);
	\node[label={[below=1]:{(counitality)}}] at ($(a_in1)!.5!(bb)$) {$=$};
	\node[spider={1}{2}, right=2 of bb] (a) {};
	\coordinate (b1) at ($(a_out1)+(.2,0)$);
	\coordinate (b2) at ($(a_out2)+(.2,0)$);
	\draw (a_out1) to (b1);
	\draw (a_out2) to (b2);
	\node[spider={1}{2}, right=2 of a] (aa) {};
	\coordinate (bb1) at ($(aa_out1)+(.6,0)$);
	\coordinate (bb2) at ($(aa_out2)+(.6,0)$);
	\coordinate (bb) at ($(bb1)!.5!(bb2)$);
	\draw (aa_out2) to (bb1);
	\draw (aa_out1) to (bb2);
	\node[label={[below=1]:{(cocommutativity)}}] at ($(a_in1)!.5!(bb)$) {$=$};
\end{tikzpicture}
\]
and the Frobenius and special axioms
\[
\begin{tikzpicture}[spider diagram,scale=.8]
	\node[spider={1}{2}] (a) {};
	\node[spider={2}{1}, below right=\leglen*2/3 and 1 of a] (b) {};
	\coordinate (c1) at (a_out1-|b_out1);
	\coordinate (c2) at (a_in1|-b_in2);
	\draw (a_out2) -- (b_in1);
	\draw (a_out1) -- (c1);
	\draw (c2) -- (b_in2);
	\node[special spider={2}{1}{\leglen}{0}, right=2.5 of a_out2] (aa) {};
	\node[special spider={1}{2}{0}{\leglen}, right=.5 of aa] (bb) {};
	\draw (aa_out1) -- (bb_in1);
	\node[spider={1}{2}, right=1.5 of bb_out2] (aaa) {};
	\node[spider={2}{1}, above right=\leglen*2/3 and 1 of aaa] (bbb) {};
	\coordinate (ccc1) at (aaa_out2-|bbb_out1);
	\coordinate (ccc2) at (aaa_in1|-bbb_in1);
	\draw (aaa_out1) -- (bbb_in2);
	\draw (aaa_out2) -- (ccc1);
	\draw (ccc2) -- (bbb_in1);
	\coordinate (h1) at ($(b_out1)!.5!(aa_in1)$);
	\coordinate (h2) at ($(aaa_in1)!.5!(bb_out1)$);	
	\coordinate (h3) at ($(aa_out1)!.5!(bb_in1)$);
	\node at (h1|-aa) {$=$};
	\node at (h2|-aa) {$=$};
	\node[below=.75 of h3] (Frob) {(Frobenius)};
	\node[spider={1}{2}, right=4 of aaa_out1] (aaaa) {};
	\node[spider={2}{1}, right=1 of aaaa] (bbbb) {};
	\draw (aaaa_out1) -- (bbbb_in1);
	\draw (aaaa_out2) -- (bbbb_in2);
	\coordinate[right=2.5 of aaaa] (aaaaa) {};
	\coordinate[right=1 of aaaaa] (bbbbb) {};
	\draw (aaaaa) -- (bbbbb);
	\node at ($(bbbb_out1)!.5!(aaaaa)$) (h4) {$=$};
	\node at (h4|-Frob) {(special)};
\end{tikzpicture}
\]
We say that $(\mu,\eta,\delta,\epsilon)$ is a \define{Frobenius structure} on
$X$, and call these four morphisms \define{Frobenius generators}. We further
refer to any morphism generated from these maps using composition, monoidal product,
identity maps, and braiding maps as a \define{Frobenius map}.
\end{definition}

\begin{example}\label[example]{ex.scFrob_on_I}
  In any symmetric monoidal category, there is a canonical Frobenius structure on
  the monoidal unit $I$. Indeed, the left and right unitors give (equal) isomorphisms
  $\lambda_I=\rho_I\colon I\otimes I\cong I$, so define $\mu_I\coloneqq\rho_I$ and
  $\delta_I\coloneqq\rho_I^{-1}$, and define $\eta_I=\id_I=\epsilon_I$. It is easy to see that this is the only possible Frobenius structure for which $\eta_I=\id_I=\epsilon_I$.
\end{example}

\begin{example}\label{ex.frob_on_strict_I}
  Let $\cat{C}$ be a symmetric monoidal category such that $I \otimes I = I$.
  For example, $\cat{C}$ could be a symmetric monoidal category with one
  object---that is, a commutative monoid considered as a one object category.

  A Frobenius structure on $I$ then consists of maps $\mu,\eta,\delta,\epsilon$,
  all of type $I\to I$. The morphisms $I\to I$ in a monoidal category always
  form a commutative monoid $(M,*,e)$; this follows from an Eckmann-Hilton
  argument \cite{Kelly.Laplaza:1980a}. The axioms of Frobenius structures then
  say that $\mu,\eta,\delta,\epsilon$, as elements of $M$, satisfy $\mu*\eta=e$,
  $\mu*\delta=e$, and $\delta*\epsilon=e$.  This implies that a Frobenius
  structure on the unit of a symmetric monoidal category can be identified with an invertible
  element $\mu$ in the monoid of scalars. 
\end{example}

\begin{example}\label[example]{ex.cospan_scFrob}
Consider the symmetric monoidal category $\cospan[]$. We will construct a
Frobenius structure on the object $1$. To do so, we need to define morphisms
$\mu\colon 1\concat 1\to 1$, $\eta\colon 0\to 1$, $\delta\colon 1\to 1\concat
1$, and $\epsilon\colon 1\to 0$ in $\cospan[]$, and then check that they satisfy
the required equations.
	
Recall that $1 \concat 1 =2$, the two element set $\{1,2\}$. Each Frobenius generator will
be the unique cospan of the required domain and codomain with apex $1$---this is
well-defined because $1$ is terminal in $\finset$.  For example, we take $\mu\colon 2\to 1$
to be the cospan
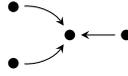
\[
\begin{tikzpicture}[x=.75cm, y=.75cm]
	\node[circle, minimum size=4pt, inner sep=0, fill, black] at (0,0) (L1) {};
	\node[circle, minimum size=4pt, inner sep=0, fill, black] at (0,1) (L2) {};
	\node[circle, minimum size=4pt, inner sep=0, fill, black] at (1,.5) (M) {};
	\node[circle, minimum size=4pt, inner sep=0, fill, black] at (2,.5) (N) {};
	\draw[function] (L1) to[bend right] (M);
	\draw[function] (L2) to[bend left] (M);
	\draw[function] (N) -- (M);
	\pgfresetboundingbox
	\useasboundingbox (L1.north) rectangle (L2.south-|N);
\end{tikzpicture}
\]
and $\epsilon\colon 1\to 0$ to be the cospan 
\[
\begin{tikzpicture}[x=.75cm, y=.75cm]
	\node[circle, minimum size=4pt, inner sep=0, fill, black] at (0,0) (L) {};
	\node[circle, minimum size=4pt, inner sep=0, fill, black] at (1,0) (M) {};
	\node[minimum size=10pt, inner sep=0, draw, dotted] at (2,0) (R) {};
	\draw[function] (L) to (M);
	\pgfresetboundingbox
	\useasboundingbox (L.north) rectangle (R.south);
\end{tikzpicture}
\]
where the dotted square represents the empty set.

One can then check that the nine equations in \cref{def.scfm} hold: in each case both sides of
the equation represent the unique cospan with apex $\ul{1}$. For example, the
associativity axiom says that the composite cospans $(\mu\otimes 1)\cp\mu$ and $(1\otimes\mu)\cp\mu$ are equal, namely to the cospan
\[
\begin{tikzpicture}[x=.75cm, y=.5cm]
	\node[circle, minimum size=4pt, inner sep=0, fill, black] at (0,0) (L1) {};
	\node[circle, minimum size=4pt, inner sep=0, fill, black] at (0,1) (L2) {};
	\node[circle, minimum size=4pt, inner sep=0, fill, black] at (0,2) (L3) {};
	\node[circle, minimum size=4pt, inner sep=0, fill, black] at (1,1) (M) {};
	\node[circle, minimum size=4pt, inner sep=0, fill, black] at (1,1) (M) {};
	\node[circle, minimum size=4pt, inner sep=0, fill, black] at (2,1) (N) {};
	\draw[function] (L1) to[bend right] (M);
	\draw[function] (L2) to (M);
	\draw[function] (L3) to[bend left] (M);
	\draw[function] (N) -- (M);
	\pgfresetboundingbox
	\useasboundingbox (L1.north) rectangle (L3.south-|N);
\end{tikzpicture}
\]
\end{example}

\begin{example}\label[example]{ex.cospanc_scFrob}
\cref{ex.cospan_scFrob} generalizes to any object in any category with finite
colimits.  Indeed, let $\cat{C}$ be a category with finite colimits. Write
$\Cat{Cospan}(\cat{C})$ for symmetric monoidal category with the objects of
$\cat{C}$ as its objects, isomorphism classes of cospans in $\cat{C}$ as its
morphisms, and coproduct $+$ as its monoidal product. Then each object $X$ of
$\Cat{Cospan}(\cat{C})$ has a canonical Frobenius structure, with Frobenius maps
exactly those cospans built from coproducts and copairings of $\id_X$. In
particular, each object $(m,x)$ of $\cospan$ has a canonical Frobenius
structure, which can be thought of as $m$ parallel copies of the Frobenius
structure on $1 \in \cospan[]$.
\end{example}

Next in \cref{ex.linrel_scfs_copy,ex.linrel_scfs_add} we give two different Frobenius structures on the same object. Let $(\linrel,\oplus)$ denote the symmetric monoidal category of finite-dimensional real
vector spaces $V$ and linear relations between them, with direct sum as the monoidal
product. Recall that a linear relation between 
$V$ and $W$---i.e.\ a morphism in $\linrel$---is a linear subspace $R\ss V\oplus
W$ of their direct sum. The composite of $R\ss V\oplus W$ and $S\ss W\oplus X$
is the relation
\begin{equation}\label{eqn.composite_relation}
  R\cp S\coloneqq\big\{(v,x)\in V\oplus X\mid\exists(w\in W)\ldotp (v,w)\in R\text{ and }(w,x)\in S.\big\}
\end{equation}
The identity morphism on $V$ is represented by the bare reflexive relation $\{(v,v')\mid v=v'\}$.

\begin{example}\label[example]{ex.linrel_scfs_copy}
We now define a Frobenius structure on the object $\rr\in\linrel$. Consider the relation $E\ss \rr\oplus\rr\oplus\rr$ given by $(a,b,c)\in E$ iff $a=b=c$. This is a linear relation because it is closed under addition and scalar multiplication; hence we can take $E$ to represent $\mu\colon(\rr\oplus\rr)\to\rr$. We can also take $E$ to represent $\delta\colon\rr\to(\rr\oplus\rr)$. We can take the maximal relation $\rr\ss\rr$ to represent $\eta\colon\rr^0\to\rr$ and $\epsilon\colon\rr\to\rr^0$.

It is easy to check that the nine equations required by \cref{def.scfm} are satisfied. For example, unitality says that the composite of $\{(a,(b,c))\mid c=a\}$ and $\{((b,c),d)\mid b=c=d\}$ should be the identity map $\{(a,d)\mid a=d\}$, and one checks that it is by working through \cref{eqn.composite_relation}.
\end{example}

\begin{example}\label[example]{ex.linrel_scfs_add}
Here we define a different Frobenius structure on the object $\rr\in\linrel$. Let $\mu$ be represented by the relation $\{(a,b,c)\mid a+b=c\}\ss\rr^3$, let $\eta$ be represented by the relation $\{0\}\ss\rr$. Similarly, let $\delta$ be represented by the relation $\{(a,b,c)\mid a=b+c\}$ and $\epsilon$ be represented by the relation $\{0\}\ss\rr$.

Again, it is easy to check that the equations required by \cref{def.scfm} are
satisfied. For example, the Frobenius law requires that for any
$a_1,a_2,b_1,b_2\in\rr$, the equation $a_1+a_2=b_1+b_2$ holds iff there exists
an $x\in\rr$ such that the equations $a_1=b_1+x$ and $x+a_2=b_2$ hold; this is
easily checked.
\end{example}


\section{Hypergraph categories} \label{sec.hypergraph_cats}
In a hypergraph category, every object has a chosen Frobenius structure, chosen compatibly with the monoidal structure.

\begin{definition}\label[definition]{def.hypergraph_cat_func}
  A \define{hypergraph category} is a symmetric monoidal category $(\cat{H},\otimes,I)$ in which each
  object $X$ is equipped with a Frobenius structure
  $(\mu_X,\eta_X,\delta_X,\epsilon_X)$, satisfying
\begin{equation}\label{eqn.hypergraph_def}
  \begin{aligned}
    \begin{tikzpicture}[spider diagram, font=\small, baseline=(bl),scale=.9]
	\node[spider={2}{1}] (a) {};
	\node[left=0 of a_in1] {$X\otimes Y$};
	\node[left=0 of a_in2] {$X\otimes Y$};
	\node[right=0 of a_out1] (llab) {$X\otimes Y$};
	\node[coordinate, right=1.5 of llab] (helper) {};
	\node[spider={2}{1}, above right=0.2 and 5 of a_in1] (b) {};
	\node[spider={2}{1}, below right=0.2 and 5 of a_in2] (c) {};
	\coordinate (x1) at (b_in1-|helper);
	\coordinate (x2) at (b_in2-|helper);
	\coordinate (x3) at (c_in1-|helper);
	\coordinate (x4) at (c_in2-|helper);
	\draw (b_in1) -- (x1);
	\draw (x3) to (b_in2);
	\draw  (x2) to (c_in1);
	\draw (c_in2) -- (x4);
	\node[left=0 of x1] {$X$};
	\node[left=0 of x2] {$Y$};
	\node[left=0 of x3] {$X$};
	\node[left=0 of x4] {$Y$};
	\node[right=0 of b_out1] {$X$};
	\node[right=0 of c_out1] {$Y$};
	\node at ($(llab)!.5!(helper)$) {$=$};
	\node[spider={0}{1}, right=5 of helper] (d) {};
	\node[right=0 of d_out1] (rlab) {$X\otimes Y$};
	\node[coordinate, right=1 of rlab] (help2) {};
	\node[spider={0}{1}, above=0.2 of help2] (e) {};
	\node[spider={0}{1}, below=0.2 of help2] (f) {};
	\node[right=0 of e_out1] {$X$};
	\node[right=0 of f_out1] {$Y$};
	\node at ($(rlab.east)!.35!(help2)$) {$=$};
	\node[spider={1}{2}, below=2.5 of a] (aa) {};
	\node[left=0 of aa_in1] {$X\otimes Y$};
	\node[right=0 of aa_out1] (h1) {$X\otimes Y$};
	\node[right=0 of aa_out2] (h2) {$X\otimes Y$};
	\coordinate (llab) at ($(h1)!.5!(h2)$) {};
	\node[spider={1}{2}, above right=0.2 and 3.5 of aa_out1] (b) {};
	\node[spider={1}{2}, below right=0.2 and 3.5 of aa_out2] (c) {};
	\coordinate (helper) at ($(b_out2)+(.5,-.25)$) {};
	\coordinate (x1) at (b_out1-|helper);
	\coordinate (x2) at (b_out2-|helper);
	\coordinate (x3) at (c_out1-|helper);
	\coordinate (x4) at (c_out2-|helper);
	\draw (b_out1) -- (x1);
	\draw (b_out2) to (x3);
	\draw (c_out1) to (x2);
	\draw (c_out2) -- (x4);
	\node[right=0 of x1] {$X$};
	\node[right=0 of x2] {$Y$};
	\node[right=0 of x3] {$X$};
	\node[right=0 of x4] {$Y$};
	\node[left=0 of b_in1] {$X$};
	\node[left=0 of c_in1] {$Y$};
	\coordinate (help2) at ($(b_in1)!.5!(c_in1)$);
	\node at ($(llab)!.5!(help2)$) {$=$};
	\node[spider={1}{0}, right=5 of helper] (d) {};
	\node[left=0 of d_in1] {$X\otimes Y$};
	\node[coordinate, right=1 of d] (help3) {};
	\node[spider={1}{0}, above right=0.25 and 1 of help3] (e) {};
	\node[spider={1}{0}, below right=0.25 and 1 of help3] (f) {};
	\node[left=0 of e_in1] (x) {$X$};
	\node[left=0 of f_in1] (y) {$Y$};
	\coordinate (rlab) at ($(x)!.5!(y)$);
	\node at ($(rlab)!.5!(d)$) {$=$};	
	\coordinate (bl) at ($(a)!.5!(aa)$);
\end{tikzpicture}
\end{aligned}
\end{equation}
as well as the \emph{unit coherence axiom}, that $\eta_I=\id_I=\epsilon_I$.

A functor $(F,\varphi)$ of hypergraph categories, or \define{hypergraph
functor}, is a strong symmetric monoidal functor $(F,\varphi)$ that preserves
the hypergraph structure. More precisely, the latter condition means that if the Frobenius structure on $X$ is $(\mu_X,\eta_X,\delta_X,\epsilon_X)$ then that on $FX$ must be 
\begin{equation}\label{eqn.hyp_functor}
  \left(\varphi_{X,X}\cp F\mu_X,\enspace
  \varphi_I \cp F\eta_X, \enspace
  F\delta_X\cp\varphi^{-1}_{X,X},\enspace 
  F\epsilon_X\cp \varphi_I^{-1}\right).
\end{equation}
We write $\hyp$ for the category with hypergraph categories as objects and
hypergraph functors as morphisms, and $\hhyp$ for the 2-category with, in
addition to these objects and morphisms, monoidal natural transformations as
2-morphisms.
\end{definition}

\begin{remark}
Note that every natural transformation between hypergraph functors is
invertible, i.e.\ a natural isomorphism. This follows from
\cref{prop.compactclosed} and the fact that every natural transformation between
compact closed categories is invertible.
\end{remark}

\begin{example}\label[example]{ex.cospanc}
Following \cref{ex.cospanc_scFrob}, the category of $\Cat{Cospan}(\cat{C})$ of
cospans in any category with finite colimits is canonically a hypergraph
category. Note in particular that $\cospan[\Lambda] =
\Cat{Cospan}(\finset_\Lambda)$ is a hypergraph category.
\end{example}

\begin{example}
  If we let $\cat{C}$ be the opposite of a category with finite limits in
  \cref{ex.cospanc}, we see that the category of $\Cat{Span}(\cat{C})$ of
  \emph{spans} in any category with finite limits is canonically a hypergraph
  category. In fact, if $\cat{C}$ is regular, then the category of relations in
  $\cat{C}$ is also a hypergraph category. In particular, the category $\rel$ of
  sets and binary relations can be equipped with the structure of a hypergraph
  category \cite{Fong:2016a}. 
\end{example}

\begin{example}
  The category of finite-dimensional vector spaces where each object each has a
  chosen basis is a hypergraph category \cite{kissinger2015finite}.
\end{example}

\begin{remark} 
  Note that the condition that the Frobenius structure on the monoidal unit be
  the structure $(\rho_I^{-1}, \id_I,\rho_I,\id_I)$ of \cref{ex.scFrob_on_I} has
  been omitted from some, but not all, previous definitions. We shall see in
  \cref{thm.stricthypergraphs} that the unit coherence axiom is crucial for the
  strictification of hypergraph categories, and hence for the equivalence with
  cospan-algebras.

  One reason that this additional unit coherence axiom may have been overlooked
  is that, in the strict case, this additional axiom does not alter the
  definition; we will prove this in \cref{prop.unique_Frob_on_I}. In
  \cref{ex.necessity_of_new_axiom} we will give an example which shows that the
  unit coherence axiom does not follow from the old ones; it really is a new
  addition.
\end{remark}

\begin{proposition}\label{prop.unique_Frob_on_I}
  Suppose $\cat{H}$ is a strict symmetric monoidal category in which each object
  is equipped with a Frobenius structure such that \cref{eqn.hypergraph_def} is
  satisfied. Then $\cat{H}$ is a hypergraph category.
\end{proposition}
\begin{proof}
  We must show that the Frobenius structure on the monoidal unit is
  $(\rho_I^{-1},\id_I,\rho_I,\id_I)$.

  First, note that in any strict monoidal category, we have $I=I \otimes I$ and
  $\rho_I =\id_I$, so the Frobenius structure constructed from the unitors on
  $I$, as detailed in \cref{ex.scFrob_on_I}, is simply equal to
  $(\id_I,\id_I,\id_I,\id_I)$. This is the unique Frobenius structure on $I$
  obeying the equations of \cref{eqn.hypergraph_def}. To see this, recall that
  by \cref{ex.frob_on_strict_I}, a Frobenius structure on $I$ simply amounts to
  a choice of invertible map $\mu\colon I \to I$.  The first equation of
  \cref{eqn.hypergraph_def} requires further that $\mu=\mu*\mu$. But the only
  monoid element that is both idempotent and invertible is the identity,
  $\id_I=\mu*\mu^{-1}=\mu*\mu*\mu^{-1}=\mu$. 
\end{proof}

\begin{example}\label{ex.necessity_of_new_axiom}
  To show that the unit coherence axiom indeed alters the definition of
  hypergraph category in general, here we provide a example of a (necessarily
  non-strict) symmetric monoidal category $\cat{X}$, equipped with a Frobenius
  structure on each object, that fails only this additional axiom.

  Let $(\cat{X},\oplus, I)$ be the symmetric monoidal category with two objects,
  $I$ and $O$, such that every homset is equal to $\{0,1\}$, such that $I\oplus
  I= O \oplus O= O$ and $I \oplus O=O \oplus I = I$, and such that composition
  and monoidal product of morphisms are all given by addition modulo 2. Note
  that the identity maps on $I$ and $O$ are both $0$. The coherence maps for the
  monoidal product are also given by the maps $0$; from this naturality and all
  coherence conditions are immediate.

  We may choose Frobenius structures $(1,1,1,1)$ on $I$ and $(0,0,0,0)$ on $O$.
  These structures obey the equations in \cref{eqn.hypergraph_def}, but do not
  obey the condition that the Frobenius structure on $I$ is
  $(\rho^{-1}_I,\id_I,\rho_I,\id_I)=(0,0,0,0)$.
\end{example}

\section{Critiques of hypergraph categories as structured categories}
\label{sec.critiques}

In this brief subsection we sketch two examples to promote the idea that hypergraph
categories should not be thought of as structured categories. First we show in
\cref{ex.hyp_doesnt_extend_equiv} that hypergraph structures do not extend along
equivalences of categories. Second we show in \cref{ex.hyp_ff_ess_not_equiv}
that an essentially surjective and fully faithful hypergraph functor may
fail to be a hypergraph equivalence. These critiques motivate the upcoming cospan-algebra perspective.

\begin{example}\label[example]{ex.hyp_doesnt_extend_equiv}
Here we will produce a category $\linrel_2$, an equivalence of categories
$F\colon\linrel_2\to\linrel$, and a hypergraph structure on $\linrel_2$ for
which there is no extension along $F$, i.e.\ there is no hypergraph structure on
$\linrel$ under which $F$ is a hypergraph functor. The idea is to let
$\linrel_2$ house two copies of $\linrel$ and to cause a problem by equipping
them with the two different Frobenius structures from
\cref{ex.linrel_scfs_copy,ex.linrel_scfs_add}.

Let $\linrel_2$ be the hypergraph category with two isomorphic copies of every
object in $\linrel$, but the same maps.%
\footnote{
The definition of $\linrel_2$ can be made more precise, once we have defined the
$(\io,\ff)$-factorization of hypergraph categories (\textsection\ref{sec.io_ff}). Namely,
consider the unique hypergraph functor $\cospan[\{1,2\}]\to\linrel$ sending
$1,2\mapsto\rr$, let
$\cospan[\{1,2\}]\To{\io}\cat{H}\To{\ff}\linrel$ be its $(\io, \ff)$-factorization, and let
$\linrel_2\coloneqq\cat{H}$.
}
By definition there is a functor $F\colon\linrel_2\to\linrel$, which is both
fully faithful and essentially surjective, so it is an equivalence. As we shall
see in detail in \cref{lemma.gen_hyp_struc}, we can put a hypergraph structure
on $\linrel_2$ by declaring a Frobenius structure on the two copies of $\rr$; we
use the two different such structures from
\cref{ex.linrel_scfs_copy,ex.linrel_scfs_add}. Then no Frobenius structure on
$\rr\in\linrel$ will satisfy \cref{eqn.hyp_functor} in
\cref{def.hypergraph_cat_func}.

\end{example}

\begin{example}\label[example]{ex.hyp_ff_ess_not_equiv}
To show that a fully faithful, essentially surjective hypergraph functor need
not be a hypergraph equivalence, we simply run \cref{ex.hyp_doesnt_extend_equiv}
the other way. Namely, let $\linrel_2$ be the hypergraph category constructed in
\cref{ex.hyp_doesnt_extend_equiv}, give $\linrel$ the additive hypergraph
structure from \cref{ex.linrel_scfs_add}, and consider the hypergraph functor $\linrel\to\linrel_2$ sending
the generator to the appropriate generator. This is essentially surjective and
fully faithful, but it is not an equivalence of hypergraph categories because,
as we saw in \cref{ex.hyp_doesnt_extend_equiv}, there is no hypergraph functor
to serve as its inverse.
\end{example}

The critique leveled by
\cref{ex.hyp_doesnt_extend_equiv,ex.hyp_ff_ess_not_equiv} is important, because
it says that in an important sense hypergraph categories do not behave like
structured categories. This critique dissolves---i.e.\ the above problems become
impossible to state---when we treat hypergraph categories as cospan-algebras.

Thus thinking of hypergraph categories as cospan-algebras has distinct
advantages. However, it also comes with a couple of costs. The first is that
cospan-algebras do not take into account 2-morphisms, i.e.\ the natural
transformations between hypergraph functors; the question of whether and/or how
this can be rectified, and indeed if it needs to rectified, remains open. The
second cost is that cospan-algebras correspond to hypergraph categories that are
objectwise-free ($\of$). Luckily, this second issue is not very important: in
\cref{sec.strictification}, we will show that  every hypergraph category is
naturally equivalent to one that is $\of$.

\section{A word on operads}\label{sec.word_on_operads}

In the introduction, we spoke of operads. Operads are generalizations of categories in which each morphism has a finite number of inputs and one output, e.g.\ $\varphi\colon x_1,\ldots,x_n\to y$. In the context of this paper, operads govern the structure of wiring diagrams like the ones in \cref{eqn.subst}, and one should imagine the $x$'s as the interior cells or boxes and the $y$ as the exterior cell or box of a wiring diagram $\varphi$.

The reader who is unfamiliar with operads need not worry: the only operads we use are those that underlie symmetric monoidal categories $\cat{M}$, where operad morphisms $\varphi$ as above come from morphisms $\varphi\colon(x_1\otimes\cdots\otimes x_n)\to y$ in $\cat{M}$.

In fact, throughout this paper we work exclusively in the monoidal setting, so
operads will disappear from the discussion. There are two reasons we bring up operads at all. First, they are a bit more general, so further work in this area will sometimes require one to use operads rather than monoidal categories. More relevant, however, is the fact that the wiring diagram pictures we want to draw more naturally fit with operads. For example, here we draw the ``same morphism'' in two ways: operadic style $\phi\colon f,g,h\to i$ on the left and monoidal style $\phi\colon f\otimes g\otimes h\to i$ on the right:
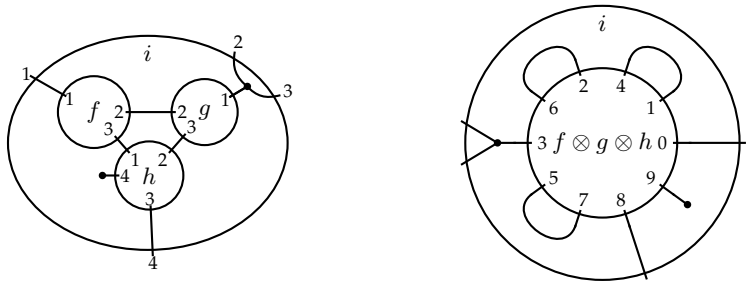
\[
\begin{tikzpicture}[unoriented WD, pack size=25pt, pack inside color=white, pack outside color=black, link size=2pt, font=\footnotesize, spacing=30pt, baseline=(Hh.north)]
	\node[pack] at (0,0) (Hf) {$f$};
	\node[pack, right=.5 of Hf] (Hg) {$g$};
	\node[pack] at ($(Hf)!.5!(Hg)+(0,-.8)$) (Hh) {$h$};
	\node[outer pack, inner xsep=3pt, inner ysep=3pt, fit=(Hf) (Hg) (Hh)] (H) {};
	\node[link] at ($(Hg.30)!.5!(H.30)$) (link1) {};
	\node[link,  left=.1 of Hh.west] (link2) {};
	\draw (Hg) -- (link1);
	\node[below=0 of H.north] {$i$};
\begin{scope}[font=\tiny]
	\draw[shorten >= -2pt] (Hf) to node[pos=-.3] {1} node[pos=1.3] {1} (H);
	\draw[shorten >= -2pt] (Hh) to node[pos=-.3] {3} node[pos=1.3] {4} (H);
	\draw (Hf) to node[pos=-1.2] {3} node[pos=2.2] {1} (Hh);
	\draw (Hf) to node[pos=-.3] {2} node[pos=1.3] {2} (Hg);
	\draw (Hg) to node[pos=-1] {3} node[pos=2] {2} (Hh);
	\draw (Hg) to node[pos=-.8] {1} (link1);
	\draw (link2) to node[pos=2.4] {4} (Hh);
	\draw[shorten >= -2pt] (link1) to[bend right] node [pos=1.4] {3} (H.20);
	\draw[shorten >= -2pt] (link1) to[bend left] node [pos=1.4] {2} (H.45);
\end{scope}
	\node[pack, right=3 of H.east, inner sep=7pt] (fgh) {$f\otimes g\otimes h$};
	\foreach \i [evaluate=\i as \ii using {\i*36}] in {0,...,9} {
		\draw (fgh.\ii) to 
			node[coordinate, pos=-2] (N\i) {}
			node[coordinate, pos=2] (o\i) {}
			+(\ii:3pt);		
	}
\begin{scope}[font=\tiny]
	\node at (N0) {0};
	\node at (N1) {1};
	\node at (N2) {4};
	\node at (N3) {2};
	\node at (N4) {6};
	\node at (N5) {3};
	\node at (N6) {5};
	\node at (N7) {7};
	\node at (N8) {8};
	\node at (N9) {9};
\end{scope}
	\node[link] at ($(o5)+(180:5pt)$) (l1) {};
	\node[link] at ($(o9)+(-36:5pt)$) (l2) {};
	\node[outer pack, fit=(fgh)] (outer) {};
	\node[below=0 of outer.north] {$i$};
	\draw (o0) to (outer.0);
	\draw (o5) to (l1);
	\draw (l1) to (outer.172);
	\draw (l1) to (outer.188);
	\draw (o8) to (outer.-72);
	\draw (o9) to (l2);
	\draw (o1) to[out=36*1, in=36*2, looseness=1.5] (o2);
	\draw (o3) to[out=36*3, in=36*4, looseness=1.5] (o4);
	\draw (o6) to[out=36*6, in=36*7, looseness=1.5] (o7);
	\pgfresetboundingbox
	\useasboundingbox (H.west|-outer.north) rectangle (outer.east|-outer.south);
\end{tikzpicture}
\]
Although monoidal-style wiring diagrams are often more difficult to visually
parse than operad-style, the symbolic notation for monoidal categories is often
easier to parse than that of operads. Thus the only place operads will appear
from now on is in visualizing wiring diagrams.

\chapter{Properties of hypergraph categories} \label{sec.properties}

In this section, we discuss some basic properties of hypergraph categories. In \cref{sec.sdccc}, we show that they are self-dual compact closed. In \cref{sec.cospan_is_free}, we show that $\cospan$ is both the free hypergraph category and the free $\of$-hypergraph category on a set $\Lambda$. In \cref{sec.io_ff} we show that any hypergraph functor can be factored as an identity-on-objects ($\io$) hypergraph functor followed by a fully faithful ($\ff$) hypergraph functor, and use this factorization to construct a Grothendieck fibration $U\colon\hyp_\of\to\smset_\List$. Finally in \cref{sec.strictification} we prove that every hypergraph category can be strictified to an equivalent hypergraph category that is objectwise-free.

\section{Hypergraph categories are self-dual compact closed.}\label{sec.sdccc}

A compact closed category is a symmetric monoidal category $(\cat{C},\otimes,I)$
such that every object $X$ is dualizable---i.e.\ there exists an object $X^*$ and
morphisms $\cupp_X\colon I\to X\otimes X^*$, depicted
$\begin{tikzpicture}[oriented WD, bbx=.8em, bby=.8em]
	\coordinate (top);
	\coordinate [below=1 of top] (bot);
  	\draw (top) to[in=180,out=180] (bot);
	\draw (top) to ($(top)+(.7,0)$);
	\draw (bot) to ($(bot)+(.7,0)$);
\end{tikzpicture}\,
$,
and $\capp_X\colon X^*\otimes X\to I$, depicted
$\begin{tikzpicture}[oriented WD, bbx=.8em, bby=.8em]
	\coordinate (top);
	\coordinate [below=1 of top] (bot);
  	\draw (top) to[in=0,out=0] (bot);
	\draw (top) to ($(top)-(.7,0)$);
	\draw (bot) to ($(bot)-(.7,0)$);
\end{tikzpicture}\,
$,
which satisfy the zigzag identities\footnote{These are often also called the
triangle identities, as one can think of $X^*$ as a left adjoint to $X$ (see
\cite{day1977note}), or the snake identities (see \cref{eqn.triangle_ids}).}:
\begin{equation}\label{eqn.triangle_ids}
\begin{aligned}
\begin{tikzpicture}[oriented WD, bbx=1em, bby=1em]
	\coordinate (topleft);
	\coordinate [below right = 2 and 3 of topleft] (botright);
	\coordinate [right=.5 of topleft] (col1);
	\coordinate [left=.5 of botright] (col2);
	\coordinate (midline) at ($(topleft)!.5!(botright)$);
	\coordinate (captop) at (topleft-|col2);
	\coordinate (capbot) at (midline-|col2);
	\coordinate (cuptop) at (midline-|col1);
	\coordinate (cupbot) at (botright-|col1);
  	\draw (topleft) to (captop)
			to[in=0,out=0] (capbot);
	\draw (cuptop) to (capbot);
	\draw (cuptop)	to[in=180,out=180] (cupbot)
			to (botright);
	\pgfresetboundingbox
	\useasboundingbox (topleft) rectangle (botright);
\end{tikzpicture}
\end{aligned}
=
\begin{aligned}
\begin{tikzpicture}[oriented WD, bbx=1em, bby=1em]
	\coordinate (left);
	\coordinate [right =3 of left] (right);
	\coordinate [above =.5 of left] (topleft);
	\coordinate [below =.5 of right] (botright);
  	\draw (left) to (right);
	\pgfresetboundingbox
	\useasboundingbox (topleft) rectangle (botright);
\end{tikzpicture}
\end{aligned}
\hspace{3cm}
\begin{aligned}
\begin{tikzpicture}[oriented WD, bbx=1em, bby=1em]
	\coordinate (botleft);
	\coordinate [above right = 2 and 3 of botleft] (topright);
	\coordinate [right=.5 of botleft] (col1);
	\coordinate [left=.5 of topright] (col2);
	\coordinate (midline) at ($(botleft)!.5!(topright)$);
	\coordinate (capbot) at (botleft-|col2);
	\coordinate (captop) at (midline-|col2);
	\coordinate (cupbot) at (midline-|col1);
	\coordinate (cuptop) at (topright-|col1);
  	\draw (botleft) to (capbot)
			to[in=0,out=0] (captop);
	\draw (cupbot) to (captop);
	\draw (cupbot)	to[in=180,out=180] (cuptop)
			to (topright);
	\pgfresetboundingbox
	\useasboundingbox (botleft) rectangle (topright);
\end{tikzpicture}
\end{aligned}
=
\begin{aligned}
\begin{tikzpicture}[oriented WD, bbx=1em, bby=1em]
	\coordinate (left);
	\coordinate [right =3 of left] (right);
	\coordinate [above =.5 of left] (topleft);
	\coordinate [below =.5 of right] (botright);
  	\draw (left) to (right);
	\pgfresetboundingbox
	\useasboundingbox (topleft) rectangle (botright);
\end{tikzpicture}
\end{aligned}
\end{equation}
This notion generalizes duals in finite-dimensional vector spaces. A compact closed category is called \emph{self-dual} if every object serves as its own dual, $X^*\coloneqq X$; the category of finite-dimensional based vector spaces (where each vector space is equipped with basis) is self-dual compact closed.

A basic property of hypergraph categories is that they are self-dual compact
closed. Indeed, a self-duality for each object can be constructed using the
Frobenius maps for each object. 

\begin{proposition} \label[proposition]{prop.compactclosed}
Every hypergraph category $\cat{H}$ is self-dual compact
closed. Moreover, each object $X$ is equipped with a canonical self-duality
defined by $\cupp_X\coloneqq\eta_X\cp\delta_X \colon I \to X \otimes X$ and
$\capp_X\coloneqq\mu_X \cp \epsilon_X \colon X \otimes X \to I$.
\end{proposition}
\begin{proof}
This result is well known; see for example \cite{rosebrugh2005generic}. The
proof is straightforward: the zigzag identities \eqref{eqn.triangle_ids} are an
immediate consequence of the Frobenius and co/unitality axioms. For example:
\begin{align*}
\begin{aligned}
\begin{tikzpicture}[spider diagram]
	\node[spider={0}{1}] (a) {};
	\node[spider={1}{2}, right=.5 of a] (b) {};
	\node[spider={2}{1}, below right=\leglen*2/3 and 1 of b] (c) {};
	\node[spider={1}{0}, right=.5 of c] (d) {};
	\coordinate[left=2 of c_in2] (in);
	\coordinate[right=2 of b_out1] (out);
	\draw (in) -- (c_in2);
	\draw (a_out1) -- (b_in1);
	\draw (b_out1) -- (out);
	\draw (b_out2) -- (c_in1);
	\draw (c_out1) -- (d_in1);
\end{tikzpicture}
\end{aligned}
=
\begin{aligned}
\begin{tikzpicture}[spider diagram]
	\node[special spider={2}{1}{\leglen}{0}] (b) {};
	\node[spider={0}{1}, left=.5 of b_in1] (a) {};
	\node[special spider={1}{2}{0}{\leglen}, right=.5 of b] (c) {};
	\node[spider={1}{0}, right=.5 of c_out2] (d) {};
	\coordinate[left=1 of b_in2] (in);
	\coordinate[right=1 of c_out1] (out);
	\draw (a_out1) -- (b_in1);
	\draw (b_out1) -- (c_in1);
	\draw (in) -- (b_in2);
	\draw (c_out1) -- (out);
	\draw (c_out2) -- (d_in1);
\end{tikzpicture}
\end{aligned}
=
\begin{aligned}
\begin{tikzpicture}
	\coordinate (in);
	\coordinate[right =2 of in] (out);
	\node[above=.2 of in] (x) {};
	\node[below=.2 of in] (y) {};
	\draw (in) -- (out);
\end{tikzpicture}
\end{aligned}
\end{align*}
\end{proof}

This means that in any hypergraph category, we have a bijection between
morphisms $X \to Y$, and morphisms $I \to X\otimes Y$.
\begin{proposition}\label[proposition]{prop.gather_and_parse}
For any two objects $X, Y$ in a self-dual compact closed category $\cat{C}$,
there is a bijection $\cat{C}(X,Y)\cong\cat{C}(I,X\otimes Y)$.
\end{proposition}
\begin{proof}
For any $f\colon X\to Y$, and any $g\colon I\to X\otimes Y$ define 
\[
\begin{tikzpicture}[oriented WD, bbx=1em, bby=1ex, bb min width=1.5em]
	\node[bb={1}{1}] (f) {$f$};
	\node[bb={0}{0},color=white, fit={(f.south west) ($(f.north east)+(.5,2.5)$)}] (fhat) {};
	\coordinate (fhat_out2) at (f_out1-|fhat.east);
	\coordinate (fhat_out1) at ($(fhat_out2)!.55!(fhat.north east)$);
	\begin{scope}[shorten >=-2pt]
  	\draw (f_out1) to (fhat_out2) node[right, font=\scriptsize] {$Y$};
  	\draw let \p1=(f_in1|-fhat_out1) in
  		(f_in1) to[in=180, out=180] (\x1,\y1) to (fhat_out1) node[right, font=\scriptsize] {$X$};
	\end{scope}
	\node[below=0.7 of f] {$\cupp_X\cp(\id_X\otimes f)$};
	\node[left=0 of fhat.west] {$\gathr{f}\coloneqq$};
	\node[bb port sep=1.2, bb={0}{2}, right=14 of f] (g) {$g$};
	\node[bb={0}{0},color=white, fit={(g.south west) ($(g.north east)+(1,2.5)$)}] (ghat) {};	
	\coordinate (ghat_out2) at (g_out2-|ghat.east);
	\coordinate (ghat_in2fake) at (ghat_out2-|ghat.west);
	\coordinate (ghat_in1) at ($(ghat_in2fake)!.55!(ghat.north west)$);	
	\begin{scope}[shorten >=-2pt]
  	\draw (g_out2) to (ghat_out2) node[right, font=\scriptsize] {$Y$} ;
  	\draw let \p1=(g_out1|-ghat_in1) in
		(g_out1) to[in=0, out=0]
		(\x1,\y1) to[in=180,out=180] (ghat_in1) node[left, font=\scriptsize] {$X$};
	\end{scope}
	\node[below=0.7 of g] {$(\id_X\otimes g)\cp(\capp_X\otimes\id_Y)$};
	\node[left=1 of ghat.west] {$\parse{g}\coloneqq$};
\end{tikzpicture}
\]
It is easy to prove that $\gathr{\cdot}$ and $\parse{\cdot}$ are mutually inverse. 
\end{proof}


We will refer to $\gathr{f}$ as the \emph{name} of $f$. This notion will be
critical for the equivalence between hypergraph categories and cospan-algebras:
given a hypergraph category $\cat{H}$, the corresponding cospan-algebra
$A_\cat{H}$ will record the \emph{names} of the morphisms $\cat{H}$, rather than
the morphisms themselves. But note that homsets $\cat{H}(X,Y)$ are indexed by
\emph{two} objects, $X$ and $Y$, while $A_\cat{H}(X)$ just depends on one. It is
the self-dual compact closed structure that allows us to switch between these
two viewpoints.


For any three objects $X,Y,Z$ in a self-dual compact closed category, we may
define a morphism $\comp^Y_{X,Z} \coloneqq \id_X \otimes \capp_Y \otimes \id_Z\colon
X \otimes Y \otimes Y \otimes Z \longrightarrow X \otimes Z$:
\begin{equation}
\comp^Y_{X,Z} \coloneqq 
\begin{aligned}
\begin{tikzpicture}[oriented WD, bbx=1em, bby=1em, font=\scriptsize]
	\coordinate (x1) at (0,3);
	\coordinate (x2) at (2,3);
	\coordinate (y1) at (0,2);
	\coordinate (y2) at (0,1);
	\coordinate (z1) at (0,0);
	\coordinate (z2) at (2,0);
	\coordinate (helper) at (1.2,1.5);
  	\draw (x1) to (x2);
  	\draw (y1) to ($(y1)+(.7,0)$) 
		   to[in=0, out=0] ($(y2)+(.7,0)$)
		   to (y2);
  	\draw (z1) to (z2);
	\node[left=0 of x1] {$X$};
	\node[left=0 of y1] {$Y$};
	\node[left=0 of y2] {$Y$};
	\node[left=0 of z1] {$Z$};
\end{tikzpicture}
\end{aligned}
\label{eq.comp}
\end{equation}
Below in \cref{prop.comp_works,prop.remember_the_name}, we show that the morphism $\comp$ acts like
composition on names and that $\comp$
can be used to recover a morphism from its name; both propositions are immediate from the zigzag identities.

\begin{proposition}\label{prop.comp_works}
For any morphisms $f\colon X \to Y$, $g \colon Y \to Z$ in a self-dual compact
closed category, we have $ (\gathr{f}\otimes \gathr{g}) \cp \comp^Y_{X,Z} =
\gathr{f\cp g}$:
\[
\begin{tikzpicture}[oriented WD, bbx=1em, bby=1ex, bb min width=1.5em]
	\node[bb={1}{1}] (f) {$f$};
	\node[bb={1}{1}, below=2.5 of f] (g) {$g$};
	\node[rounded corners=\bbcorners, thick, fit={(g.south west) ($(f.north east)+(.5,2.5)$)}] (fg1) {};
	\coordinate (fg1_out2) at (f_out1-|fg1.east);
	\coordinate (fg1_out1) at ($(fhat_out2)!.55!(fhat.north east)$);
	\coordinate (fg1_out4) at (g_out1-|fg1.east);
	\coordinate (fg1_out3) at ($(fg1_out4)!.5!(fg1_out2)$);
	\coordinate (helper) at (f_out1|-fg1_out3);
	\coordinate (topcup1) at (f_in1|-fg1_out1);
	\begin{scope}[shorten >=-2pt]
  	\draw (f_in1) to[in=180, out=180] (topcup1) to (fg1_out1);
  	\draw (g_out1) to (fg1_out4);
	\end{scope}
  	\draw let \p1=(g_in1|-fg1_out3) in
  		(g_in1) to[in=180, out=180] (\x1,\y1) -- (helper) to[in=0,out=0] (f_out1);
	\node[bb={1}{1}, right=5 of helper] (fg) {$f\cp g$};
	\node[bb={0}{0}, draw=none, fit={(fg.south west) ($(fg.north east)+(.5,2.5)$)}] (fghat) {};
	\coordinate (fghat_out2) at (fg_out1-|fghat.east);
	\coordinate (fghat_out1) at ($(fghat_out2)!.55!(fghat.north east)$);
	\begin{scope}[shorten >=-2pt]
  	\draw (fg_out1) to (fghat_out2);
  	\draw let \p1=(fg_in1|-fghat_out1) in
  		(fg_in1) to[in=180, out=180] (\x1,\y1) to (fghat_out1);
	\end{scope}
	\node at ($(fg1.east)!.5!(fghat.west)$) {$=$};
	\pgfresetboundingbox
	\useasboundingbox (topcup1) rectangle (fg1_out4-|fghat_out2);
\end{tikzpicture}
\]
\end{proposition}

\smallskip

\begin{proposition} \label[proposition]{prop.remember_the_name}
For any morphism $f\colon X\to Y$ in a self-dual compact closed category, we have $(\id_X\concat\gathr{f})\cp\comp^X_{\emptylist,Y}=f$:
\[
\begin{tikzpicture}[oriented WD, bbx=1em, bby=1ex, bb min width=1.5em]
	\node (f) {};
	\node[bb={1}{1}, below=2.5 of f] (g) {$f$};
	\node[bb={0}{0}, draw=none, fit={(g.south west) ($(f.north east)+(.5,2.5)$)}] (fg1) {};
	\coordinate (fg1_out1) at ($(fhat_out2)!.55!(fhat.north east)$);
	\coordinate (fg1_out2) at (f-|fg1.east);
	\coordinate (fg1_out4) at (g_out1-|fg1.east);
	\coordinate (fg1_out3) at ($(fg1_out4)!.55!(fg1_out2)$);
	\coordinate (helper) at (g.east|-fg1_out3);
	\coordinate (fg1_in1) at (f-|fg1.west);
	\begin{scope}[shorten >=-2pt]
  	\draw (g_out1) to (fg1_out4);
	\end{scope}
  	\draw let \p1=(g_in1|-fg1_out3) in
  		(g_in1) to[in=180, out=180] (\x1,\y1) -- (helper) to[in=0,out=0]
		(f-|g.east) --(fg1_in1);
	\node[bb={1}{1}, right=5 of helper] (fg) {$f$};
	\node[bb={0}{0}, color=white, fit={($(fg.south west)-(.5,2.5)$) ($(fg.north east)+(.5,2.5)$)}] (fghat) {};
	\coordinate (fghat_in1) at (fg_in1-|fghat.west);
	\coordinate (fghat_out1) at (fg_out1-|fghat.east);
	\draw (fghat_in1) to (fg_in1);
	\draw (fghat_out1) to (fg_out1);
	\node at ($(fg1.east)!.5!(fghat.west)$) {$=$};
	\pgfresetboundingbox
	\useasboundingbox (fg1_in1) rectangle (fg1_out4-|fghat_out1);
\end{tikzpicture}
\]
\end{proposition}

\smallskip

\begin{example} \label{ex.comp}
In $\cospan$ the morphism $\comp^Y_{X,Z}$ is given by the cospan
\begin{equation*}
\begin{aligned}
\begin{tikzpicture}[unoriented WD, pack size=25pt, pack inside color=white, pack
outside color=black, link size=2pt, font=\footnotesize, spacing=30pt]
	\node[pack] at (0,0) (f) {$f$};
	\node[pack, right=.6 of f] (g) {$g$};
	\node[outer pack, inner xsep=3pt, inner ysep=3pt, fit=(f) (g) ] (fg) {};
	\node[link,  left=.2 of f.west, label={[above=0, font=\scriptsize]:{$X$}}] (link1) {};
	\node[link, label={[above=0, font=\scriptsize]:{$Y$}}] at ($(f)!.5!(g)$) (link2) {};
	\node[link,  right=.2 of g.east, label={[above=0, font=\scriptsize]:{$Z$}}] (link3) {};
	\draw (fg.west) -- (link1);
	\draw (f.west) -- (link1);
	\draw (f.east) -- (link2);
	\draw (g.west) -- (link2);
	\draw (g.east) -- (link3);
	\draw (fg.east) -- (link3);
\end{tikzpicture}
\hspace{5em}
\begin{tikzpicture}[xscale=.4, yscale=.5]
	\node[link, label={[left=0, font=\scriptsize]:{$X$}}] at (0,3) (x1) {};
	\node[link, label={[left=0, font=\scriptsize]:{$Y$}}] at (0,2) (x2) {};
	\node[link, label={[left=0, font=\scriptsize]:{$Y$}}] at (0,1) (x3) {};
	\node[link, label={[left=0, font=\scriptsize]:{$Z$}}] at (0,0) (x4) {};
	\node[link, label={[above=0, font=\scriptsize]:{$X$}}] at (3,3) (n1) {};
	\node[link, label={[above=0, font=\scriptsize]:{$Y$}}] at (3,1.5) (n2) {};
	\node[link, label={[above=0, font=\scriptsize]:{$Z$}}] at (3,0) (n3) {};
	\node[link, label={[right=0, font=\scriptsize]:{$X$}}] at (6,3) (y1) {};
	\node[link, label={[right=0, font=\scriptsize]:{$Z$}}] at (6,0) (y2) {};
\begin{scope}[font=\tiny, decoration=brace]
	\draw[decorate] ($(x2)+(-1.3,-.3)$) to node[left] {$f$} ($(x1)+(-1.3,.3)$);
	\draw[decorate] ($(x4)+(-1.3,-.3)$) to node[left] {$g$} ($(x3)+(-1.3,.3)$);
\end{scope}
\begin{scope}[function]
	\draw (x1) to (n1);
	\draw (x2) to (n2);
	\draw (x3) to (n2);
	\draw (x4) to (n3);
	\draw (y1) to (n1);
	\draw (y2) to (n3);
\end{scope}
\end{tikzpicture}
\end{aligned}
\end{equation*}
\[X\concat Y \concat Y \concat Z \To{\id_X\concat\copair{\id_Y,\id_Y} \concat \id_Z} X \concat
Y\concat Z \xleftarrow{\id_X \concat!_Y\concat \id_Z} X \concat Z.\]

\end{example}

\cref{ex.comp} will be useful later, when we see that not only are the Frobenius structures of hypergraph categories controlled by cospans, but so are the identities and the composition law! 

\section{Free hypergraph categories} \label{sec.free_cats}

In this section we show that $\cospan$ is both the free hypergraph category and
the free $\of$-hypergraph category on a set $\Lambda$. We first discuss the
relationship between $\cospan[]$ and Frobenius monoids.

\subsection{Cospans and Frobenius monoids }
\cref{ex.cospan_scFrob}, where we define a certain Frobenius structure on the object $1$ in $\cospan[]$, is
central to the interplay between cospan-algebras and hypergraph categories. This
is because $\cospan[]$ is free special commutative Frobenius monoid on one
generator.

Write $\sigma\colon 2 \to 2$ for the cospan
\[
\begin{tikzpicture}[x=1cm, y=.75cm]
	\node[circle, minimum size=4pt, inner sep=0, fill, black] at (0,0) (L1) {};
	\node[circle, minimum size=4pt, inner sep=0, fill, black] at (0,1) (L2) {};
	\node[circle, minimum size=4pt, inner sep=0, fill, black] at (1,0) (M1) {};
	\node[circle, minimum size=4pt, inner sep=0, fill, black] at (1,1) (M2) {};
	\node[circle, minimum size=4pt, inner sep=0, fill, black] at (2,0) (R1) {};
	\node[circle, minimum size=4pt, inner sep=0, fill, black] at (2,1) (R2) {};
	\draw[function] (L1) -- (M2);
	\draw[function] (L2) -- (M1);
	\draw[function] (R1) -- (M1);
	\draw[function] (R2) -- (M2);
	\pgfresetboundingbox
	\useasboundingbox (L1.north) rectangle (R2.south);
\end{tikzpicture}
\]
and $\id\colon 1 \to 1$ for the cospan
\begin{equation}\label{eqn.identity_cospan}
\begin{tikzpicture}[x=1cm, y=.75cm]
	\node[circle, minimum size=4pt, inner sep=0, fill, black] at (0,0) (L1) {};
	\node[circle, minimum size=4pt, inner sep=0, fill, black] at (1,0) (M1) {};
	\node[circle, minimum size=4pt, inner sep=0, fill, black] at (2,0) (R1) {};
	\draw[function] (L1) -- (M1);
	\draw[function] (R1) -- (M1);
	\pgfresetboundingbox
	\useasboundingbox (L1.north) rectangle (R1.south);
\end{tikzpicture}
\end{equation}
\begin{lemma} \label[lemma]{lem.cospan_generators}
The category $(\cospan[],\concat)$ is generated, as a symmetric monoidal category, by the
morphisms $\mu$, $\eta$, $\delta$, $\epsilon$. That is, given any cospan
$c=(m \to p \from n)$, we may write down an expression that is equal to $c$, using only the cospans $\mu$, $\eta$, $\delta$, $\epsilon$; the cospans $\sigma$, $\id$; composition;
and monoidal product.
\end{lemma}
\begin{proof}
Note that any function can be factored as a permutation, followed by an
order-preserving surjection, followed by an order-preserving injection. Applying
this to each leg of a cospan $c= (\ul{m} \to \ul{p}\from\ul{n})$ gives a factorization
\begin{equation}\label{eqn.factor_cospan}
\begin{tikzcd}
	\ul{m}\ar[r, "\cong"]&
	\ul{m}\ar[r, two heads, "\text{ord.}"] &
	\ul{m'}\ar[r, tail, "\text{ord.}"] &
	\ul{p} &
	\ul{n'}\ar[l, tail, "\text{ord.}"'] &
	\ul{n}\ar[l, two heads, "\text{ord.}"'] &
	\ul{n}.\ar[l, "\cong"']
\end{tikzcd}
\end{equation}
Each of these six functions may be viewed as a cospan, with the other leg
supplied by the identity map on the codomain of the function. Since the pushout
of a morphism $a$ along an identity map is just $a$ again, this gives a
factorization of $c$ into the composite of six cospans. It remains to show each
of these cospans can be built from the required generators.

The key idea is that the permutations may be constructed by (composites of) braidings
$\sigma$, the order-preserving surjection $\ul{m} \surj \ul{m'}$ by multiplications,
$\ul{m'} \inj \ul{p}$ by units, $\ul{n'} \leftarrowtail \ul{n}$ by counits, and
$\ul{p} \twoheadleftarrow \ul{n'}$ by comultiplications. 

To elaborate, for the permutations, observe that any transposition of adjacent
elements may be constructed as the monoidal product of $\sigma$ with some number
of identities $\id$. It is then well known from group theory that transpositions
of adjacent elements generate all permutations. For order-preserving surjections
$s\colon \ul{m} \surj \ul{m'}$, observe that this $s$ is the monoidal product of
surjections $s_i\colon s^{-1}(i) \surj 1$ for each $i \in \ul{m'}$. Thus without
loss of generality we need only consider surjections $\ul{m} \surj \ul{1}$. Any
such surjection may be written as $(\id \concat \cdots \concat \id\concat \mu)\cp
\cdots \cp (\id \concat \mu) \cp \mu$, where this is the composite of $m-1$
terms, each of which is a sum of $m-2$ identities and $\mu$.  For order-preserving injections $j\colon{m'} \inj \ul{p}$, we may take
the monoidal product $j_1\concat \cdots \concat j_{p}$, where $j_i = \eta$ if
$j^{-1}(i) = \varnothing$, and $j_i = \id$ otherwise.  The remaining cases are
analogous.
\end{proof}

The above argument is perhaps clearest through a detailed example.
\begin{example}
We now build the cospan \cref{eqn.cospan_pic}---shown to the left below---from the generators in \cref{lem.cospan_generators}.
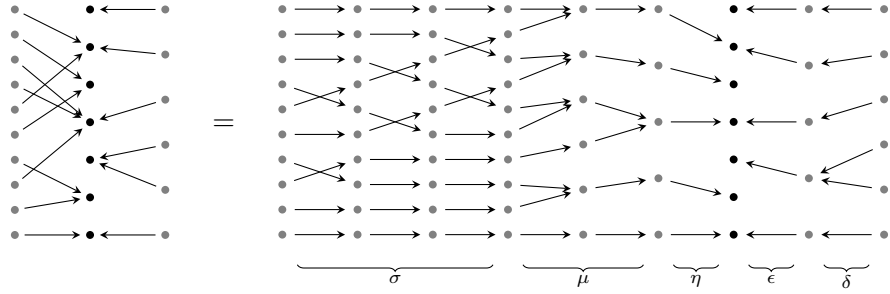
\[
\begin{aligned}
\begin{tikzpicture}
\begin{scope}[every node/.append style={circle, minimum size=3pt, inner sep=0,
fill, gray}]
	\foreach \i in {0,...,9} {
		\node at (0,-\i/3+1) (L\i) {};
	}
	\foreach \i in {0,...,6} {
		\node[black] at (1,-\i/2+1) (M\i) {};
	}
	\foreach \i in {0,...,5} {
		\node at (2,-3*\i/5+1) (N\i) {};
	}
\end{scope}
\node at (-.2,-2.7) {};
\begin{scope}[shorten <= 2pt, shorten >= 2pt, ->,>=stealth]
	\draw (L0) to (M1);
	\draw (L1) to (M2);
	\draw (L2) to (M3);
	\draw (L3) to (M3);
	\draw (L4) to (M1);
	\draw (L5) to (M2);
	\draw (L6) to (M5);
	\draw (L7) to (M3);
	\draw (L8) to (M5);
	\draw (L9) to (M6);
	\draw (N0) to (M0);
	\draw (N1) to (M1);
	\draw (N2) to (M3);
	\draw (N3) to (M4);
	\draw (N4) to (M4);
	\draw (N5) to (M6);
\end{scope}
\end{tikzpicture}
\end{aligned}
\quad
\begin{aligned}
\begin{tikzpicture}
\node at (0,1) {};
\node at (0,-.5) {$=$};
\node at (0,-2.7) {};
\end{tikzpicture}
\end{aligned}
\quad
\begin{aligned}
\begin{tikzpicture}
\begin{scope}[every node/.append style={circle, minimum size=3pt, inner sep=0,
fill, gray}]
	\foreach \i in {0,...,9} {
		\node at (0,-\i/3+1) (a\i) {};
	}
	\foreach \i in {0,...,9} {
		\node at (1,-\i/3+1) (b\i) {};
	}
	\foreach \i in {0,...,9} {
		\node at (2,-\i/3+1) (c\i) {};
	}
	\foreach \i in {0,...,9} {
		\node at (3,-\i/3+1) (d\i) {};
	}
	\foreach \i in {0,...,5} {
		\node at (4,-3*\i/5+1) (e\i) {};
	}
	\foreach \i in {0,...,4} {
		\node at (5,-3*\i/4+1) (f\i) {};
	}
	\foreach \i in {0,...,6} {
		\node[black] at (6,-\i/2+1) (g\i) {};
	}
	\foreach \i in {0,...,4} {
		\node at (7,-3*\i/4+1) (h\i) {};
	}
	\foreach \i in {0,...,5} {
		\node at (8,-3*\i/5+1) (i\i) {};
	}
\end{scope}
\begin{scope}[font=\tiny, decoration=brace]
	\draw[decorate] ($(d9)+(-.2,-.4)$) to node[below] {$\sigma$} ($(a9)+(.2,-.4)$);
	\draw[decorate] ($(f4)+(-.2,-.4)$) to node[below] {$\mu$} ($(d9)+(.2,-.4)$);
	\draw[decorate] ($(g6)+(-.2,-.4)$) to node[below] {$\eta$} ($(f4)+(.2,-.4)$);
	\draw[decorate] ($(h4)+(-.2,-.4)$) to node[below] {$\epsilon$} ($(g6)+(.2,-.4)$);
	\draw[decorate] ($(i5)+(-.2,-.4)$) to node[below] {$\delta$} ($(h4)+(.2,-.4)$);
\end{scope}
\begin{scope}[shorten <= 3pt, shorten >= 3pt, ->,>=stealth]
	\draw (a0) to (b0);
	\draw (a1) to (b1);
	\draw (a2) to (b2);
	\draw (a3) to (b4);
	\draw (a4) to (b3);
	\draw (a5) to (b5);
	\draw (a6) to (b7);
	\draw (a7) to (b6);
	\draw (a8) to (b8);
	\draw (a9) to (b9);
	\draw (b0) to (c0);
	\draw (b1) to (c1);
	\draw (b2) to (c3);
	\draw (b3) to (c2);
	\draw (b4) to (c5);
	\draw (b5) to (c4);
	\draw (b6) to (c6);
	\draw (b7) to (c7);
	\draw (b8) to (c8);
	\draw (b9) to (c9);
	\draw (c0) to (d0);
	\draw (c1) to (d2);
	\draw (c2) to (d1);
	\draw (c3) to (d4);
	\draw (c4) to (d3);
	\draw (c5) to (d5);
	\draw (c6) to (d6);
	\draw (c7) to (d7);
	\draw (c8) to (d8);
	\draw (c9) to (d9);
	\draw (d0) to (e0);
	\draw (d1) to (e0);
	\draw (d2) to (e1);
	\draw (d3) to (e1);
	\draw (d4) to (e2);
	\draw (d5) to (e2);
	\draw (d6) to (e3);
	\draw (d7) to (e4);
	\draw (d8) to (e4);
	\draw (d9) to (e5);
	\draw (e0) to (f0);
	\draw (e1) to (f1);
	\draw (e2) to (f2);
	\draw (e3) to (f2);
	\draw (e4) to (f3);
	\draw (e5) to (f4);
	\draw (f0) to (g1);
	\draw (f1) to (g2);
	\draw (f2) to (g3);
	\draw (f3) to (g5);
	\draw (f4) to (g6);
	\draw (h0) to (g0);
	\draw (h1) to (g1);
	\draw (h2) to (g3);
	\draw (h3) to (g4);
	\draw (h4) to (g6);
	\draw (i0) to (h0);
	\draw (i1) to (h1);
	\draw (i2) to (h2);
	\draw (i3) to (h3);
	\draw (i4) to (h3);
	\draw (i5) to (h4);
\end{scope}
\end{tikzpicture}
\end{aligned}
\]
\end{example}

\cref{lem.cospan_generators} implies that to define a symmetric monoidal functor $\cospan[] \to \cat{C}$
we simply need to say where to send the generators of $\cospan[]$, and check that the relevant
equations between these generators hold. The following proposition says that
these equations are exactly the axioms of special commutative Frobenius monoids.

\begin{proposition}[Cospan is the theory of special commutative Frobenius
monoids.] \label[proposition]{prop.cospan_as_scfm_theory}
Let $(\cat{C},\otimes)$ be a symmetric monoidal category. There is a
one-to-one correspondence:
\[
\left\{ \begin{array}{c} \mbox{special commutative} \\ 
\mbox{Frobenius monoids in } \cat{C} \end{array} \right\} 
\leftrightarrow
\left\{ \begin{array}{c} \mbox{strict symmetric monoidal functors} \\ 
(F,\varphi)\colon (\cospan[],\concat) \to (\cat{C},\otimes)
\end{array} \right\}.
\]
\end{proposition}

\begin{proof}[Proof (sketch).]
Suppose we have a symmetric monoidal functor $(F,\varphi)\colon (\cospan[],\concat)
\to (\cat{C},\otimes)$.  Let $\mu,\eta,\delta,\epsilon$ be the Frobenius
generators in $\cospan[]$ as defined in \cref{ex.cospan_scFrob}. It is
straightforward to verify that 
\[
(F1,\enspace 
\varphi_{1,1}\cp F\mu,\enspace
\varphi_{\varnothing}\cp F\eta,\enspace 
F\delta\cp \varphi_{1,1}^{-1},\enspace 
F\epsilon\cp \varphi_\varnothing^{-1})
\]
is a special commutative Frobenius monoid in $\cat{C}$.

The converse is trickier. Suppose $(X,\mu_X,\eta_X,\delta_X,\epsilon_X)$ is a special
commutative Frobenius monoid in $\cat{C}$; we wish to define a strict symmetric
monoidal functor $\frob\colon \cospan[] \to \cat{C}$. We send the object $1$ to
$\frob(1):= X$; this implies $m \in \cospan[]$ maps to $\frob(m) \coloneqq
X^{\otimes m}$. Using \cref{lem.cospan_generators}, to define a candidate
strict symmetric monoidal functor we only need to say where to map the cospans
$\mu$, $\eta$, $\delta$, and $\epsilon$. This is easy: we map them to the
corresponding Frobenius generator on $X$.

Verifying functoriality, however, amounts to a technical exercise verifying that
the axioms of special commutative Frobenius monoids exactly describe pushouts of
finite sets.  This is treated at a high level, using distributive laws for
props, in \cite{lack2004composing}, and also remarked upon in
\cite{rosebrugh2005generic}; we are not aware of any more detailed
treatment in writing. Once functoriality is verified, it is straightforward to
also check that $\frob$ defines a strict symmetric monoidal functor.

These constructions are evidently inverses, and so we have the stated one-to-one
correspondence.
\end{proof}

\subsection{$\cospan$ as free hypergraph category.}\label{sec.cospan_is_free}

We now wish to show that $\cospan$ is a free hypergraph category on $\Lambda$.
We begin with a lemma that provides an easy way to equip an $\of(\Lambda)$
symmetric monoidal category (see \cref{def.objectwise_free_SMC}) with a
hypergraph structure: assign a Frobenius structure to each element of $\Lambda$.

\begin{definition}\label[definition]{def.objectwise_free_SMC}
An \emph{objectwise-free structure} on a strict symmetric monoidal category
$(\cat{C},\otimes)$ consists of a set $\Lambda$ and an isomorphism of monoids
$\List(\Lambda)\cong\ob(\cat{C})$. In this case we say that $(\cat{C},\otimes)$
is $\of$ or $\of(\Lambda)$.
\end{definition}

\begin{lemma}\label[lemma]{lemma.gen_hyp_struc}
Suppose that $\cat{C}$ is an $\of(\Lambda)$ symmetric monoidal category.
Assigning a Frobenius structure to each object $[l]$, for each $l \in \Lambda$,
induces a unique hypergraph structure on $\cat{C}$.

Furthermore, if $\cat{C}$ is as above, $\cat{D}$ is a hypergraph category, and
$F\colon\cat{C}\to\cat{D}$ is a symmetric monoidal functor, then $F$ is a
hypergraph functor iff $F$ preserves the Frobenius structure on each
$l\in\Lambda$.
\end{lemma}
\begin{proof}
Suppose that for each $l\in\Lambda$ we are given a Frobenius structure
$(\mu_l,\eta_l,\delta_l,\epsilon_l)$. We need to show that this uniquely
determines a Frobenius structure on every object, satisfying
\eqref{eqn.hypergraph_def} and restricting to the chosen one on each
$l\in\Lambda$. Any object in $\cat{C}$ can be uniquely written as a list
$[l_1,\cdots,l_n]$ for some $n\in\nn$. By induction, we may assume $n=0$ or
$n=2$. When $n=0$ the Frobenius structure is by definition given by the unitors,
while when $n=2$ the Frobenius structure on the monoidal product $[l,m]=l\concat m$ is
forced to be that given by \cref{eqn.hypergraph_def}.

The second claim is similar and straightforward.
\end{proof}

\begin{remark}
It will be useful to give a more explicit description of the construction from \cref{lemma.gen_hyp_struc}, at least in the case of $\mu$, in order to fix ideas. Given an object $l=l_1\concat\cdots\concat l_n$ and a multiplication map $\mu_i\colon l_i\concat l_i\to l_i$ for each $i$, the multiplication map $\mu_l$ is given by
\begin{equation}\label{eqn.explicit_mu_lists}
\begin{aligned}
\begin{tikzpicture}[spider diagram, font=\small, baseline=(bl)]
	\node[spider={2}{1}] (a) {};
	\node[left=0 of a_in1] {$l_1\concat\cdots\concat l_n$};
	\node[left=0 of a_in2] {$l_1\concat\cdots\concat l_n$};
	\node[right=0 of a_out1] (llab) {$l_1\concat\cdots\concat l_n$};
	\node[coordinate, right=1.5 of llab] (helper) {};
	\node[spider={2}{1}, above right=0.3 and 3 of llab] (b) {};
	\node[spider={2}{1}, below right=0.3 and 3 of llab] (c) {};
	\coordinate (x1) at (b_in1-|helper);
	\coordinate[below=.7 of x1] (x2);
	\coordinate (x4) at (c_in2-|helper);
	\coordinate[above=.7 of x4] (x3);
	\draw (b_in1) -- (x1);
	\draw (x3) to (b_in2);
	\draw  (x2) to (c_in1);
	\draw (c_in2) -- (x4);
	\node[left=0 of x1] (in11) {$l_1$};
	\node[left=0 of x2] (in1n) {$l_n$};
	\node[left=0 of x3] (in21) {$l_1$};
	\node[left=0 of x4] (in2n) {$l_n$};
	\node at ($(in11)!.45!(in1n)$) {$\vdots$};
	\node at ($(in21)!.45!(in2n)$) {$\vdots$};
	\node at ($(in11)!.35!(in1n)+(.6,0)$) {$\vdots$};
	\node at ($(in21)!.35!(in2n)+(.6,0)$) {$\vdots$};
	\node at ($(b)!.45!(c)$) {$\vdots$};
	\node[right=0 of b_out1] (out1) {$l_1$};
	\node[right=0 of c_out1] (outn) {$l_n$};
	\node at ($(out1)!.45!(outn)$) {$\vdots$};
	\node at ($(llab.east)!.3!(helper)$) {$=$};
\end{tikzpicture}
\end{aligned}
\end{equation}
\[
\mu_l\colon (l_1\concat\cdots\concat l_n)\concat (l_1\concat\cdots\concat l_n) \cong
l_1 \concat l_1 \concat\cdots\concat l_n \concat l_n \To{\mu_1
\concat \dots \concat \mu_n} l_1\concat\cdots\concat l_n.
\]
\end{remark}

\begin{example}\label[example]{ex.hypergraph_structures}
The category $\cospan[\Lambda]$ can be given the structure of a hypergraph category. Indeed, it is enough by \cref{lemma.gen_hyp_struc} to give a Frobenius structure on each $l\in\Lambda$. We assign them all the same structure, namely the one given in \cref{ex.cospan_scFrob}.

Similarly, since $\linrel$ is objectwise-free on $\rr$, the Frobenius structures on $\rr$ given in \cref{ex.linrel_scfs_copy,ex.linrel_scfs_add} induce two different hypergraph structures on $\linrel$.
\end{example}

\begin{corollary}\label[corollary]{cor.cospan_setlist_hyp}
The functor $\cospan[-]\colon\smset_{\List}\to\smcat$ from \cref{prop.cospan_setlist_cat} factors through the inclusion $\hyp_\of\ss\smcat$, giving a functor
\[\cospan[-]\colon\smset_{\List}\to\hyp_\of.\]
\end{corollary}
\begin{proof}
In \cref{ex.hypergraph_structures}, we showed that $\cospan$ is a hypergraph category for each $\Lambda$, and it is objectwise-free because $\ob(\cospan)=\List(\Lambda)$. If $f\colon\Lambda\to\List(\Lambda')$ is a function, we need to check that $\cospan[f]$ preserves the Frobenius structure $(\mu,\eta,\delta,\epsilon)$ on every object. This is a simple calculation; we carry it out for $\mu$ and leave the others to the reader.

By \cref{lemma.gen_hyp_struc}, it suffices to check that $\mu$ is preserved for for an arbitrary $l\in\Lambda$. The cospan $\mu_l$ is shown on the left of the diagram below, and if $f(l)$ is a list of length $n$, then by \cref{eqn.cospan_on_kleisli_morphisms}, $\cospan[f](\mu_l)$ is shown on the right
\[
\begin{tikzcd}[row sep=15pt]
	{1\concat 1}\ar[dr, bend right, "\copair{l,l}"']\ar[r,
	"\copair{\id,\id}"]&{1}\ar[d, "l"]&{1}\ar[dl, bend left, "l"]\ar[l, equal]\\
	&\Lambda\ar[d, "f"]\\[5pt]
	&\List(\Lambda')
\end{tikzcd}
\hspace{.7in}
\begin{tikzcd}[row sep=55pt,column sep=35pt]
	\ul{n}\concat\ul{n}\ar[dr, bend right, "\copair{f(l),f(l)}"']\ar[r,
	"\copair{\id_n,\id_n}"]&\ul{n}\ar[d, "{f(l)}"]&\ul{n}\ar[dl, bend left, "f(l)"]\ar[l, equal]\\
	&\Lambda'
\end{tikzcd}
\]
But this cospan is exactly the one from \cref{lemma.gen_hyp_struc}; see also \cref{eqn.explicit_mu_lists}.
\end{proof}

We also denote by $\cospan[-]$ the composite of the functor from \cref{cor.cospan_setlist_hyp} with the faithful inclusion $\smset\to\smset_\List$ and the fully faithful inclusion $\hyp_\of\to\hyp$:
\[
\begin{tikzcd}[column sep=45pt]
	\smset\ar[d, "(\ff)"']\ar[r,"{\cospan[-]}"]&\hyp\\
	\smset_\List\ar[r,"{\cospan[-]}"']&\hyp_\of\ar[u, "(\ff)"']
\end{tikzcd}
\]

The following theorem states that $\cospan$ is the free hypergraph category on
$\Lambda$. In particular, this theorem produces a hypergraph
functor
\[
\frob_{\cat{H}}\colon \cospan[\ob(\cat{H})] \to \cat{H}
\]
which is so-named because its image provides all the Frobenius
morphisms on all the objects of $\cat{H}$. In fact, $\frob$ arises as the counit
of an adjunction.

\begin{theorem}\label[theorem]{thm.adjunction_cospan_ob}
$\cospan$ is the free hypergraph category on the set $\Lambda$. That is, there is an adjunction
\[
\begin{tikzcd}[column sep=45pt]
	\smset\ar[r, shift left, "{\cospan[-]}"]&
	\hyp.\ar[l, shift left, "\ob"]
\end{tikzcd}
\]
\end{theorem}
\begin{proof}
We want to show $\ob$ is right adjoint to $\cospan[-]$, so we provide a unit
transformation and counit transformation and check the triangle identities.

For any $\Lambda\in\smset$, we have $\ob(\cospan)=\List(\Lambda$), so we take the unit map $\Lambda\to\ob(\cospan)$ to be the unit natural transformation $\sing$ from the $\List$ monad. 

Suppose $\cat{H}$ is a hypergraph category; for the counit of the adjunction, we
need a hypergraph functor $\cospan[\ob(\cat{H})]\to\cat{H}$. Note that for each object of
$\cat{H}$, \cref{prop.cospan_as_scfm_theory} gives a strong symmetric monoidal
functor $\cospan[] \to \cat{H}$. Observing that $\cospan[\ob(\cat{H})]$ is the
coproduct, in the category of symmetric monoidal categories and strong symmetric
monoidal functors, of $\ob(\cat{H})$-many copies of $\cospan[]$, the copairing of
all these functors thus gives a strong symmetric monoidal functor
$\frob_{\cat{H}}\colon \cospan[\ob(\cat{H})] \to \cat{H}$. It is straightforward
to observe that this functor is hypergraph.

It remains to check that $\frob$ is natural (as its subscripts $\cat{H}$ vary), and that the
triangle identities hold. The map $\frob$ is natural because hypergraph functors
$\cat{H}\to\cat{H}'$ are required to preserve Frobenius structures. Finally, for the triangle
identities, we need to check that the following diagrams commute:
\[
\begin{tikzcd}[column sep=0]
  \cospan \ar[rd, pos=.4, "{\cospan[\sing_\Lambda]}"'] \ar[rr, equal]
&&
  \cospan
\\
&
  \cospan[\ob(\cospan)]\ar[ru, "{\frob_{\cospan}}"']
\end{tikzcd}
\hspace{.35in}
\begin{tikzcd}[column sep=0]
&
	\ob(\cospan[\ob(\cat{H})])\ar[rd, pos=.6, "\ob(\frob_{\cat{H}})"]
\\
	\ob(\cat{H})\ar[ru, pos=.4, "\sing_{\ob(\cat{H})}"]\ar[rr, equal]
&&
	\ob(\cat{H})
\end{tikzcd}
\]
Both are straightforward calculations.
\end{proof}

\subsection{$\cospan$ as the free hypergraph category over
$\Lambda$.}\label{sec.cospan_is_free2}

\begin{corollary}\label[corollary]{cor.factor_adj}
The functor $\cospan[-]\colon\smset_\List\to\hyp_\of$, constructed in \cref{cor.cospan_setlist_hyp}, is fully faithful and has a right adjoint: 
\[
\begin{tikzcd}[column sep=large]
	\smset_\List\ar[r, shift left, "{\cospan[-]}"]&
	\hyp_\of.\ar[l, shift left, "\gens"]
\end{tikzcd}
\]
Moreover, the components $\frob_{\cat{H}}\colon\cospan[\gens(\cat{H})]\to\cat{H}$ of the counit transformation are $\io$ hypergraph functors.
\end{corollary}
\begin{proof}[Proof (sketch).]
As a right adjoint to $\cospan[-]$, we propose the functor $\gens$ given by sending an $\of$-hypergraph category $(\Lambda,\cat{H},i)$ to the set $\Lambda$ of generators. It is clearly functorial.

The proof that $\gens$ is right adjoint to $\cospan[-]$ is analogous to, though a bit easier than, that of \cref{thm.adjunction_cospan_ob}. Rather than the unit map being $\sing$, the unit of the $\List$ monad, here it is simply the identity map $\Lambda\to\Lambda$ in $\smset_\List$, so the triangle identities become trivial. For the $\cat{H}$-component $\frob_{\cat{H}}$ of the counit transformation, simply replace $\ob(\cat{H})$ with $\Lambda$ throughout the proof. For any $x\in\cat{H}$ we have a list $(x_1,\ldots,x_n)\in\List(\Lambda)$ with $x=x_1\concat\cdots\concat x_n$, so the hypergraph functor $\frob_{\cat{H}}$ is indeed identity-on-objects.

Finally it is well-known that a left adjoint is fully faithful iff the corresponding unit map is a natural isomorphism, and indeed for any $\Lambda$, the unit map $\Lambda\to\gens(\cospan[\Lambda])=\Lambda$ is the identity. 
\end{proof}

\begin{proposition}\label[proposition]{prop.cospan_frob_id}
In the case $\cat{H}=\cospan$, the counit map
$
\frob_{\cospan}\colon\cospan\to\cospan
$
is the identity.
\end{proposition}
\begin{proof}
The main idea here is that the counit selects out the Frobenius morphisms of
$\cospan$, and since $\cospan$ is free, these are \emph{all} the morphisms.

More precisely, observe that since left adjoints preserve coproducts, and
$\cospan[-]$ is a left adjoint (\cref{cor.factor_adj}), we now know $\cospan \cong
\coprod_{l \in \Lambda} \cospan[]$ in $\hyp$. Fix $l\in \Lambda$. The Frobenius
structure on the object $[l] \in \cospan$ is given by \cref{ex.cospan_scFrob},
and thus the corresponding map $\frob_l \colon \cospan[] \to \cospan$ given by
\cref{prop.cospan_as_scfm_theory} is precisely the inclusion into the $l$th
summand. Moreover, by the proof of \cref{thm.adjunction_cospan_ob},
$\frob_{\cospan}$ is built as the copairing over $\Lambda$ of these maps
$\frob_l$. Thus $\frob_{\cospan}=[\frob_l]_{l \in \Lambda} = \id_{\cospan}$, as
required.
\end{proof}

\begin{remark}
To summarize, consider the following diagram of categories and functors:
\begin{equation}\label{eqn.3_adjunctions_square}
\begin{tikzcd}[row sep=large, column sep=50pt]
	\smset\ar[r, shift left=5pt, "{\cospan[-]}"]\ar[d, shift left=5pt, "i"]
	\ar[d, phantom, "\scriptstyle \vdash"]\ar[r, phantom, "\scriptstyle \bot"]
&
	\hyp\ar[d, shift left=5pt, "\Str"]\ar[l, shift left=5pt, "\ob"]
\\
	\smset_\List\ar[r, shift left=5pt, "{\cospan[-]}"]\ar[u, shift left=5pt, "\List"]
	\ar[r, phantom, "\scriptstyle \bot"]
&
	\hyp_\of\ar[u, shift left=5pt, "U"]\ar[l, shift left=5pt, "\gens"]
\end{tikzcd}
\end{equation}
The left-hand adjunction is the usual one between $\smset$ and the Kleisli
category of the $\List$ monad. The top adjunction was proved in
\cref{thm.adjunction_cospan_ob}, while the bottom adjunction was proved in \cref{cor.factor_adj}. The right-hand map $U$ just sends an $\of$-hypergraph category $(\Lambda,\cat{H},i)$ to the underlying hypergraph category $\cat{H}$, and the strictification functor $\Str$ will be constructed in \cref{thm.stricthypergraphs}, namely as the underlying 1-functor of \cref{eqn.strictification_functor}. But beware that $U$ and $\Str$ are not adjoint: there are faux unit and counit maps $\cat{H}\to U\Str(\cat{H})\to\cat{H}$, but they do not satisfy the triangle identities. In some sense this right-hand part of the diagram is stronger than the rest---it is the shadow of the 2-equivalence $\hhyp_\of\cong\hhyp$ from \cref{thm.stricthypergraphs}---but in another sense it is weaker in that there is no adjunction between the underlying 1-categories.
\end{remark}

\section{Factoring hypergraph functors}\label{sec.io_ff}

Hypergraph functors naturally factor into two sorts: those that are identity-on-objects ($\io$) and those that are fully faithful ($\ff$)---roughly speaking, ``identity on morphisms''. Indeed, given a hypergraph functor $F\colon\cat{H}_1\to\cat{H}_2$,  define a new hypergraph category $\cat{H}_F$ as follows: its objects are the same as those of the domain, and for every two objects $x,y\in\ob(\cat{H}_F)$, the hom-set is that of the their images under $F$,
\begin{equation}\label{eqn.def_hyperfact}
	\ob(\cat{H}_F)\coloneqq\ob(\cat{H}_1)
	\qquad\text{and}\qquad
	\cat{H}_F(x,y)\coloneqq\cat{H}_2(Fx, Fy).
\end{equation}
The monoidal unit object and the monoidal product on objects in $\cat{H}$ are inherited from $\cat{H}_1$, and the monoidal product on morphisms together with all the Frobenius structures are inherited from $\cat{H}_2$. One easily constructs an identity-on-objects functor $\cat{H}_1\To{\io}\cat{H}_F$ and a fully faithful functor $\cat{H}_F\To{\ff}\cat{H}_2$, of which the composite is $F$,
\[
\begin{tikzcd}[row sep=0]
	\cat{H}_1\ar[rr, "F"]\ar[dr, bend right=10pt, "\io"']&&\cat{H}_2\\
	&\cat{H}_F\ar[ur, bend right=10pt, "\ff"']
\end{tikzcd}
\]
Of course, there are details to check, but we leave them to the reader.

\begin{remark}
The above forms an \emph{orthogonal factorization system} $(\io,\ \ff)$ on the 2-category $\hhyp$. See \cite{Spivak.Schultz.Rupel:2016a} for a definition and a similar result in the case of traced and compact closed categories. However, we will not need to use this fact, so we omit the proof.
\end{remark}

In fact the $(\ff,\io)$ factorization is special in that it leads to a fibration of categories, as we will show in \cref{prop.groth_fib_hypof}; it will help to first prove a lemma.

\begin{lemma}\label[lemma]{lemma.factoring_and_frob}
Let $g\colon\Lambda_1\to\Lambda_2$ be a morphism in $\smset_\List$, let
$\cat{H}_2$ be a hypergraph category such that $\gens(\cat{H}_2)=\Lambda_2$, and let
$\frob_2\colon\cospan[\Lambda_2]\to\cat{H}_2$ be the counit map on $\cat{H}_2$ of the
adjunction $\cospan[-] \dashv \gens$ from \cref{cor.factor_adj}. 

Consider the $(\io,\ff)$ factorization $\cospan\To{i_1}\cat{H}_1\To{G}\cat{H}_2$ of the composite $\cospan[g]\cp\frob_2$:
\begin{equation}\label{eqn.factor_functor}
\begin{tikzcd}[column sep=45pt]
	\cospan[\Lambda_1]\ar[r, "{\cospan[g]}"]\ar[d, "i_1\ (\io)"']&
	\cospan[\Lambda_2]\ar[d, "\frob_2"]\\
	\cat{H}_1\ar[r, "G\ (\ff)"']&\cat{H}_2
\end{tikzcd}
\end{equation}
Then we have $i_1=\frob_1$, the counit map on $\cat{H}_1$.
\end{lemma}
\begin{proof}
\cref{cor.factor_adj} gives the bijection $\hyp(\cospan[\Lambda_1],\cat{H_1})
\cong \smset_\List(\Lambda_1,\gens(\cat{H}_1))$. Note that $\Lambda_1 =
\gens(\cat{H}_1)$. The functor $i_1$ induces the identity map
$\Lambda_1\to\Lambda_1$ on generators and hence maps to the identity map on
$\Lambda_1$ under this bijection. Since $\frob_1$ does the same, the two
functors must be equal.
\end{proof}

\begin{proposition}\label[proposition]{prop.groth_fib_hypof}
The functor $\gens\colon\hyp_\of\to\smset_\List$ from \cref{cor.factor_adj} is a split Grothendieck fibration.
\end{proposition}
\begin{proof}
We first want to show $\gens$ is a fibration, so suppose given a diagram
\[
\begin{tikzcd}
	&\cat{H}_2\ar[d, maps to, "\gens"]
	\\
	\Lambda_1\ar[r,"g"']&\Lambda_2
\end{tikzcd}
\]
We want to find a cartesian morphism $G$ over $g\in\smset_\List$. Since
$\gens(\cat{H}_2)=\Lambda_2$, we can factor $\cospan[g]\cp\frob_2$ as in
\cref{lemma.factoring_and_frob} to obtain the commutative square
\cref{eqn.factor_functor}. We claim that the map $G\colon \cat{H}_1\to\cat{H}_2$ is cartesian. So suppose given a solid-arrow diagram as to the
left below; it is equivalently described by the solid-arrow diagram to the
right:
\[
\begin{tikzcd}[column sep=small]
	\cat{H}_0\ar[rr, bend left, "H"]\ar[r, dashed]\ar[d, maps to]&
	\cat{H}_1\ar[r, "G"']\ar[d, maps to]&
	\cat{H}_2\ar[d, maps to]\\
	\Lambda_0\ar[r, "f"']&
	\Lambda_1\ar[r, "g"']&
	\Lambda_2
\end{tikzcd}
\hspace{.8in}
\begin{tikzcd}[column sep=large]
	\cospan[\Lambda_0]\ar[r, "{\cospan[f]}"]\ar[d, "(\io)"]&
	\cospan[\Lambda_1]\ar[r, "{\cospan[g]}"]\ar[d, "(\io)"]&
	\cospan[\Lambda_2]\ar[d, "(\io)"]\\
	\cat{H}_0\ar[rr, bend right=15pt, "H"']\ar[r, dashed, "F"]&
	\cat{H}_1\ar[r, "{G\;(\ff)}"]&
	\cat{H}_2
\end{tikzcd}
\]
We need to show there is a unique dashed map $F\colon\cat{H}_0\to\cat{H}_1$ making the bottom triangle on the right-hand diagram commute. But because the vertical maps are $\io$, we take $F$ on objects to agree with $\cospan[f]$, and because $G$ is fully faithful, we take $F$ on morphisms to agree with $H$. This is the only possible choice to make the diagrams commute, and it will be a hypergraph functor because $G$ and $H$ are.

We have proved that $\gens$ is a Grothendieck fibration. It is split, meaning that our choices of Cartesian maps are closed under composition, because \cref{eqn.def_hyperfact} defines the factorization system up to equality.
\end{proof}

In general, split Grothendieck fibrations $p\colon E\to B$ can be identified
with functors $\corners{p}\colon B\to\smcat\op$. In the case of
\cref{prop.groth_fib_hypof}, the functor
$\corners{\gens}\colon\smset_\List\to\smcat\op$ shall be denoted
$\hyp_{\of(-)}$. It sends an object $\Lambda$ to the category
$\hyp_{\of(\Lambda)}$ of hypergraph categories that are objectwise-free on
$\Lambda$ and the $\io$ hypergraph functors between them. It sends a morphism
$f\colon\Lambda_1\to\List(\Lambda_2)$ to the functor
$\hyp_{\of(f)}\colon\hyp_{\of(\Lambda_2)}\to\hyp_{\of(\Lambda_1)}$ defined by
factorization as in \cref{eqn.factor_functor}; in other words, on objects
$\hyp_{\of(f)}$ maps $\cat{H}_2$ to the hypergraph category $\cat{H}_1$ given by
the factorization
\begin{equation}\label{eqn.factor_functor2}
\begin{tikzcd}[column sep=45pt]
	\cospan[\Lambda_1]\ar[r, "{\cospan[f]}"]\ar[d, "\frob_1"']\ar[d, "(\io)"]&
	\cospan[\Lambda_2]\ar[d, "\frob_2"]\\
	\hyp_{\of(f)}(\cat{H}_2) \coloneqq \cat{H}_1\ar[r, "F\;(\ff)"']&\cat{H}_2
\end{tikzcd}
\end{equation}

The opposite direction, taking a functor $B\to\smcat\op$ and returning a
fibration over $B$ is called the Grothendieck construction. We immediately have
the following.

\begin{corollary}\label[corollary]{cor.hypof_fibration}
There is an equivalence of categories
\[
  \hyp_\of
  \quad\To{\cong}
  \int^{\Lambda\in\smset_\List}\hyp_{\of(\Lambda)}.
\]
\end{corollary}
%

\section{Strictification of hypergraph categories}\label{sec.strictification}

In this subsection we prove that every hypergraph category is hypergraph equivalent to a strict hypergraph category. In fact, there is a 2-equivalence $\hhyp\cong\hhyp_\of$. This coherence result will be the first step in formalizing the relationship between hypergraph categories and cospan-algebras.

Define the 2-category $\hhyp_\of$ as the full sub-2-category of objectwise-free
hypergraph categories. That is, the objects of $\hhyp_\of$ are hypergraph
categories $\cat{H}$ such that there exists a set $\Lambda$ and a bijection
$i\colon\List(\Lambda)\to\ob(\cat{H})$, and given two $\of$-hypergraph
categories $(\cat{H},\Lambda,i)$ and $(\cat{H}',\Lambda',i')$, the
hom-category between them is simply
\[\hhyp_\of\big((\cat{H},\Lambda,i),(\cat{H}',\Lambda',i')\big)\coloneqq\hhyp(\cat{H},\cat{H}').\]
There is an
obvious forgetful functor $U\colon\hhyp_\of\to\hhyp$, and by construction it is
fully faithful.

\begin{theorem}\label[theorem]{thm.stricthypergraphs}
The functor $U\colon\hhyp_\of\to\hhyp$ is a 2-equivalence. In particular it is essentially surjective, i.e.\ every hypergraph category is hypergraph equivalent to an objectwise-free hypergraph
  category.
\end{theorem}
\begin{proof}
  Since $U$ is fully faithful by definition, it suffices to show that it is essentially surjective.
  
  Let $(\cat{H},\ot)$ be a hypergraph category. As $\cat{H}$ is, in particular, a symmetric monoidal category, a standard
  construction (see Mac Lane \cite[Theorem XI.3.1]{MacLane:1998a}) gives an equivalent
  strict symmetric monoidal category $\cat{H}_{\str}$, the \emph{strictification of $\cat{H}$}, whose construction we detail here.
  
  Let $\Lambda\coloneqq\ob(\cat{H})$. The set of objects in the strictification
  is $\ob(\cat{H}_\str)\coloneqq\List(\Lambda)$, i.e., finite lists
  $[x_1,\ldots,x_m]$ of objects in $\cat{H}$. For each such list, let
  $Px\coloneqq(((x_1 \ot x_2) \ot \dots) \ot x_m) \ot I$ denote the
  ``pre-parenthesized product of $x$'' in $\cat{H}$ with all open parentheses at
  the front. Note that $P$ applied to the empty list is the monoidal unit $I$,
  and that for any pair of lists $x,y$ there is a canonical isomorphism
  $P(\copair{x,y})\cong\copair{Px,Py}$,
  \begin{multline}\label{eqn.pre_paren_iso}
  (((((x_1\otimes x_2)\otimes\cdots\otimes x_m)\otimes y_1)\otimes y_2)\otimes\cdots\otimes y_n)\otimes I\\\cong
  (((x_1\otimes x_2)\otimes\cdots\otimes x_m)\otimes I)\otimes (((y_1\otimes y_2)\otimes\cdots\otimes y_n)\otimes I).
  \end{multline}
  The morphisms $[x_1,\ldots,x_m] \to [y_1,\ldots,y_n]$ in $\cat{H}_\str$ are the morphisms $Px\to Py$ in $\cat{H}$, and composition is inherited from $\cat{H}$. The monoidal structure on objects in $\cat{H}_{\str}$ is given by
  concatenation of lists; the monoidal unit is the empty list. The monoidal product of two morphisms in $\cat{H}_\str$ is
  given by their monoidal product in $\cat{H}$ pre- (and post-) composed with the
  canonical isomorphism (and its inverse) from \cref{eqn.pre_paren_iso}.
  
  By design, the associators and unitors of $\cat{H}_{\str}$ are simply
  identity maps, and the braiding $\copair{x,y}\to\copair{y,x}$ is given by the braiding
  $Px \ot Py \to Py \ot Px$ in $\cat{H}$, similarly pre- and post-composed with
  the isomorphisms from \cref{eqn.pre_paren_iso}. This defines a strict symmetric monoidal
  category \cite{MacLane:1998a}, and it is objectwise-free on $\Lambda$ by construction. This construction is 2-functorial: given a strong monoidal functor between monoidal categories (resp.\ a monoidal natural transformation between monoidal functors), there is an evident strict monoidal functor (resp.\ a monoidal natural transformation) between strictifications.

  To make $\cat{H}_{\str}$ into a hypergraph category, we equip each
  object $x=[x_1,\ldots,x_n]$ with an Frobenius structure
  $(\mu,\eta,\delta,\epsilon)$ using the monoidal product, over $i=1,\ldots,n$, of corresponding Frobenius structures $(\mu_i,\eta_i,\delta_i,\epsilon_i)$ from $\cat{H}$, and pre- or post-composition with canonical isomorphisms from \cref{eqn.pre_paren_iso}. For example, the multiplication $\mu$ on $x\in\ob(\cat{H}_\str)$ 
  is given by 
  \[P(\copair{x,x})\cong((((x_1 \ot x_1) \ot (x_2 \ot x_2)) \ot \dots) \ot (x_n \ot x_n)) \ot I\To{((\mu_1\otimes\mu_2)\otimes\cdots\otimes\mu_n)\otimes\id_I}P(x).
  \]
  As the coherence maps are natural, each special
  commutative Frobenius monoid axiom for this data on $[x_1,\ldots,x_n]$ reduces
  to a list of the corresponding axioms for the objects $x_i$ in $\cat{H}$.
  Similarly, the coherence axioms and naturality of the coherence maps imply the
  Frobenius structure on the monoidal product of objects is given by the
  Frobenius structures on the factors in the required way.
  
  Thus we have upgraded $\cat{H}_\str$ to a hypergraph category. Moreover, this construction is 2-functorial; all that needs to be checked is that the usual strictification of a hypergraph functor $\cat{H}\to\cat{H}'$ preserves the hypergraph structure on $\cat{H}_\str$ and $\cat{H}'_\str$ as defined above, which is easy to see.
  Thus we have a 2-functor
  \begin{equation}\label{eqn.strictification_functor}
  	\Str\colon\hhyp\to\hhyp_\of,
	\end{equation}
  and it remains to prove the equivalence of $\cat{H}$ and $\cat{H}_\str$.

  Mac Lane's standard construction further gives strong symmetric monoidal
  functors $P\colon \cat{H}_{\str} \to \cat{H}$, extending the map $P$ above,
  and $S\colon \cat{H} \to \cat{H}_{\str}$ sending $x \in \cat{H}$ to the
  length-1 list $[x]\in\cat{H}_{\str}$, and $P$ and $S$ form an equivalence of
  symmetric monoidal categories. 
  
  Moreover, it is straightforward to check that $P$ and $S$ preserve the
  hypergraph structure defined above, and thus form an equivalence of hypergraph
  categories. The fact that $P$ preserves the hypergraph structure follows from
  the compatibility of the Frobeinus structures with the monoidal product
  required in the definition of hypergraph category.
  
  Note in particular that $\cat{H}$ must obey the unit coherence axiom (see \cref{def.hypergraph_cat_func}) in order for the
  Frobenius structure on the monoidal unit $\varnothing$ of $\cat{H}_{\str}$ to
  map to the Frobenius structure on its image $P(\varnothing) = I$ of $\cat{H}$.
  By construction, the Frobenius structure on $\varnothing$ just comprises
  identity maps $\varnothing \to \varnothing$; indeed, since $\cat{H}_\str$ is
  strict, \cref{prop.unique_Frob_on_I} shows this is the only Frobenius
  structure it could have. The image of this Frobenius structure under $P$ is
  then defined by $P$'s monoidal coherence maps, and these
  coherence maps define precisely the canonical Frobenius structure on $I$
  detailed in \cref{ex.scFrob_on_I}. 
\end{proof}

\chapter{Cospan-algebras and hypergraph categories are equivalent}\label{sec.hyp_cospan_alg}
In \cref{thm.stricthypergraphs} we showed that there is a 2-equivalence between the bicategories $\hhyp$ and
$\hhyp_\of$. Our remaining goal is to show there is an equivalence between the
(1-) categories $\hyp_\of$ and $\calg$. We will build this equivalence in parts. 

In \cref{ssec.hcca} we produce a functor $\hcca\colon\hyp_{\of(\Lambda)}\to\lax(\cospan,\smset)$ natural in $\Lambda$, and in \cref{ssec.cahc} we produce a functor $\cahc$ in the opposite direction. In \cref{ssec.equivalences} we prove that $\hcca$ and $\cahc$ are mutually inverse, giving an equivalence of categories
\begin{equation}\label{eqn.equiv_hyp_calg_Lambda}
\hyp_{\of(\Lambda)} \cong \lax(\cospan,\smset).
\end{equation}
These equivalences will again be natural in $\Lambda\in\smset_\List$, so we will be able to gather them together into a single equivalence, $\hyp_{\of} \cong \calg$.

\section{From hypergraph categories to cospan-algebras }
\label{ssec.hcca}

Our aim in this subsection is to provide one half of the equivalence
\eqref{eqn.equiv_hyp_calg_Lambda}, converting any hypergraph category $\cat{H}$ into a cospan-algebra $A_{\cat{H}}$. This is given by the following. 

\begin{proposition} \label[proposition]{prop.hyptocalg}
For any $\Lambda\in\smset_\List$, we can naturally construct a functor
\[
\hcca\colon \hyp_{\of(\Lambda)} \to \lax(\cospan,\smset).
\]
\end{proposition}

This will be proved on page \pageref{proof.hyptocalg}. First we prove two lemmas, which we use to define $\hcca$ on objects and on
morphisms of $\hyp_{\of(\Lambda)}$ respectively.

\begin{lemma}\label{lem.hcca_on_obj}
  Let $\cat{H}$ be an $\of$ hypergraph category with $\Lambda=\gens(\cat{H})$;
  by \cref{cor.factor_adj} we have an identity-on-objects hypergraph functor
  $\frob\colon\cospan\to\cat{H}.$ The set of maps out of the monoidal unit
  $I\in\cat{H}$ defines a lax symmetric monoidal functor
\begin{align}
	\nonumber
	\hcca[\cat{H}]\colon\cospan&\to\smset\\
	\label{eqn.Psi_on_obj}
	\hcca[\cat{H}](-)&\coloneqq\cat{H}\big(I,\frob(-)\big)
\end{align}
\end{lemma}
\begin{proof}
The formula \eqref{eqn.Psi_on_obj} makes sense not only for objects in $\cospan$ but also for morphisms, and it makes clear how to endow $A_\cat{H}$ with  a lax structure. Indeed, given a morphism $f\colon X\to Y$ in $\cospan$, composing with $\frob(f)$ induces a function $\cat{H}(I,\frob(X))\to\cat{H}(I,\frob(Y))$, and this defines $A_\cat{H}$ on morphisms. For the laxators, we need a function $\gamma\colon\{1\}\to A_\cat{H}(0)$ and a function $\gamma_{X,Y}\colon A_\cat{H}(X)\times A_\cat{H}(Y)\to \hcca[\cat{H}](X\concat Y)$ for any $X,Y\in\cospan$. Since $\frob$ is $\io$ (\cref{cor.factor_adj}), we can define the functions $\gamma$ and $\gamma_{X,Y}$ as follows:
\begin{gather}\label{eqn.laxator_cospan_alg}
  \{1\}\To{\id_I}\cat{H}(I,I)= A_\cat{H}(\emptylist),\\\nonumber
  \cat{H}(I,\frob X)\times\cat{H}(I,\frob Y)\To{\otimes}\cat{H}(I,\frob X\otimes \frob Y)= \hcca[\cat{H}](X\concat Y).
\end{gather}
It is easy to check that these satisfy the necessary coherence conditions.
\end{proof}

We next want to define $\hcca$ on morphisms, so suppose that $\cat{H}$ and $\cat{H}'$ are hypergraph categories. Morphisms between them in $\hyp_{\of(\Lambda)}$ are $\io$ hypergraph functors $F\colon\cat{H}\to\cat{H}'$, and morphisms between their images in $\lax(\cospan,\smset)$ are natural transformations $\alpha\colon A_\cat{H}\to \hcca[\cat{H}']$
\[
\begin{tikzcd}[column sep=large]
	\cospan
		\ar[r, bend left=15pt, "{\hcca[\cat{H}]}"]
		\ar[r, bend right=15pt, "{\hcca[\cat{H}']}"']
		\ar[r, phantom, "\scriptstyle\Downarrow\alpha"]
	&
	\smset.{\color{white}Cos_\Lambda}
\end{tikzcd}
\]
So given $F$, in order to define $\alpha\coloneqq\hcca[F]$, we first note that $F\colon\cat{H}\to\cat{H}'$ induces a commutative diagram of $\io$ hypergraph functors
\begin{equation}\label{eqn.io_cospan_triangle}
\begin{tikzcd}[row sep=small]
	&\cospan\ar[dl, "\frob"']\ar[dr, "\frob'"]\\
	\cat{H}\ar[rr, "F"']&&\cat{H}'.
\end{tikzcd}
\end{equation}

\begin{lemma}\label{lem.hcca_on_mor}
Let $F\colon \cat{H} \to \cat{H'}$ be an $\io$ hypergraph functor between
hypergraph categories over $\Lambda$. For any $X\in\cospan$, define $\alpha_X\colon A_\cat{H}(X)\to \hcca[\cat{H}'](X)$ as the composite
\begin{equation}\label{eqn.Psi_on_mor}
	A_\cat{H}(X)\coloneqq\cat{H}(I,\frob(X))\Too{F}\cat{H}'(FI,F\circ\frob (X))=\cat{H'}(I,\frob'(X))=:\hcca[\cat{H}'](X).
\end{equation}
This defines a natural transformation $\alpha\colon \cat{H} \to \cat{H'}$.
\end{lemma}
\begin{proof}
The naturality and monoidality of $\alpha$, i.e.\ the commutativity of the following diagrams for any $X,Y$ and $f\colon X\to Y$ in $\cospan$
\[
\begin{tikzcd}[column sep=18pt,font=\footnotesize]
  A_\cat{H}(X)
  	\ar[d, "A_\cat{H}(f)"']\ar[r, "\alpha_X"]&
	\hcca[\cat{H}'](X)
		\ar[d, "A_\cat{H}(f)"]\\
  A_\cat{H}(Y)
  	\ar[r, "\alpha_Y"']&
	\hcca[\cat{H}'](X)	
\end{tikzcd}
\hspace{.3in}
\begin{tikzcd}[column sep=-5pt,font=\footnotesize]
	&\{1\}\ar[dl, "\gamma"']\ar[dr, "\gamma'"]\\
	A_\cat{H}(\emptylist)\ar[rr, "\alpha_\emptylist"']&&\hcca[\cat{H}'](\emptylist)
\end{tikzcd}
\hspace{.3in}
\begin{tikzcd}[column sep=20pt,font=\footnotesize]
	\hcca[\cat{H}](X)\times \hcca[\cat{H}](Y)
		\ar[d, "\alpha_X\times\alpha_Y"']
		\ar[r, "\gamma_{X,Y}"]
	&
	\hcca[\cat{H}](X\concat Y)
		\ar[d, "\alpha_{X\concat Y}"]
	\\
	\hcca[\cat{H}'](X)\times \hcca[\cat{H}'](Y)
		\ar[r, "\gamma_{X,Y}'"']	
	&
	\hcca[\cat{H}'](X\concat Y)
\end{tikzcd}
\]
arise from the functoriality and strong monoidality of $F$.
\end{proof}

\begin{proof}[Proof of \cref{prop.hyptocalg}]\label{proof.hyptocalg}
Choose any set $\Lambda$.%
\footnote{To be more precise we might use $\Lambda$ to annotate the functor
$\hcca$ as $\hcca^\Lambda$, but we leave off the superscript for typographical
reasons.}
We define $\hcca$ on objects by \cref{lem.hcca_on_obj} and on morphisms by
\cref{lem.hcca_on_mor}. It remains to check this is functorial, and that it is
natural as $\Lambda$ varies in $\smset_\List$.

If $F=\id_{\cat{H}}$ is the identity then each component $\alpha_X=\hcca[F](X)$
defined in \cref{eqn.Psi_on_mor} is also the identity. Similarly, given two
composable $\io$ hypergraph functors $\cat{H}\To{F}\cat{H}'\To{F'}\cat{H}''$,
the associated commutative diagrams \cref{eqn.io_cospan_triangle} compose, and
again by \cref{eqn.Psi_on_mor}, we have $\hcca[F](X)\cp\hcca[F'](X)=\hcca[F\cp
F'](X)$ for any $X\in\cospan$.

We now show that the above is natural in $\Lambda\in\smset_\List$. Let $f\colon\Lambda\to\List(\Lambda')$ be a function; we want to show that the following square commutes:
\begin{equation}\label{eqn.naturality1}
\begin{tikzcd}
	\hyp_{\of(\Lambda')}\ar[r, "\hcca"]\ar[d, "\hyp_{\of(f)}"']&
	\lax(\cospan[\Lambda'],\smset)\ar[d, "{\cospan[f]}"]\\
	\hyp_{\of(\Lambda)}\ar[r, "\hcca"']&
	\lax(\cospan[\Lambda],\smset)
\end{tikzcd}
\end{equation}
That is, for any $\cat{H}'\in\hyp_{\of(\Lambda')}$, we want to show
$\cospan[f]\cp\hcca[\cat{H}']=\hcca[\hyp_{\of(f)}(\cat{H}')]$ as functors
$\cospan\to\smset$. Recalling from \cref{eqn.Psi_on_obj} the definition
$\hcca[\cat{H}] = \frob \cp \cat{H}(I,-)$, this is simply the commutativity of
the diagram
\[
\begin{tikzcd}[column sep=15pt,row sep=15pt]
	\cospan[\Lambda]\ar[rr, "{\cospan[f]}"]\ar[d, "\frob"']
	&
	& \cospan[\Lambda']\ar[d, "\frob'"]
	\\
	\hyp_{\of(f)}(\cat{H}') \ar[rr,"F"]
	\ar[dr,"{\hyp_{\of(f)}(\cat{H}')(I,-)}"']
	&
	&\cat{H}' \ar[dl,"{\cat{H}'(I,-)}"] \\
	& \smset
\end{tikzcd}
\]
which commutes by \cref{eqn.factor_functor2} and the fact that $F$ is a fully
faithful hypergraph functor.
\end{proof}

\begin{remark}
An analogous construction can be used to obtain a cospan-algebra from any
hypergraph category $\cat{H}$, even if it is not objectwise-free. Rather than use the counit map $\frob_\cat{H}\colon\cospan[\gens(\cat{H})]\to\cat{H}$ from \cref{cor.factor_adj}, we use the counit map $\frob'_\cat{H}\colon\cospan[\ob(\cat{H})]\to\cat{H}$ from \cref{thm.adjunction_cospan_ob}. Then in place of the lax monoidal functor $\cat{H}(I,\frob_{\cat{H}}(-))$ from \cref{eqn.Psi_on_obj}, one defines a lax monoidal functor $\cospan[\ob(\cat{H})]\to\smset$ by $\cat{H}(I,\frob_{\cat{H}}'(-))$. 

Although we will not give full details here, it is easy to show this
construction extends to a functor $\hyp \longrightarrow \calg$.
\end{remark}

To conclude this subsection, we use the functor $\hcca$ to construct the initial
cospan-algebra over $\Lambda$; this will be useful in what follows. Recall from
\cref{cor.factor_adj} that $\cospan$ is the free hypergraph category over
$\Lambda$. 

\begin{remark}\label{rem.partII}
Applying the functor $\hcca$ from \cref{prop.hyptocalg} to the hypergraph category $\cospan$ yields $\prt_\Lambda$ from \cref{ex.part}:
\[\hcca[\cospan]=\prt_\Lambda.\]
Indeed, \cref{prop.cospan_frob_id} says that the map $\frob_{\cospan}$ used
in \cref{eqn.Psi_on_obj} to define $\hcca[\cospan]$ is in fact the identity.
Thus $\hcca[\cospan](X)=\cospan(\emptylist,X)$, which is exactly $\prt_\Lambda(X)$. 
\end{remark}

We often denote $\prt_\Lambda$ simply by $\prt$ if $\Lambda$ is clear from context.

\begin{proposition}\label[proposition]{prop.ica_initial_cospan_algebra}
$\prt\colon\cospan\to\smset$ is the initial
cospan-algebra over $\Lambda$. 
\end{proposition}
\begin{proof}
Let $(A,\gamma)\colon\cospan\to\smset$ be a cospan-algebra over $\Lambda$. We need to show
there is a unique monoidal natural transformation $\alpha\colon\prt\to A$. Given
$X \in \cospan$, define
\begin{align}
\alpha_X\colon \prt(X)=\cospan(\emptylist,X) &\longrightarrow A(X); \nonumber \\
f\enspace &\longmapsto \big(1\To{\gamma}A\emptylist \To{Af}
AX\big). \label{eqn.def_unique_nattrans}
\end{align}
This is natural because $A$ is a functor: $A(g)(A(f)(\gamma))=A(f\cp
g)(\gamma)$. To prove that $\alpha$ is monoidal, we must show that for any
$f\colon \emptylist\to X$ and $g\colon \emptylist\to Y$ in $\cospan$:
\[
\begin{tikzcd}
	\ul{1}\ar[r, "\gamma"]\ar[d, "\gamma"']&
	A(\emptylist)\ar[r, "\Delta"]&
	A(\emptylist)\times A(\emptylist)\ar[r, "A(f)\times A(g)"]\ar[d, "\gamma_{\emptylist,\emptylist}"]&[20pt]
	A(X)\times A(Y)\ar[d, "\gamma_{X,Y}"]\\
	A(\emptylist)\ar[rr, "A(\lambda^{-1})"']&&
	A(\emptylist\otimes \emptylist)\ar[r, "A(f\otimes g)"']&
	A(X\otimes Y)
\end{tikzcd}
\]
where $\Delta$ is the diagonal map. This follows from the monoidality of $A$ and
the fact that $\Delta=\lambda^{-1}\colon \ul{1} \to \ul{1}\times \ul{1}$.

Finally, we must show that the definition of $\alpha$ in \cref{eqn.def_unique_nattrans} is the only possible one. To see this, first note that by \cref{eqn.laxator_cospan_alg}, the laxator
$\gamma^{\prt}\colon\ul{1}\to\prt(\varnothing)$ sends $1\mapsto\id_\varnothing$. Then since $\alpha$ is assumed to be a monoidal natural transformation, the following diagram commutes for any $X\in\cospan$:
\[
\begin{tikzcd}[column sep=large]
  \ul{1} \ar[drr, bend left=10pt, "f"] 
  \ar[dr, pos=.7, "\gamma^{\prt}"'] 
  \ar[ddr, bend right=25pt, "\gamma"'] \\
  [-15pt] 
  & \prt(\emptylist)\ar[r, "\prt(f)"']\ar[d, "\alpha_\emptylist"] 
  & \prt(X)\ar[d, "\alpha_X"] \\
  & A(\emptylist)\ar[r, "A(f)"'] 
  & A(X)
\end{tikzcd}
\]
and this forces $\alpha_X(f)=\gamma\cp A(f)$ as in \cref{eqn.def_unique_nattrans}.
\end{proof}

\section{From cospan-algebras to hypergraph categories}\label{ssec.cahc}
Our aim in this subsection is to provide the other half of the equivalence
\eqref{eqn.equiv_hyp_calg_Lambda}, converting any cospan-algebra $A$ into a
hypergraph category $\cahc[A]$. The aim of this subsection is detail the
following construction.

\begin{proposition} \label[proposition]{prop.calgtohyp}
For any $\Lambda\in\smset_\List$, we can naturally construct a functor
\[
\cahc\colon \lax(\cospan,\smset) \to \hyp_{\of(\Lambda)}.
\]
\end{proposition}

This will be proved on page \pageref{proof.calgtohyp}. As in the previous subsection, we first prove two lemmas that we will use to
define this functor on objects (\cref{lem.calgtohypobjects}) and then on
morphisms (\cref{lem.calgtohypmorphisms}). We will then again conclude the
subsection with some observations on what this implies about the interaction
between cospans and composition in hypergraph categories.

Given a cospan-algebra $(A,\gamma)$, we construct a hypergraph category
$\cahc[A]$ with objects and Frobenius structure coming from $\cospan$, and the
homsets coming from the image of objects under $A$.

\begin{lemma} \label[lemma]{lem.calgtohypobjects}
Let $A\colon\cospan\to\smset$ be a lax monoidal functor. We may define a
strict hypergraph category $\cahc[A]\in\hyp_{\of(\Lambda)}$ with:
\begin{itemize}
\item objects given by lists in $\Lambda$,
\item morphisms $X \to Y$ given by the set $A(X\concat Y)$,
\item monoidal structure arising from the monoidal structure on $A$, and
\item composition, identity, and hypergraph structure arising from the images of a certain cospans under $A$.
\end{itemize}
\end{lemma}
\begin{proof}
We first detail the structure of $\cahc[A]$, outlined above. For this we will need to give explicit names to the laxator maps, say $\gamma\colon\{1\}\to A(\emptylist)$ and $\gamma_{X,Y}\colon A(X)\times A(Y)\to A(X\concat Y)$ for any $X,Y\in\cospan$.

We define $\ob(\cahc[A])\coloneqq\List(\Lambda)$ and for any lists $X,Y$ in $\Lambda$, we define the homset
\begin{equation}\label{eqn.homset_HA}
	\cahc[A](X,Y) \coloneqq A(X\concat Y).
\end{equation}
The monoidal unit is the empty list $\emptylist$, the monoidal product on objects is given by concatenation of lists. The monoidal product on
morphisms is given by the lax structure on $A$:
\[
A(W\concat X) \times A(Y\concat Z) \To{\gamma_{W\concat X,Y \concat Z}} 
A(W \concat X \concat Y \concat Z) \To{A(\id_W \concat \sigma_{X,Y}\concat \id_Z)}
A(W \concat Y \concat X \concat Z).
\]
This is strict, so we need not define unitors and associators. 

The structure maps in $\cahc[A]$---the composition, identity, braiding, and Frobenius maps---are
constructed using the image under $A$ of particular cospans, as we now explain. 

Let $\Lambda$ be a set, and $X,Y,Z\in\Lambda$, and recall the morphism $\comp^Y_{X,Z}$ from \cref{ex.comp}. We define composition $\cahc[A](X,Y)\times\cahc[A](Y,Z)\to\cahc[A](X,Z)$ by
the formula
\begin{equation}\label{eq.composition}
A(X\concat Y) \times A(Y \concat Z) \To{\gamma_{X\concat Y,Y \concat Z}} A(X \concat Y
\concat Y \concat Z) \To{A(\comp^Y_{X,Z})} A(X \concat Z).
\end{equation}

All of the remaining structure maps in $\cahc[A]$ arise in similar ways, from cospans of the form $s\colon \emptylist \to X\concat Y$, where
$X$ and $Y$ are the domain and codomain of the map being constructed. Indeed,
given such an $s$, the composite $\{1\}\To{\gamma} A(\emptylist) \To{A(s)}
A(X\concat Y)$ gives an element of $A(X \concat Y)=\cahc[A](X,Y)$. 
The six required cospans $s$ are given as follows:%
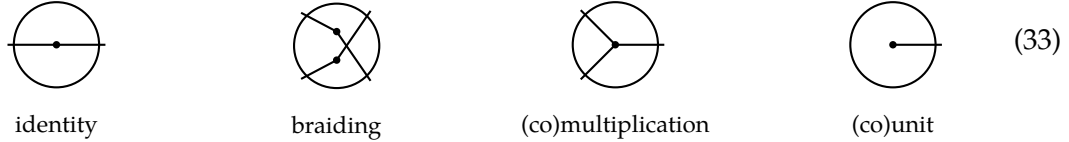
\footnote{As in \cref{eqn.id_and_hyp_struc}, the wiring diagram for a cospan of the form $\emptylist\to X\concat Y$ will have no interior cells, and the ports of the exterior cell are partitioned into two disjoint sets.}
\begin{equation}\label{eqn.id_and_hyp_struc}
\begin{tikzpicture}[unoriented WD, pack size=25pt, pack inside color=white, pack outside color=black, link size=2pt, font=\footnotesize, spacing=30pt, baseline=(un.south)]
	\node[link] at (0,0) (i) {};
	\node[outer pack, inner xsep=10pt, inner ysep=10pt, fit=(i)] (outi) {};
	\draw (outi.west) -- (i);
	\draw (outi.east) -- (i);
	\node[below=.2 of outi] (id) {identity};
	\node[link,above right=4pt and 3.5 of i] at (0,0) (u) {};
	\node[link,below =8pt of u] (d) {};
	\node[outer pack, inner xsep=10pt, inner ysep=4.5pt, fit=(u) (d)] (outb) {};
	\draw (outb.north west) -- (u);
	\draw (outb.south west) -- (d);
	\draw (outb.north east) -- (d);
	\draw (outb.south east) -- (u);
	\node[below=.2 of outb] (br) {braiding};
	\node[link, right=7 of i] at (0,0) (m) {};
	\node[outer pack, inner xsep=10pt, inner ysep=10pt, fit=(m)] (outm) {};
	\draw (outm.north west) -- (m);
	\draw (outm.south west) -- (m);
	\draw (outm.east) -- (m);
	\node[below=.2 of outm] (mult) {(co)multiplication};
	\node[link, right=10.5 of i] at (0,0) (un) {};
	\node[outer pack, inner xsep=10pt, inner ysep=10pt, fit=(un)] (outu) {};
	\draw (outu.east) -- (un);
	\node[below=.2 of outu] (unit) {(co)unit};
\end{tikzpicture}
\end{equation}

It remains to check the above data obeys the hypergraph category axioms, but this follows
from routine calculation. In particular, the associativity, identity, symmetric,
and hypergraph laws reduce to facts about the composition operation in $\cospan$ (this is easy to prove; see \cref{ex.cospan_scFrob} for intuition) although
we must also use the naturality of $\gamma$ and the functoriality of $A$ to make this reduction. The interchange law additionally uses the fact that $A$ is a \emph{symmetric} monoidal functor.
\end{proof}

A morphism $A \to B$ of lax monoidal functors $\cospan\to\smset$ consists of a collection of
functions $\alpha_X\colon A(X) \to B(X)$, one for each $X\in\List(\Lambda)$. Thus by \cref{eqn.homset_HA}, $\alpha_{X\concat Y}$ provides a map $\cahc[A](X,Y) \to
\cahc[B](X,Y)$ for all $X,Y\in\ob(\cahc[A])=\ob(\cahc[B])$. This is exactly the
data of an $\io$ hypergraph functor $F_\alpha\colon\cahc[A] \to
\cahc[B]$, though it remains to check that $F_\alpha$ is well-defined; we do that next.

\begin{lemma} \label[lemma]{lem.calgtohypmorphisms}
Let $\alpha\colon A \to B$ be a morphism of cospan-algebras.
Setting $F_\alpha(f)\coloneqq \alpha(f)$ for each $f \in \cahc[A](X,Y)$ 
defines an $\io$ hypergraph functor $F_\alpha\colon \cahc[A] \to \cahc[B]$.
\end{lemma}
\begin{proof}
The relevant axioms are routine consequences of the fact that $\alpha$ is a
monoidal natural transformation. For example, $F_\alpha$ preserves composition
when the following diagram commutes: 
\[
\begin{tikzcd}
A(X \concat Y) \times A(Y \concat Z) \ar[r,"\gamma_A"] \ar[d, "\alpha\times \alpha"] & A(X
\concat Y \concat Y \concat Z) \ar[r, "Ac"] \ar[d, "\alpha"] & A(X \concat Z)
\ar[d,"\alpha"] \\
B(X \concat Y) \times B(Y \concat Z) \ar[r, "\gamma_B"] & B(X\concat Y \concat Y \concat Z)
\ar[r, "Bc"] & B(X \concat Z)
\end{tikzcd}
\]
where $c$ is the cospan defining composition (see \eqref{eq.composition}).  This
is always true: the first square is the monoidality of $\alpha$, and the second
is the naturality of $\alpha$ with respect to $c$. Similarly, $F_\alpha$
preserves identities and all hypergraph structure as by the unit monoidality law
for $\alpha$ and by the naturality of $\alpha$ with respect to the cospans that
define the identity, braiding and Frobenius maps.

Note that also that $F_\alpha$ strict monoidal: $\cahc[A]$ and $\cahc[B]$ have
the same objects and monoidal product on objects, and $F_\alpha$ is identity on
objects.
\end{proof}

\begin{proof}[Proof of \cref{prop.calgtohyp}.]\label{proof.calgtohyp}
Choose any set $\Lambda$. For any lax monoidal functor
$A\colon\cospan\to\smset$,  \cref{lem.calgtohypobjects} produces a hypergraph
category $\cahc[A]\in\hyp_{\of(\Lambda)}$, and for any morphism $\alpha\colon A\to B$, \cref{lem.calgtohypmorphisms} produces a hypergraph functor between them. Functoriality is
straightforward: given composable cospan-algebra morphisms $\alpha$ and $\beta$,
$(F_\alpha\cp F_\beta)(-) = \beta(\alpha(-)) = F_{\alpha\cp\beta}(-)$.

Now suppose that $f\colon\Lambda\to\Lambda'$ is a morphism in $\smset_\List$. For naturality, we need to check the commutativity of a diagram much like \cref{eqn.naturality1}, which comes down to checking that $\cat{H}_{(\cospan[f]\cp A')}=\hyp_{\of(f)}\cat{H}_{A'}$ holds for any $A'\colon\cospan[\Lambda']\to\smset$. While an equality of categories may seem strange, these two categories are defined to have the same objects---namely both are $\ob(\cospan)$---as well as the same morphisms. Indeed, by the definition of $\hyp_\of$ in \cref{eqn.def_hyperfact,eqn.factor_functor2}, the definition of $\cahc$ in \cref{eqn.homset_HA}, and the fact that $\frob$ is $\io$, we have
\[
  \cat{H}_{(\cospan[f]\cp A')}(X,Y)\coloneqq
  A'(\cospan[f](X)\oplus\cospan[f](Y))=:
  \hyp_{\of(f)}\cat{H}_{A'}(X,Y)
\]
for any $X,Y\in\cospan$. This completes the proof.
\end{proof}

\begin{remark}
The main difference in perspective between hypergraph categories and
cospan-algebras is that the structure of hypergraph categories involves both
operations and special morphisms, whereas the structure of cospan-algebras
involves just operations. Indeed, a hypergraph category
$\cat{H}\in\hyp_{\of(\Lambda)}$ has the 2-ary operations of composition and
monoidal product, as well as the identity morphism $\id_X$ and four Frobenius
morphisms $\mu_X,\eta_X,\delta_X,\epsilon_X$ for every $X\in\cat{H}$. We saw in
\cref{eq.composition,eqn.id_and_hyp_struc} that both the operations and the
special morphisms can be encoded in various cospans---morphisms in
$\cospan$---and that a cospan-algebra $A$ turns them all into operations.
\end{remark}

We can now put several different ideas together. Recall the initial
cospan-algebra $\prt\colon\cospan\to\smset$ from
\cref{prop.ica_initial_cospan_algebra} and the name bijection
$\gathr{\cdot}\colon\cat{H}(X,Y)\to\cat{H}(\emptylist,X\concat Y)$ from
\cref{prop.gather_and_parse}. The above construction
(\cref{lem.calgtohypobjects}) constructs from the initial cospan-algebra over
$\Lambda$ a hypergraph category $\cat{H}_\prt$ over $\Lambda$, which comes
equipped with the universal map $\frob\colon \cospan \to \cat{H}_\prt$ that
selects its Frobenius morphisms. The following proposition tells us these
Frobenius morphisms are simply the names of the corresponding cospans.

\begin{lemma}\label[lemma]{lemma.frob_name}
Let $\prt$ be the initial cospan-algebra, let $\cat{H}_\prt$ be the corresponding
hypergraph category, and let $\frob\colon \cospan \to \cat{H}_\prt$ be the universal
map. Then for any $f\colon X\to Y$ in $\cospan$, we have $\frob(f)=\gathr{f}$.
\end{lemma}
\begin{proof}
Note that by definition we have $\cat{H}_\prt(X,Y)=\prt(X\ast Y)
=\cospan(\emptylist,X\concat Y)$, so the above identity is well typed. Since
$\cospan$ is the coproduct of $\Lambda$-many copies of $\cospan[]$, and since
$\cospan[]$ is generated by the Frobenius generators
(\cref{lem.cospan_generators}) it is enough to check that these two maps $\frob$
and $\gathr{\cdot}$ agree on the Frobenius generators. This is true by
construction as detailed in \cref{eqn.id_and_hyp_struc} of
\cref{lem.calgtohypobjects}.
\end{proof}

\begin{corollary} \label[corollary]{cor.frobenius_from_names_of_cospans}
Let $A\colon\cospan\to\smset$ be a cospan-algebra, let $c\colon X\to Y$ be a
cospan, and consider the counit map $\frob_A\colon\cospan\to\cahc[A]$ from \cref{cor.factor_adj}. Then
\[\frob_A(c)=A(\gathr{c})(\gamma).\]
\end{corollary}
\begin{proof}
Let $\prt\colon\cospan\to\smset$ be the initial cospan-algebra, let $\alpha_!\colon\prt\to A$ be the unique map, and let $F_!\colon\cat{H}_\prt\to\cahc[A]$ be the associated $\io$ hypergraph functor given by \cref{lem.calgtohypmorphisms}. By naturality of the counit $\frob$, we have $\frob_A=\frob_\prt\cp F_!$.

Take any $c\colon X\to Y$ in $\cospan$. By \cref{lemma.frob_name}, $\frob_\prt(c)=\gathr{c}$, and by definition of $F_!$ we have $\frob_A(c)=F_!(\frob_\prt(c))=\alpha_!(\gathr{c})$. The result then follows from \cref{prop.ica_initial_cospan_algebra}, specifically \cref{eqn.def_unique_nattrans}, which says $\alpha_!(\gathr{c})=A(\gathr{c})(\gamma).$
\end{proof}

\section{The equivalence between $\of$-hypergraph categories and cospan-algebras }\label{ssec.equivalences}

We are now ready to prove the equivalence.

\begin{theorem} \label[theorem]{thm.equivoverlambda}
The functors from \cref{prop.hyptocalg,prop.calgtohyp} define an equivalence of categories:
\[
\begin{tikzcd}
\hyp_{\of(\Lambda)} 
\ar[r, shift left=1, "\hcca"] &
\lax(\cospan,\smset).
\ar[l, shift left=1, "\cahc"]
\end{tikzcd}
\]
Moreover, this equivalence is natural in $\Lambda\in\smset_\List$.
\end{theorem}
%
%
\begin{proof}
Choose a set $\Lambda$. We shall show that $\cahc[\hcca]$ is naturally isomorphic to the identity
functor on $\hyp_{\of(\Lambda)}$, and that $\hcca[\cahc]$ is in fact equal to
the identity functor on $\lax(\cospan,\smset)$. We will then be done because both $\cahc$ and $\hcca$ were shown to be natural in $\Lambda\in\smset_\List$.

Let's first consider the case of $\cahc[\hcca]$. The hypergraph
category $\cahc[A_\cat{H}]$ has the same objects as $\cat{H}$ and has homsets
$\cahc[A_\cat{H}](X,Y) = \cat{H}(I,X\concat Y)$. Define an $\io$
hypergraph functor $\nu_\cat{H}\colon \cat{H} \to \cahc[A_\cat{H}]$ by sending $f
\in \cat{H}(X,Y)$ to its name $\nu_\cat{H}(f)\coloneqq\gathr{f}$, as in \cref{prop.gather_and_parse}. 

This map $\nu_\cat{H}$ is a well-defined functor due to the compact closure axioms.
Indeed, let $f\colon X \to Y$ and $g \colon Y \to Z$ be morphisms in $\cat{H}$.
By \cref{eq.composition}, the composite of names $\gathr{f} \cp \gathr{g}$ and the name of the composite are equal in
$\cahc[A_\cat{H}]$, by \cref{prop.comp_works}.

Preservation of the hypergraph structure
follows from the Frobenius axioms. For example, given the multiplication
$\mu_X$ on some object $X$ in $\cat{H}$, its image under $\nu$ is given by
\[
\begin{tikzpicture}[oriented WD, spider diagram, bbx=1em, bby=1ex, bb min width=1.5em]
	\node[bb={1}{1}] (mu) {$\mu$};
	\node[bb={0}{0}, draw=none, fit={(mu.south west) ($(mu.north east)+(0,3.5)$)}] (muhat) {};
	\coordinate (muhat_out2) at (mu_out1-|muhat.east);
	\coordinate (muhat_out1) at ($(muhat_out2)!.55!(muhat.north east)$);
	\begin{scope}[shorten >=-2pt]
  	\draw (mu_out1) to (muhat_out2) node[right, font=\scriptsize] {$X\concat X$} ;
  	\draw let \p1=(mu_in1|-muhat_out1) in
  		(mu_in1) to[in=180, out=180] (\x1,\y1) to (muhat_out1)
		node[right, font=\scriptsize] {$X$};
	\end{scope}
	\node[spider={1}{2}, right=8 of mu] (dot) {};
	\node[bb={0}{0}, color=white, fit={($(dot.south west)+(-1,-1)$)
	($(dot.north east)+(1,4.5)$)}] (dothat) {};
	\coordinate (dothat_out2) at (dothat.east|-dot_out1);	
	\coordinate (dothat_out3) at (dothat.east|-dot_out2);
	\coordinate (dothat_out1) at ($(dothat_out2)!.5!(dothat.north east)$);
	\begin{scope}[shorten >=-2pt]
  	\draw (dot_out1) to (dothat_out2);
  	\draw (dot_out2) to (dothat_out3);
  	\draw let \p1=(dot_in1|-dothat_out1) in
  		(dot_in1) to[in=180, out=180] (\x1,\y1) to (dothat_out1);	
	\end{scope}
	\node (eq1) at ($(muhat.east)!.7!(dothat.west)$) {=};
	\node[spider={0}{3}, right=8.5 of eq1] (spi) {};
	\node[bb={0}{0}, color=white, fit={(spi.west) ($(spi_out1)+(-1,0)$) ($(spi_out3)+(-.8,0)$)}] (spihat) {};
	\coordinate (spihat_out1) at (spihat.east|-spi_out1);	
	\coordinate (spihat_out2) at (spihat.east|-spi_out2);
	\coordinate (spihat_out3) at (spihat.east|-spi_out3);
	\begin{scope}[shorten >=-2pt]
  	\draw (spi_out1) to (spihat_out1);
  	\draw (spi_out2) to (spihat_out2);
  	\draw (spi_out3) to (spihat_out3);
	\end{scope}	
	\node (eq2) at ($(dothat.east)!.6!(spihat.west)$) {=};	
\end{tikzpicture}
\]
which is exactly the morphism specified by the cospan $\emptylist \to X\concat X
\concat X$ in $\cahc[A_\cat{H}]$; see \cref{eqn.id_and_hyp_struc}.

We thus see that $\nu_{\cat{H}}$ defines an identity-on-objects hypergraph
functor. Moreover, compact closure (\cref{prop.compactclosed}) implies $\nu$ is
fully faithful, and hence $\nu_{\cat{H}}$ has an inverse hypergraph functor. We must check these functors $\nu_\cat{H}$ are natural in $\cat{H}$, i.e.\ that for
any hypergraph functor $F\colon \cat{H} \to \cat{H}'$ the following naturality square commutes:
\[
\begin{tikzcd}
	\cat{H}\ar[r, "\nu_{\cat{H}}"]\ar[d, "F"']&
	\cahc[{\hcca[\cat{H}]}]\ar[d, "{\cahc[A_F]}"]\\
	\cat{H}'\ar[r, "\nu_{\cat{H}'}"']&
	\cahc[{\hcca[\cat{H}']}]
\end{tikzcd}
\]
This is so because all the maps in the square are $\io$ and for any morphism $f \in \cat{H}$ we have $\gathr{F(f)}=F(\gathr{f})$. Thus $\nu\colon\id_{\hyp_{\of(\Lambda)}}\to\cahc[\hcca]$ is a natural isomorphism, as desired.

Next we consider the case of $\hcca[\cahc]$; we want to show that for any cospan-algebra $A$, there is an equality $\hcca[{\cahc[A]}]=A$ of lax symmetric monoidal
functors $\cospan \to \smset$. On the objects of $\cospan$ this is
straightforward: the cospan-algebra $\hcca[{\cahc[A]}]$ maps an
object $X \in \cospan$ to $\cahc[A](\emptylist,X)= A(X)$, so by definition
$\hcca[{\cahc[A]}](X) = A(X)$. For morphisms, let $c\colon X \to Y$ be a cospan over
$\Lambda$. Then $\hcca[{\cahc[A]}](c)$ is the function $A(X) = \cahc[A](\emptylist,X) \to \cahc[A](\emptylist,Y)= A(Y)$ given by composition with $\frob(c)\in\cahc[A](X,Y)$,
where  $\frob\colon\cospan\to\cahc[A]$ is the functor from \cref{cor.factor_adj}. By \cref{cor.frobenius_from_names_of_cospans}, however,
$\frob(c)= A(\hat{c})(\gamma)$. By definition of composition in $\cahc[A]$
\eqref{eq.composition} and \cref{prop.remember_the_name}, this implies that
$\hcca[{\cahc[A]}](c)$ is exactly the function $A(c)$. Thus
$\hcca[\cahc]$ is identity-on-objects.

Next, we consider the action of $\hcca[\cahc]$ on morphisms. Suppose that
$\alpha\colon A \to B$ is a morphism of cospan-algebras, i.e.\ a
monoidal natural transformation. We shall show that
$\hcca[{\cahc[\alpha]}] =\alpha$. Indeed, $\cahc[\alpha]\colon \cahc[A] \to
\cahc[B]$ maps each $f \in \cahc[A](\emptylist,X)$ to $\alpha_X(f)$, and hence 
$\hcca[{\cahc[\alpha]}](X)\colon A(X) \to B(X)$ maps each $f \in
A(X)=\cahc[A](\emptylist,X)$ to $\alpha(f) \in B(X) = \cahc[B](\emptylist,X)$. This is what we
wanted to show.

We have shown $\hcca[\cahc] = \id_{\lax(\cospan,\smset)}$, completing the proof.
\end{proof}

\begin{example}
Recall the hypergraph category $\linrel$ with the addition Frobenius structure given
in \cref{ex.linrel_scfs_add}, where for example the linear relation corresponding to $\mu$ is $\{(a,b,c)\mid a+b=c\}\ss\rr^3$. We shall construct $\linrel'\coloneqq\cahc[{\hcca[\linrel]}]$, and
observe that it is hypergraph equivalent to $\linrel$. However, one may notice that the definition of $\mu$ is not symmetric with respect to $a,b,c$, e.g.\ in contrast with what one might call the symmetric version, $\{(a,b,c)\mid a+b+c=0\}$. Since cospan-algebras have no notion of domain and codomain, we will see that the isomorphism $\nu\colon\linrel\to\linrel'$ must rectify the asymmetry with a minus-sign.

By \cref{eqn.laxator_cospan_alg}, the lax symmetric monoidal functor $\hcca[\linrel]\colon
\cospan[] \to \smset$ sends each natural number $n$ to $\linrel(\rr^0,\rr^n)$, which we identify with the set of linear
subspaces of $\rr^n$, and it sends each cospan $m \to n$ to the unique
corresponding linear relation $\rr^m \to \rr^n$ defined by the Frobenius maps.

At first blush, the homsets $\linrel'(m,n)$ appear to be the same as those of $\linrel$, but this is not quite so:
\[
	\linrel(m,n)=\{R\ss\rr^m\oplus\rr^n\}
	\qquad\text{whereas}\qquad
	\linrel'(m,n)=\{R\ss\rr^{m+n}\}.
\]
These are certainly isomorphic, but in more than one way. The particular isomorphism $\nu$ constructed in \cref{thm.equivoverlambda} is given by sending $R$ to its name $\nu(R)=\gathr{R}$, which itself arises \emph{via the Frobenius structures in $\linrel$}; see \cref{sec.sdccc}. Unwinding the definitions, we have
\begin{align*}
	\nu(R)&=
	\{(a,b)\in\rr^{m+n}\mid\exists(a'\in\rr^m)\ldotp a+a'=0\wedge(a',b)\in R\}\\&=
	\{(a,b)\in\rr^{m+n}\mid (-a,b)\in R\}.
\end{align*}
\end{example}

\begin{remark}
The isomorphism between the categories of cospan-algebras and hypergraph
categories is a special case of the fact that there exists an isomorphism
\[
	\hyp^\io_{\cat{H}/}\cong\Cat{Lax}(\cat{H},\smset).
\]
between the category $\lax(\cat{H},\smset)$ of lax symmetric monoidal functors
$\cat{H} \to \smset$, and the coslice category $\hyp_{\cat{H}/}^{\mathrm{io}}$
over $\cat{H}$ of hypergraph categories and identity-on-objects
hypergraph functors between them.  Above we have just taken $\cat{H}$ to
be the free hypergraph category $\cospan$ over $\Lambda$. Nonetheless, the proof
above generalizes to the case where $\cat{H}$ is any hypergraph category.
\end{remark}

We can now easily prove the main theorem.

\begin{theorem} \label[theorem]{thm.main}
We have an equivalence of categories
\[
\hyp_\of \cong \calg.
\]
\end{theorem}
\begin{proof}
\cref{thm.equivoverlambda} provides an equivalence $\hyp_{\of(\Lambda)}\cong\lax(\cospan,\smset)$, natural in $\Lambda\in\smset_\List$. Since the Grothendieck construction is functorial, the middle map below is also an equivalence
\[
  \hyp_\of \cong
  \int^{\Lambda\in\smset_\List}\hyp_{\of(\Lambda)}\cong
  \int^{\Lambda \in \smset_\List} \lax(\cospan,\smset)\cong
  \calg.
\]
and the first and third equivalence follow from \cref{cor.hypof_fibration,prop.calg_fibration}.
\end{proof}

\printbibliography

\end{document}